\documentclass[12pt]{amsart}

\usepackage{amsxtra,amssymb,amsmath,amscd,url,listings}
\usepackage{amsrefs}
\usepackage{enumitem}
\newlist{enumth}{enumerate}{1}
\setlist[enumth]{label=\emph{(\arabic*)}, ref=(\arabic*)}

\usepackage[utf8]{inputenc}
\usepackage{eucal}
\usepackage{fullpage}
\usepackage{scrtime}
\usepackage[colorlinks]{hyperref}

\usepackage{mathrsfs}
\usepackage[scr=rsfs,cal=boondox]{mathalfa}

\usepackage{datetime}
\usepackage[all]{xy} 
\usepackage{graphicx}
\usepackage{tikz-cd}
\tikzset{commutative diagrams/arrow style=math font}
\usepackage{xspace}

\usepackage{tensor}
\usepackage{braket}

\usepackage{todonotes}

\shortdate

\newcommand{\good}{non-solvable\xspace}

\newcommand{\goodp}{non-solvable\xspace}
\newcommand{\lgood}{\good}
\newcommand{\lgoodp}{\good}
\newcommand{\ssgood}{\good\xspace}

\DeclareMathOperator{\Aut}{Aut}

\newcommand{\Aff}{\mathrm{Aff}}

\newcommand{\Jor}{\mathrm{Jor}}

\DeclareMathSymbol{A}{\mathalpha}{operators}{`A}%
\DeclareMathSymbol{B}{\mathalpha}{operators}{`B}%
\DeclareMathSymbol{C}{\mathalpha}{operators}{`C}%
\DeclareMathSymbol{D}{\mathalpha}{operators}{`D}%
\DeclareMathSymbol{E}{\mathalpha}{operators}{`E}%
\DeclareMathSymbol{F}{\mathalpha}{operators}{`F}%
\DeclareMathSymbol{G}{\mathalpha}{operators}{`G}%
\DeclareMathSymbol{H}{\mathalpha}{operators}{`H}%
\DeclareMathSymbol{I}{\mathalpha}{operators}{`I}%
\DeclareMathSymbol{J}{\mathalpha}{operators}{`J}%
\DeclareMathSymbol{K}{\mathalpha}{operators}{`K}%
\DeclareMathSymbol{L}{\mathalpha}{operators}{`L}%
\DeclareMathSymbol{M}{\mathalpha}{operators}{`M}%
\DeclareMathSymbol{N}{\mathalpha}{operators}{`N}%
\DeclareMathSymbol{O}{\mathalpha}{operators}{`O}%
\DeclareMathSymbol{P}{\mathalpha}{operators}{`P}%
\DeclareMathSymbol{Q}{\mathalpha}{operators}{`Q}%
\DeclareMathSymbol{R}{\mathalpha}{operators}{`R}%
\DeclareMathSymbol{S}{\mathalpha}{operators}{`S}%
\DeclareMathSymbol{T}{\mathalpha}{operators}{`T}%
\DeclareMathSymbol{U}{\mathalpha}{operators}{`U}%
\DeclareMathSymbol{V}{\mathalpha}{operators}{`V}%
\DeclareMathSymbol{W}{\mathalpha}{operators}{`W}%
\DeclareMathSymbol{X}{\mathalpha}{operators}{`X}%
\DeclareMathSymbol{Y}{\mathalpha}{operators}{`Y}%
\DeclareMathSymbol{Z}{\mathalpha}{operators}{`Z}%

\renewcommand{\leq}{\leqslant}
\renewcommand{\geq}{\geqslant}
 
\numberwithin{equation}{section}

\newcommand{\uple}[1]{\text{\boldmath${#1}$}}

\def\stacksum#1#2{{\stackrel{{\scriptstyle #1}}
{{\scriptstyle #2}}}}

\def\setminus{\mathchoice
    {\mathbin{\vrule height .72ex width 1.61ex depth -.38ex}}% 12
    {\mathbin{\vrule height .72ex width 1.61ex depth -.38ex}}% 12
    {\mathbin{\vrule height .50ex width 0.85ex depth -.28ex}}%  9
    {\mathbin{\vrule height .20ex width 0.570ex depth -.24ex}}%  7
}

\newcommand{\Id}{\mathrm{Id}}

\newcommand{\bfalpha}{\uple{\alpha}}
\newcommand{\bfP}{\uple{P}}
\newcommand{\bfgamma}{\uple{\gamma}}
\newcommand{\bfchi}{\uple{\chi}}
\newcommand{\bfrho}{\uple{\rho}}
\newcommand{\bfxi}{\uple{\xi}}
\newcommand{\bfzeta}{\uple{\zeta}}

\newcommand{\bfbeta}{\uple{\beta}}

\newcommand{\bfv}{\uple{v}}
\newcommand{\bfr}{\uple{r}}
\newcommand{\bfs}{\uple{s}}
\newcommand{\bfu}{\uple{u}}

\newcommand{\bfx}{\uple{x}}
\newcommand{\bfy}{\uple{y}}

\newcommand{\bfS}{\mathbf{S}}

\newcommand{\bfK}{\mathbf{K}}

\newcommand{\Cc}{\mathbf{C}}

\newcommand{\Aa}{\mathbf{A}}

\newcommand{\Zz}{\mathbf{Z}}
\newcommand{\Pp}{\mathbf{P}}
\newcommand{\Rr}{\mathbf{R}}
\newcommand{\Gg}{\mathbf{G}}
\newcommand{\Gm}{{\mathbf{G}_{m}}}

\newcommand{\Qq}{\mathbf{Q}}

\newcommand{\Fq}{\Ff_q}
\newcommand{\Fqt}{\Ff^\times_q}

\newcommand{\ovFqt}{\ov{\Ff}^\times_q}
\newcommand{\kt}{{k^\times}}

\newcommand{\Ff}{\mathbf{F}}

\newcommand{\bFq}{\overline{\Ff}_q}
\newcommand{\bQl}{\overline{\Qq}_{\ell}}

\newcommand{\mcV}{\mathscr{V}}
\newcommand{\mcT}{\mathcal{T}}

\newcommand{\mcW}{\mathscr{W}}

\newcommand{\mcN}{\mathscr{N}}
\newcommand{\mcM}{\mathscr{M}}
\newcommand{\mcH}{\mathscr{H}}

\def\loccit{loc.\kern3pt cit.{}\xspace}
\def\cf{see\kern.3em}
\def\Cf{See\kern.3em}
\def\eg{e.g.\kern.3em}

\def\resp{\text{resp.}\kern.3em}

\newcommand{\mods}[1]{\,(\mathrm{mod}\,{#1})}

\newcommand{\what}{\widehat}

\newcommand{\hautb}{\mathbf{B}}

\DeclareMathOperator{\mfm}{\mathfrak{m}}

\newcommand{\ra}{\rightarrow}

\newcommand{\surjecte}{\twoheadrightarrow}

\DeclareMathOperator{\rank}{rank}
\DeclareMathOperator{\rk}{rk}

\DeclareMathOperator{\Kl}{\mathrm{Kl}}

\DeclareMathOperator{\Gal}{Gal}

\DeclareMathOperator{\Res}{Res}

\DeclareMathOperator{\Tr}{Tr}

\newcommand{\eps}{\varepsilon}
\renewcommand{\rho}{\varrho}
\renewcommand{\tilde}{\widetilde}

\newcommand{\SL}{\mathbf{SL}}
\newcommand{\GL}{\mathbf{GL}}
\newcommand{\PGL}{\mathbf{PGL}}

\newcommand{\Sp}{\mathbf{Sp}}

\DeclareMathOperator{\Hyp}{Hyp}
\DeclareMathOperator{\HYP}{\mathscr{H}}

\DeclareMathOperator{\KL}{\mathscr{K}\!\ell}
\DeclareMathOperator{\Lie}{Lie}

\newcommand{\demi}{{\textstyle{\frac{1}{2}}}}

\newcommand{\sheaf}[1]{\mathscr{{#1}}}

\DeclareMathSymbol{\gena}{\mathord}{letters}{"3C}
\DeclareMathSymbol{\genb}{\mathord}{letters}{"3E}

\theoremstyle{plain}
\newtheorem{theorem}{Theorem}[section]
\newtheorem*{theorem*}{Theorem}
\newtheorem{lemma}[theorem]{Lemma}
\newtheorem{corollary}[theorem]{Corollary}

\newtheorem{proposition}[theorem]{Proposition}

\theoremstyle{remark}

\theoremstyle{definition}

\newtheorem{definition}[theorem]{Definition}

\newtheorem{remark}{Remark}[section]
\newtheorem{example}[remark]{Example}

\newcommand{\abs}[1]{\left\lvert#1\right\rvert}

\newcommand{\mcL}{\mathscr{L}}

\newcommand{\mcD}{\mathscr{D}}
\newcommand{\mcF}{\mathscr{F}}
\newcommand{\mcK}{\mathscr{K}}

\newcommand{\mcG}{\mathscr{G}}

\newcommand{\mfn}{\mathfrak{n}}

\newcommand{\mfh}{\mathfrak{h}}

\newcommand{\ovQl}{{\ov{\Qq}_{\ell}}}

\newcommand{\vphi}{\varphi}

\renewcommand{\geq}{\geqslant}
\renewcommand{\leq}{\leqslant}

\newcommand{\ov}[1]{\overline{#1}}

\newcommand{\peter}[1]{\langle{#1}\rangle}

\newcommand\sumsum{\mathop{\sum\sum}\limits}
\newcommand\sumdsum{\mathop{\sum\cdots\sum}\limits}
\newcommand\sumsumsum{\mathop{\sum\sum\sum}\limits}

\newcommand{\Der}{{\mathrm{D}_c^{\mathrm{b}}}}

\begin{document}

\title{Bilinear forms with trace functions}
 
\author{\'Etienne Fouvry}
\address{Universit\'e Paris--Saclay,   CNRS \\
Laboratoire de Math\'ematiques d'Orsay\\
  91405 Orsay  \\France}
\email{etienne.fouvry@universite-paris-saclay.fr}

\author{Emmanuel Kowalski}
\address{ETH Z\"urich -- D-MATH\\
  R\"amistrasse 101\\
  CH-8092 Z\"urich\\
  Switzerland} \email{kowalski@math.ethz.ch}

\author{Philippe Michel}
\address{EPFL/SB/TAN, Station 8, CH-1015 Lausanne, Switzerland }
\email{philippe.michel@epfl.ch}

\author{Will Sawin}
\address{Princeton University, Department of Mathematics, Fine Hall, Washington Road, Princeton, NJ 08540, USA }
\email{wsawin@math.princeton.edu}

\subjclass[2010]{}

\begin{abstract}
  We obtain non-trivial bounds for bilinear sums of trace functions
  below the P\'olya-Vinogradov range assuming only that the geometric
  monodromy group of the underlying $\ell$-adic sheaf satisfies certain
  simple structural properties, in contrast to previous works which
  handled only special cases of  Kloosterman
  and hypergeometric sheaves.
  % studied by N. Katz in his PUP books satisfy
  % % these assumptions.

  Our approach builds on a general ``soft'' stratification theorem for
  sums of products of trace functions, based on an idea of Junyan Xu, combined
  with a new robust version of the Goursat--Kolchin--Ribet criterion.
\end{abstract}

\date{\today}

\maketitle 

\begin{flushright}
  \textit{Dedicated to Nick Katz\\
    ``Too old to break and too young to tame''}
\end{flushright}
\setcounter{tocdepth}{1}
\tableofcontents

\section{Introduction}

\subsection{Bilinear sums of trace functions}

A central part of analytic number theory deals with the issue of
estimating the ``correlation''
\[
  \sum_{m}f(m)\overline{g(m)}
\]
between arithmetic sequences of various types. In many important
instances, this difficult problem can be transformed into a slightly
more tractable question of bounding non-trivially some \emph{bilinear
  forms} of the type
\[
  \sum_{m,n}\alpha_m\beta_nk(m,n),
\]
where the ``kernel''~$k$ is supposed to be relatively well-understood,
whereas not much is assumed to be known of the sequences $(\alpha_m)$
and~$(\beta_n)$, except for their average size for instance.  We
recall that the point of the bilinear form is that, using only such
basic knowledge of the sequences $(\alpha_m)$ and~$(\beta_n)$, one can
obtain highly non-trivial results due to the oscillations of $k(m,n)$.

The prototypical example of this transformation is found in the
decompositions of the von Mangoldt function in terms of bilinear (or
multilinear) expressions, first discovered by Vinogradov. In this
case, as in a number of others, the kernel is of the form
$k(m,n)=K(mn)$ for a suitable function~$K$. Moreover, for non-zero
integers~$b$ and~$c$, the ``monomial'' kernels
$k_{b,c}(m,n)=K(m^bn^c)$ are also of particular interest, and in fact
appear naturally in our motivating application.

In these, as explained in Section~\ref{sec-general}, we consider the
further restriction that the function~$K$ is a \emph{trace function}
modulo a prime number~$q$, and we seek estimates when $m$, $n$ run
over intervals
\[
 1 \leq M\leq m<2M\leq q,\quad\quad
  1\leq N\leq n<2N\leq q,
\]
with the key goal (motivated by our applications, and many others) of
obtaining non-trivial bounds for~$M$ and~$N$ as small as possible
compared to~$q$.

This is the same context as in various previous papers (see~\cite{FI} by
Friedlander and Iwaniec, \cite{FMAnn} by Fouvry and Michel and
\cite{KMSAnn}, \cite{Pisa} by Kowalski, Michel and Sawin). The key
achievement of the present work is to obtain \emph{estimates as good as
  those of these papers, but in much greater generality}. More
precisely, whereas it was previously always assumed that~$K$ was of a
specific type (essentially some exponential sum of specific type, such
as hyper-Kloosterman sums), we are now able to succeed with
\emph{qualitative} conditions on~$K$, and these conditions are, on the
one hand, known to hold for a  wide variety of trace functions and,
on the other hand, extremely robust under a variety of natural transformations.

In order to state a simplified version of the main result, we recall
that a trace function~$K$ modulo~$q$ is associated to a certain
algebraic object. We assume that~$K$ is the trace function of a
constructible~$\ell$-adic sheaf~$\mcF$ on the affine line over the
finite field~$\Ff_q$, where $\ell$ is a prime $\not=q$, which is a
middle-extension sheaf pure of weight~$0$. We moreover identify $\ovQl$
with~$\Cc$, so we can view~$K$ as a complex-valued function. There are
intrinsic ``symmetry groups'' attached to~$\mcF$, called the geometric
and arithmetic monodromy groups, which have been explicitly computed by
Katz in many instances; we denote by~$G$ the geometric monodromy group,
which can be viewed as a closed subgroup of~$\GL_r(\Cc)$ for some
non-negative integer~$r$. Moreover, there is a numerical complexity
invariant~$c(\mcF)$ of~$\mcF$, which is a non-negative integer. A
simplified version of our main result is:

\begin{theorem}\label{thmType2simple}
  Assume that $G$ acts irreducibly on~$\Cc^r$ and that one of the
  following conditions holds:
  \begin{enumth}
  \item $G$ is infinite,
  \item or the group~$G$ is finite and \good.  \end{enumth}
  
  Let~$\delta>0$ be a real number and let~$1\leq M,N\leq q/2$ be
  integers such that
  \[
    q^{\delta}\leq M,N,\quad\quad MN\geq q^{3/4+\delta}.
  \]

  There exists~$\eta>0$, depending only on~$\delta$, such that for
  any sequences of complex numbers\footnote{\ We recall that the
    notation $n\sim N$ means that $N\leq n<2N$, and similarly for
    $m\sim M$, etc.}
  \[
    (\alpha_m)_{m\sim M},\quad (\beta_n)_{n\sim N}
  \]
  we have
  \[
    \sum_{m\sim M}\sum_{n\sim N}\alpha_m\beta_n
    K(mn)
    \ll \Bigl(\sum_{m}|\alpha_m|^2\Bigr)^{1/2}
    \Bigl(\sum_{n}|\beta_n|^2\Bigr)^{1/2}
    (MN)^{1/2-\eta},
  \]
  where the implicit constant depend only on  $\delta$ and on the complexity of $\mcF$.
\end{theorem}

\begin{remark}
  (1) If $G$ is infinite, a theorem of Deligne implies
  % the purity of $\mcF$ implies, by the work of
  % Deligne, 
  that the connected component of the identity $G^0$ is a non-trivial
  semisimple group (see \cite{WeilII}*{Cor.\,(1.3.9)}); in particular,
  since $G^0$ is perfect as an algebraic group, it follows that $G$ is
  \good. Thus, asking that one of conditions (1) and~(2) holds is
  equivalent to the simpler condition that~$G$ be non-solvable. We chose
  the present formulation to emphasize the non-obvious fact that no
  condition is required when~$G$ is infinite.

  (2) In a previous version of this work, we had proved the results of
  this paper under the assumption that $\mcF$ satisfied a property we
  called begin \emph{gallant}.  This stated that the action of~$G$
  on~$\Cc^r$ should be irreducible and moreover one of the following
  conditions was satisfied:
  \begin{enumerate}
  \item\label{goodinfinite} the identity component~$G^0$ of~$G$ is a
    simple algebraic group;
  \item\label{goodfinite} or the group~$G$ is finite and contains a
    quasisimple\footnote{\ Namely, a perfect subgroup~$H$ such that the
      quotient $H$ by its center is simple, necessarily non-abelian.}
    normal subgroup~$N$ acting irreducibly on~$\Cc^r$.
  \end{enumerate}

  It is straightforward that these conditions are satisfies if~$G$ is
  not solvable. Thus, the present paper supersededs entirely the
  previous preprint.
\end{remark}

\begin{example}
  In particular, Theorem~\ref{thmType2simple} recovers the main results
  of~\cite{Pisa}.
  % (since the complexity of the relevant sheaves are bounded in term
  % of~$r$ only).
\end{example}

We refer to Section~\ref{sec-general} below for the detailed statements
of our most general results, and to Section~\ref{sec-galant} for a wide
variety of concrete examples of trace functions where the conditions of
this result (or of the more general variant in
Definition~\ref{defgoodsheafintro}) are satisfied. In
Section~\ref{sec-toric}, we present our motivating application.

Although the strategy we follow remains similar to that of previous
works, there are three key new ingredients which are essential to
obtain a bound in the generality in which we work:
\begin{itemize}
\item the general Quantitative Sheaf Theory of Sawin (as expounded by
  Forey, Fres\'an, Kowalski and Sawin~\cite{qst}) allows us to perform
  complicated algebraic transformations on the sheaves we work with while keeping
   control of their complexity, whereas one previously had to do
  this in an ad-hoc way, which was essentially only possible given a
  ``formula'' for the trace function (see,
  e.g.,~\cite{KMSAnn}*{Prop.\,4.24} for an instance of this);

\item we establish new flexible and robust variants of the
  Goursat--Kolchin--Ribet framework of Katz~\cite[\S\,1.8]{ESDE}, which
  allow us in practice to obtain cancellation in ``sums of products''
  situations (in the sense of~\cite{sumproducts}) under qualitative
  conditions on the mono\-dromy (such as those in
  Theorem~\ref{thmType2simple}), and not simply for specific groups (see
  Section~\ref{sec:goursat}, and in particular Corollary~\ref{cor-sop},
  which is an estimate in the spirit of~\cite{sumproducts} with
  independent interest). One particular interesting point is that these
  new results allow us to treat instances of sheaves with \emph{finite}
  mo\-nodromy groups (see Section~\ref{sec-finite} for examples);

\item we make essential use of a beautiful idea of Xu~\cite{Xu}, which
  roughly speaking states that bounds for the moments of a trace
  function imply stratification results, from which one can deduce
  further analytic properties which turn out to be significantly
  stronger than those which could be obtained directly from the moment
  bounds (see Theorem~\ref{XuStep}). Note that our version
  of this idea relies also in an essential way on Quantitative Sheaf
  Theory.
\end{itemize}

\begin{remark}
  Our method is in one respect less efficient than that of~\cite{Pisa}:
  whereas the latter provides estimates where the dependency on the
  complexity (when it applies) is \emph{polynomial}, this is not the
  case in this paper, due to the use of qualitative results in
  Quantitative Sheaf Theory.
\end{remark}

\subsection{General statements}
\label{sec-general}

We will state precisely our most general results in this section.
This requires more terminology and background material in algebraic
geometry; Section~\ref{secladicgeneral} gives precise definitions and
references.

We first define formally the class of sheaves for which our results
apply. 
%We recall again that a finite group~$G$ is called
%\emph{quasisimple} if it is perfect and if the quotient of~$G$ by its
%center~$Z(G)$ is simple (in which case it is elementary that~$G/Z(G)$
%is non-abelian; see, e.g.,~\cite[Lemma\,9.1]{isaacs}).

\begin{definition}[Non-solvable sheaves]\label{defgoodsheafintro}
  Let $k$ be a finite field and~$\ell$ a prime number invertible
  in~$k$. Let~$\mcF$ be an $\ell$-adic sheaf on~$\Aa^1_{k}$ with generic
  rank~$r\geq 1$, which is mixed of \emph{integral} weights~$\leq 0$ and
  whose restriction to some open dense subset is pure of weight~$0$. We
  say that~$\mcF$ is \emph{\good} if the geometric monodromy group
  of~$\mcF$ acts irreducibly on the generic fiber and is \good as an
  algebraic group.
  % We say that~$\mcF$ is \emph{\lgoodp} if~$\mcF$ is \good and moreover
  % is mixed of \emph{integral} weights~$\leq 0$ and its restriction to
  % some open dense subset is pure of weight~$0$.
\end{definition}

% \begin{remark}\label{rm-an-good}
%   In the case of a finite
%   group~$G$, the definition of a \good group is equivalent to the
%   following: there exists a perfect normal subgroup~$N$
%   of~$G$ and a short exact sequence
%   \[
%     1\to A\to N\to S\to 1
%   \]
%   with~$A$ central in~$G$
%   and~$S$ a non-abelian simple group (in other words, we need not
%   insist a priori that~$A$ coincide the center of~$N$).

%   Indeed, this condition holds if~$N$ is quasisimple,
%   with~$A=Z(N)$. Conversely, if we have such an exact sequence, then
%   the image of the center~$Z(N)$ in~$S$ is a normal commutative
%   subgroup, hence is trivial, so that~$Z(N)\subset A$. Since the
%   converse inclusion holds by assumption, we have $A=Z(N)$.
% \end{remark}

%We note that a \good sheaf is geometrically irreducible by
%definition. Moreover, again by definition, the sheaves appearing in
%Theorem~\ref{thmType2simple} are \goodp.

In addition to general bilinear forms (classically called ``type II''
sums), we will consider special bilinear forms (called ``type I'') of
the form
\[
  \sum_{m}\sum_n\alpha_m K(mn^c),
\]
for which one can usually prove slightly stronger results with slightly
easier proofs.

Our main estimates are stated in the following theorem:

\begin{theorem}[Type I and II estimates]\label{thmType12}
  Let~$b$, $c$ be non-zero integers.  Let $q$ be a prime and let~$\mcF$
  be a non-solvable $\ell$-adic middle-extension sheaf
  over~$\Aa^1_{\Ff_q}$.

  Let~$K$ be the trace function of~$\mcF$.  Let $l\geq 2$ be an integer
  and let $M$, $N\geq 1$ be integers and $\mcM\subset\Fqt$ a subset of
  size $|\mcM|=M$.

  \begin{enumth}
  \item Suppose that
    \[
      % \label{assumptypeIbasic}
      10q^{1/l}\leq N\leq q^{1/2+1/(2l)}.
    \]
    For any $\eps>0$ and for any tuple $\bfalpha=(\alpha_m)_{m\in \mcM}$
    of complex numbers, the estimate
    \begin{equation}
      \label{BtypeI}
      \sum_{m\in \mcM}\sum_{n\sim N}\alpha_mK(mn^c)
      \ll q^{\eps}\|\bfalpha\|_2M^{1/2}N
      \Bigl(\frac{q^{1+3/(2l)}}{MN^2}\Bigr)^{1/(2l)}
    \end{equation}
    holds, where the implicit constant depends at most on $\eps$,
    $c$, $l$ and on the complexity of $\mcF$.
    
\item Suppose that 
  \begin{equation}
    \label{assumptypeIIbasic}
    10q^{3/(2l)}\leq N\leq q^{1/2+3/(4l)}.
  \end{equation}
  For any $\eps>0$ and for any families
  $\bfalpha=(\alpha_m)_{m\in \mcM}$ and $\bfbeta=(\beta_n)_{n\sim N}$ of
  complex numbers, the estimate
  \begin{equation}
    \label{BtypeII}
    \sum_{m\in\mcM}\sum_{n\sim N}\alpha_m\beta_nK(mn^c)
    \ll q^{\eps}\|\bfalpha\|_2\|\bfbeta\|_2(MN)^{1/2}
    \Bigl(\frac{1}{M}
    +\Bigl(\frac{q^{\tfrac{3}{4}+\tfrac{7}{4l}}}{MN}\Bigr)^{1/l}\Bigr)^{1/2}
  \end{equation}
  holds, where the implicit constant depends at most on $\eps$, 
  $c$, $l$ and on the complexity of $\mcF$.
\end{enumth}
\end{theorem}

\begin{remark}
  (1) The ``trivial'' bound for (general) type I sums is of the form
  \[
    \Bigl|\sum_{m}\sum_n\alpha_m k(m,n)\Bigr|
    \leq \|\bfalpha\|_2\, M^{1/2}N\,
    \|k\|_{\infty}.
  \]
  
  Taking $l$ large enough, we see that the bound \eqref{BtypeI} is
  non-trivial, in the sense of providing an estimate of the form
  \[
    \sum_{m}\sum_{n}\alpha_mK(mn^c)
    \ll\|\bfalpha\|_2(M^{1/2}N)^{1-\eta}
  \]
  for some $\eta>0$, as long as
  there exists $\delta>0$ such that $MN^2\geq q^{1+\delta}$, in which
  case $\eta$ depends on~$\delta$. In
  particular, we obtain a non-trivial estimate in the case
  \[
    M=N=q^{\frac{1}{3}(1+\delta)}
  \]
  for any~$\delta>0$.

  (2) For type II sums, the trivial bound is of the form
  \[
    \Bigl|\sum_{m}\sum_n\alpha_m \beta_n k(m,n)\Bigr| \leq
    \|\bfalpha\|_2\|\bfbeta\|_2\, (MN)^{1/2}\, \|k\|_{\infty}.
  \]

  Here, taking $l$ large enough, the bound \eqref{BtypeII} gives an
  estimate of the form
  \[
    \sum_{m}\sum_n\alpha_m\beta_nK(mn^c)\ll
    \|\alpha\|_2\|\beta\|_2(MN)^{1/2-\eta}
  \]
  for some~$\eta>0$ as soon as
  \[
    q^{\delta} \leq M,N,\ MN \geq q^{3/4+\delta}
  \]
  for some $\delta>0$, with~$\eta$ depending on~$\delta$. It follows
  that Theorem~\ref{thmType12} implies
  Theorem~\ref{thmType2simple}, and also that we have a non-trivial
  estimate in the special case
  \[
    M=N=q^{\frac{3}{8}(1+\delta)}
  \]
  for~$\delta>0$.

  (3) Our main target in this paper is to have non-trivial
  bounds in the largest possible range of uniformity of~$M$ and~$N$
  relative to $q$. There is no doubt that a more sophisticated
  treatment could be used to improve the actual saving compared to the
  trivial bound, in \emph{some} ranges of~$M$ and~$N$.
\end{remark}

\begin{remark}
  The possibility of considering kernels of general monomial types,
  $K(m^bn^c)$ first appeared in the work of Nunes~\cite{nunesarx}, who
  studied the distribution of squarefree numbers in large arithmetic
  progressions (here the relevant monomial was $mn^2$); we refer also
  to~\cite{nunes} for a stronger result of Nunes using a result of
  Pierce.
\end{remark}

The extension from $m$ varying over an interval $[M,2M]$ to $m$ varying
over a general subset $\mcM\subset \Fqt$ was suggested by Shparlinski.
It has the following potentially useful corollary, where we extend a
trace function~$K$ modulo~$q$ to the projective line over~$\Ff_q$ by
mapping the point at infinite to~$0$.\footnote{\ In some cases, this is
  not the most natural extension, but this has no impact on the result.}
We use the usual action of $\PGL_2(\Ff_q)$ on the projective line by
fractional linear transformations.

\begin{corollary}
  With notation as above, let~$\gamma\in\PGL_2(\Ff_q)$ and let~$b$ be a
  non-zero integer.  Define
  \begin{gather*}
    B_{b,c,\gamma}(\bfalpha,N)=\sumsum_{m\sim M,\ n\sim N}\alpha_m
    K((\gamma\cdot m)^bn^c),
    \\
    B_{b,c,\gamma}(\bfalpha,\bfbeta)=\sumsum_{m\sim M,\ n\sim
      N}\alpha_m\beta_n K((\gamma\cdot m)^bn^c).
  \end{gather*}
  
  The sums $B_{b,c,\gamma}(\bfalpha,N)$ satisfy the estimates of
  Theorem~\ref{thmType12}, (1), and the sums
  $B_{b,c,\gamma}(\bfalpha,\bfbeta)$ satisfy the estimates of
  Theorem~\ref{thmType12}, (2), with implied constants depending also
  on~$b$.
\end{corollary}

\begin{proof}
% It allows, for instance, to bound bilinear sums of the shape
%   \begin{gather*}
%     $$B_{b,c,\gamma}(\bfalpha,N)=\sumsum_{m\sim M,\ n\sim N}\alpha_m K((\gamma.m)^bn^c)$$
%     or
%     $$B_{b,c,\gamma}(\bfalpha,\bfbeta)=\sumsum_{m\sim M,\ n\sim
%       N}\alpha_m\beta_n K((\gamma.m)^bn^c)$$
%   \end{gather*}
%   where $b\in\Zz-\{0\}$, $\gamma=\begin{pmatrix} a'&b'\\c'&d'
%   \end{pmatrix}\in\PGL_2(\Fq)$ is acting on $\Pp^1(\Fq)$ by fractional
%   linear transformations: $\gamma.z=\frac{a'z+b'}{c'z+d'}$ and
%   $\bfalpha=(\alpha_m)_{m\sim M},\ \bfbeta=(\beta_n)_{n\sim N}$ are
%   complex numbers indexed by the integers of the intervals
%   $[M,2M),[N,2N)$.
  Indeed, if we define the set
  $$\mcM_\gamma^b=\{(\gamma\cdot m)^b,\ m\sim M\},$$
  and for $m\in\mcM^b_\gamma$ we define
  $$
  \alpha'_{m}=\sum_{\substack{\tilde{m}\sim M\\
      (\gamma\cdot \tilde{m})^b=m}}\alpha_{\tilde{m}},
  $$
  then we have 
  \[
    B_{b,c,\gamma}(\bfalpha,\bfbeta)=
    \sum_{m\in\mcM}\sum_{n\sim N}\alpha'_m\beta_nK(mn^c),.
  \]

  Since
  $$
  \|\bfalpha'\|^2_2\leq \sumsum_\stacksum{m_1,m_2\sim M}{(\gamma\cdot
    m_1)^b=(\gamma\cdot m_2)^b}|\alpha_{m_1}||\alpha_{m_2}|\leq
  \sumsum_\stacksum{m_1,m_2\sim M}{(\gamma\cdot m_1)^b=(\gamma\cdot
    m_2)^b}|\alpha_{m_1}|^2\leq |b|\|\bfalpha\|^2_2\Bigl(1+\frac{M}q\Bigr).
  $$
  
  Hence for $M\leq q$ and under the same conditions on $N$, the bounds
  \eqref{BtypeI} and \eqref{BtypeII} hold for
  $B_{b,c,\gamma}(\bfalpha,N)$ and $B_{b,c,\gamma}(\bfalpha,\bfbeta)$
  with the implicit constants also depending on $b$\footnote{ one could replace the map $m\mapsto (\gamma.m)^b$ by any non-constant morphism of $\Pp^1_{\Fq}$ of absolutely bounded degree; however the present example comes up naturally in applications}.
\end{proof}

Moreover, with some motivations in mind, we will prove a ``trilinear''
bound.
 
\begin{theorem}\label{thmtriplesum}
  Let~$a$, $b$, $c$ be non-zero integers.  Let $q$ be a prime and
  let~$\mcF$ be a \lgood $\ell$-adic sheaf over~$\Ff_q$ for some
  prime~$\ell\not=q$. Let~$K$ be the trace function of~$\mcF$.

  Let~$l\geq 2$ be an integer and let~$J$, $M$, $N\geq 1$ be integers.
  Suppose that
  \[
   % \label{LMandNbound}
    J\leq 4q,\quad\quad MN\leq 4q.
  \]

  For any~$\eps>0$ and for any families 
  \[
    \bfalpha=(\alpha_j)_{j\sim J},\quad
    \bfbeta=(\beta_m)_{m\sim M},\quad
    \bfgamma=(\gamma_n)_{n\sim N}
  \]
  of complex numbers of modulus~$\leq 1$, the estimate
  \begin{equation}\label{trilinearbound}
    \sumsumsum_{j\sim J,m\sim M,n\sim N}\alpha_j\beta_m\gamma_n
    K(j^am^bn^c)
    \ll
    q^{\eps}JMN\Bigl(\frac{q^{1/2}}{MN}+\frac{q}{J^lMN}\Bigr)^{1/(2l)}	
  \end{equation}
  holds, where the implicit constant depends on $\eps$, $a$, $b$, $c$,
  $l$ and on the complexity of~$\mcF$.
\end{theorem}

\begin{remark}
  The bound \eqref{trilinearbound} is non-trivial as long as
  \[
    J\geq q^{\delta},\quad MN\geq q^{1/2+\delta},
  \]
  for some~$\delta>0$, taking $l$ sufficiently large so that
  $J^lMN\geq q^{1+\delta}$.
\end{remark}

Our methods make it possible to bound other kinds of multilinear sums of
\lgood trace functions with summation variables varying over ranges that
may be quite short. For instance, following Xi and Zheng~\cite{XZ}, one
can show that, as long as
\[
  q^{1/6+\delta} \leq M \leq q^{1/2-\delta},\quad MN^2\geq
  q^{1/2+\delta}
\]
for some $\delta>0$, the ``quadrilinear'' sum
\[
  \sumsum_{m_1,m_2\sim M}\alpha_{m_1}\alpha_{m_2}\sumsum_{n_1,n_2\sim
    N}K(\ov m_1n_1-\ov m_2n_2),
\]
is bounded by $\ll \|\bfalpha\|_2^2MN^2(MN)^{-\eta}$ for some $\eta>0$,
depending on~$\delta$. This bound is non-trivial for
\[
  M=N=q^{1/6+\delta}
\]
and $\delta\in \mathopen]0,1/3\mathclose[$.

The proof uses the analytic approach in \cite{XZ} together with the following useful special
case of Proposition~\ref{prop-Sigmagalant}.

\begin{proposition}\label{multicorrelation}
  Let~$q$ be a prime number and let~$\mcF$ be a \lgood sheaf
  on~$\Aa^1_{\Ff_q}$.
  
  For~$l\geq 2$ an even integer, $r,s\in \Fq$ and
  $\uple{v}\in \Ff_q^{2l}$, let
  \[
    \bfK(r,s,\bfv)=\prod_{i=1}^lK(s(r+v_i)) \ov{K(s(r+v_{i+l}))}
  \]
  and 
  \[
    \Sigma_I(\uple{v})=\sum_{(r,s)\in\Fq\times \Fqt}\bfK(r,s,\bfv).
  \]

  There exists an integer~$C\geq 1$, depending only on~$l$ and on the
  complexity of~$\mcF$ such that, if $q\geq C$, there exist algebraic
  varieties $\mcV_1$ and $\mcV_1^{\Delta}$ over $\Ff_q$ such that
  \begin{gather*}
    \Aa^{2l}_{\Fq}\supset \mcV_1\supset \mcV_1^\Delta
    \\
    \dim (\mcV_1) \leq \frac{3}{2}l,\quad \dim
    (\mcV^\Delta_1) \leq l,
    \\
    \max(\deg(\mcV_1),\deg(\mcV_1^{\Delta}))\leq C
  \end{gather*}
  and 
  \[
    |\Sigma_{I}(\bfv)|\leq
    \begin{cases}
      Cq&\text{ for }\bfv\not\in \mcV_1(\Fq),
      \\
      Cq^{3/2}&\text{ for }\bfv\not\in \mcV^\Delta_1(\Fq),
      \\
      Cq^2&\text{ for } \bfv\in \mcV_1^{\Delta}(\Ff_q).
    \end{cases}
    \]
 \end{proposition}
   
 Combining this with the counting Lemma~\ref{lm-sz}, we deduce that for
 any $V\in [1,q/2[$, the estimate
 \[
   \frac{1}{V^{2l}}\sum_{\bfv\in [V,2V]^{2l}}\Bigl
   |\sum_{(r,s)\in\Fq\times \Fqt}\bfK(r,s,\bfv)\Bigr |\ll
   q\Bigl(1+\frac{q}{V^l}\Bigr),
 \]
 holds, where the implicit constant depends only on $l$ and on the
 complexity of $\mcF$ (this is a special case of Proposition
 \ref{pr-strat}).  This proves that, for any $V\geq q^{1/l}$, the sum
 $\Sigma_I(\bfv)$ achieves square-root cancellation on average over
 $\bfv\in [V,2V]^{2l}$.

As another example, one can check, using Proposition \ref{multicorrelation} or suitable variants, that the results of \cite{Xi} extend to trace functions of \lgood sheaves.

\subsection{Cubic toroidal moments}\label{sec-toric}

The original motivation which led to Theorem~\ref{thmType12} was the study of what we call \emph{cubic
  toroidal moments}\footnote{\ Also called \emph{cubic mixed moments}.}
of special values of Dirichlet $L$-functions at the central point,
generalizing the quadratic case considered by Fouvry, Kowalski and
Michel in~\cite{FKMAA}; we refer to the introduction of this paper for a
general discussion.

In the present context, we consider a prime number~$q$ and non-zero
integers~$a$, $b$, $c$, and we define
\[
  %\label{triplemoment}
  M_{a,b,c}(q)=
  \frac{1}{q-1}\sum_{\chi\mods
    q}L(1/2,\chi^a)L(1/2,\chi^b)L(1/2,\chi^c).
\]

We wish to evaluate this moment asymptotically as
$q\rightarrow +\infty$.

After permuting $(a,b,c)$ and applying complex conjugation with
$L(1/2,\overline{\chi})=\overline{L(1/2,\chi)}$, we see that it is enough
to consider the moments $M_{a,b,c}(q)$ and $M_{a,b,-c}(q)$ when $a$, $b$
and~$c$ are \emph{positive} integers, which we will always assume from now on in this section.

Define exponential sums modulo~$q$ by
\[
  K_{a,b,c}(u;q)=\frac{1}{q}\sum_{\substack{x,y,z\in\Ff_q\\
      x^ay^bz^c=u}}e\Bigl(\frac{x+y+z}{q}\Bigr),\quad\quad
  u\in\Ff_q^{\times}.
\]

One can show that there exists a natural lisse $\ell$-adic sheaf
$\mcK_{a,b,c}$ on the multiplicative group with trace
function~$u\mapsto K_{a,b,c}(u)$, whose complexity is bounded in terms
of $(a,b,c)$ only. This sheaf is pure of weight~$0$.

In the companion paper~\cite{FKMSmoment}, we
prove the following result.

\begin{theorem}\label{mainthmabc}
  Let $a$, $b$, $c\geq 1$ be setwise coprime integers.
  Let~$\eps\in\{-1,1\}$. Assume that $\mcK_{a,b,\eps c}$ is \good in
  the sense of \textup{Definition~\ref{defgoodsheafintro}}.

  There exist $\eta=\eta(a,b,c)>0$ and $D_{a,b,\eps c}\geq 1$ 
  such that the  estimates 
  \begin{align*}
    M_{a,b,\eps c}(q)&\geq D_{a,b,\eps c}+O(q^{-\eta}),\\
    M_{a,a,\eps c}(q)&=D_{a,a,\eps c}+O(q^{-\eta}),\hbox{ if  moreover $a=b$},
  \end{align*}
  hold for all primes~$q$, where the implied constant depends only
  on~$(a,b,c)$.
\end{theorem}

In order for this result to be applicable, we also prove, using the
works of Katz and Katz--Tiep, that almost all the sheaves
$\mcK_{a,b,c}$ are \good.

\begin{theorem}
  Let $a$, $b$, $c\geq 1$ be setwise coprime integers.

  \begin{enumth}
  \item The sheaf $\mcK_{a,b,c}$ is \emph{\goodp}.
  \item The sheaf $\mcK_{a,b,-c}$ is \emph{\goodp} unless one of the
    following holds:
  \begin{itemize}
  \item $a=c$ or $b=c$, so that $(a,b,-c)=(c,b,-c)$ or $(a,c,-c)$
    (triples of that shape are called \emph{induced});
  \item $a+b=c$ and~$c\in \{2,3,4\}$ (e.g. $(1,1,-2),\ (1,2,-3)$ or
    $(1,3,-4)$); (triples of that shape are called \emph{solvable}).
    % , and up to switching $a$ and~$b$, a solvable triple is one of
    % $(1,1,-2)$, $(1,2,-3)$, $(1,3,-4)$.
    % \item Otherwise, $a+b+c\equiv 0\mods 2$ and
    %   \[
    %     \max(a+b,c)-(a,c)-(b,c)+1=4.
    %\]
  \end{itemize}
\end{enumth}
\end{theorem}

As a corollary we obtain a non-vanishing result:

\begin{corollary}\label{Ldonotvanish} Let $a$, $b$, $c$ be three
  non-zero setwise coprime integers. For~$q\geq q_0$ large enough, the lower
  bound
  \[
    |\{\chi\mods q\,\mid\,
    L(1/2,\chi^{a})L(1/2,\chi^b)L(1/2,\chi^c)\not=0\}|\gg
    q/\log^{12}q
  \]
  holds, with an implied constant depending only on $a$, $b$, $c$
  and~$q_0$.
\end{corollary}

In a similar vein, Berta and zur Verth~\cite{BV}, using the results of
this paper, obtained non-vanishing results for products of two
Dirichlet $L$-values under angular constraints on the associated Gauss
sums (compare with~\cite{MAMS}*{Thm.\,1.8}):

\begin{theorem}[Berta--zur Verth]
  Let $a,b\in \Zz\setminus \{0\}$ and let $I\subset \Cc^{(1)}$ be an
  interval with positive measure in the unit circle. There exists a
  real number $c=c(a,b,I)>0$ such that for all large enough prime
  numbers~$q$, the bound
  \[
    |\{\chi\mods q,\ \chi\not=1,\ L(1/2,\chi^a)L(1/2,\chi^b)\not=0,\
    \eps(\chi)\in I\}|\geq c(q-2),
  \]
  holds, where $\eps(\chi)$ denotes the unitarily normalized Gauss sum
  of~$\chi$.
\end{theorem}

\subsection{The case of solvable groups}

It remains to examine whether a form of Theorem \ref{thmType12} holds
when the geometric monodromy group $G$ is a finite solvable group. For
the methods of the present paper to work, one needs at least $G$ to
contain some large core subgroup $N$ that is non-nilpotent
(see Proposition~\ref{pr-abstract-gkr-nonnilp}). A choice that works in many
but not all situations is $N=G^{(1)}=[G,G]$ the derived subgroup; recall
that $G^{(1)}$ is non-nilpotent if and only if the quotient $G/F(G)$ is
non-abelian (and therefore admits irreducible representations of
dimension $\geq 2$, see \S \ref{secdiagTc}) where $F(G)$ denotes the
nilpotent radical of $G$ (or ``Fitting subgroup''). If $G^{(1)}$ is
nilpotent one would be left to check whether some rank $1$ sheaves have
small groups of affine automorphisms (cf. \S \ref{rank1}). In
forthcoming papers, we will establish Theorem \ref{mainthmabc} for the
induced and for the solvable triples. The induced ones lead either to
rank $1$ or to non-solvable sheaves. For the solvable triples, the
geometric monodromy groups are $S_2,S_3$ and $S_4$. The first is abelian
and the arguments of \S \ref{rank1} apply; the second and third are
solvable but not nilpotent and we choose $N=S_3$ and $N=A_4$
respectively.

\subsection{Outline of the paper}

\begin{itemize}
\item In Section~\ref{secladicgeneral}, we discuss the relevant facts
  from algebraic geometry. Of particular importance is our general
  form of Xu's idea (see Theorem~\ref{XuStep}).

\item In Section~\ref{sec:goursat}, we recall Goursat's Lemma and
  present our generalized version of the Goursat--Kolchin--Ribet
  criterion (Proposition~\ref{pr-abstract-gkr-nonnilp}). Our main
  application is Proposition \ref{pr-good-gkrlemma} and we conclude the
  section with its proof when the geometric monodromy is infinite.
\item Section \ref{gkrlemmafinite} is devoted to the proof of Proposition \ref{pr-good-gkrlemma} when the geometric monodromy group is finite; the proof is longer as it requires the introduction of data associated with the normal subgroups of the monodromy group.  
\item These results are exploited in Section \ref{diagonalestimates} to control the loci where various families of sums of products of trace functions in one variable fail to exhibit square-root cancellation.
\item In Section~\ref{bilinearstart}, we present in a general abstract
  form the basic implication of the ``shift by $+uv$'' method to study
  bilinear forms, reducing statements like Theorem~\ref{thmType12}
  to estimates for certain ``complete'' sums of the kernel.
    This part does not require $K$ to be a trace function and we express
  it in the form of Proposition~\ref{propreduct}. We also
  explain at the end of this section how this proposition will be
  applied for trace functions.
  % the
  % outcome are the two Propositions \ref{propreductypeI} and
  % \ref{propreductypeII}.
  
\item Sections~\ref{secstrat} and~\ref{secproofmainthm} provide the core
  of the proof: combining Xu's idea with other tools, we show how to
  obtain good estimates (stratification theorems) within families of auxiliary complete sums arising from
  Proposition~\ref{propreduct}: these sums are sums of products of affine transforms of $K$ and its complex conjugate in {\em two or three variables}. The method of moments allows to reduce that question to a similar one but for families of sums  {\em in one variable}. For these we use the results of \S \ref{diagonalestimates}.

\item In Section~\ref{sec-trilinear}, we give the proof of
  Theorem~\ref{thmtriplesum}.

\item Finally, Section~\ref{sec-galant} provides a series of examples
  of \good sheaves, relying especially on the work of Katz
  (see~\cite{GKM} and~\cite{ESDE} in particular). We include examples
  with finite monodromy group, especially those arising from
  hypergeometric sheaves, from the tables of Beukers and
  Heckman~\cite{BH}. This section can be read independently of the proof of
  the main results, and is meant to show the variety of situations
  where the results apply.
  
\end{itemize}

\subsection*{Acknowledgements}

This research was partially supported by the SNF grants $197045$,
$219220$, the SNF-ANR ``Etiene'' grant $10003145$, the NSF grant DMS-2502029, and a Sloan Research Fellowship. 

We are grateful to Filippo Berta, Peter Sarnak, Svenja zur Verth and
Ping Xi for  encouragements and useful comments. 

EF thanks EPFL and ETHZ for their hospitality and for providing an environment conducive to a fruitful collaboration.

\subsection*{Notation}

For $z\in\Cc$, we denote $e(z)=e^{2i\pi z}$.

Given $N>0$ we write $n\sim N$ for $n\in \Zz\cap [N,2N)$.

Given a tuple of complex numbers indexed by a finite set $\mcM$, $\bfalpha=(\alpha_m)_{m\in \mcM}$ and $p>0$, we set
$$\|\bfalpha\|_p:=(\sum_{m\in\mcM}|\alpha_m|^p)^{1/p},\ \|\bfalpha\|_\infty:=\max_{m\in\mcM}|\alpha_m|.$$

We use the notation $f\ll g$ and $f=O(g)$, for $f$ and $g$ defined on a
set~$X$, synonymously: either notation means that there exists a
constant~$C\geq 0$ such that $|f(x)|\leq Cg(x)$ for all $x\in X$.

We will often have statements saying that a certain function $f$ is  bounded by another function $g$
``polynomially'' in terms of certain quantities, say $\alpha$, $\beta$; by this we mean that there exist real numbers~$c\geq 0$ and~$A\geq 0$ such
that for all $x$ in the domain of definition of~$f$ and~$g$, the
inequality
\[
  |f(x)|\leq c(\alpha\beta)^A
\]
holds. This will apply in particular to certain implicit constants in
bounds $f\ll g$; we then mean that
\[
  |f(x)|\leq c(\alpha\beta)^Ag(x)
\]
for suitable~$c$ and~$A$ as above.

% Given $k$ a field and $(r,s)\in k\times \kt$ we denote by $\gamma_{r,s}$ the affine transformation acting on $\Pp^1(k)$
% $$\gamma_{r,s}:x\mapsto s(x+r).$$

% We denote by $\Aff(k)$ the group of $k$-points of affine transformations and by 
% $$\slo:\gamma_{r,s}\in\Aff(k)\mapsto s\in\kt$$
% the ``slope'' morphism.

% Given $c\in\Zz-\{0\}$ we denote by $\gamma_{r,s,c}$ the morphism on $\Pp^1(k)$
% $$\gamma_{r,s,c}:x\mapsto s(x+r)^c.
% $$

\section{Algebro-geometric preliminaries}\label{secladicgeneral}

We recall here some facts from algebraic geometry and establish our
notation and conventions.

Let~$k$ be a field. An algebraic variety~$X$ over~$k$ is a reduced and
separated scheme of finite type over the spectrum of~$k$.

By the \emph{degree} of a subvariety $X$ of projective space, we mean
the sum of the degrees of its irreducible components (which may have
different dimensions).

For a prime number $\ell$ different from the characteristic of~$k$, we
denote by $\Der(X,\bQl)$ the bounded derived category of complexes of
$\ell$-adic constructible sheaves on~$X$. An object of this category
will be called simply an $\ell$-adic \emph{complex}; by $\ell$-adic
\emph{sheaf} on~$X$, we will mean a constructible $\ell$-adic sheaf. We
will sometimes just speak of \emph{complex} or \emph{sheaf}, if no
confusion concerning $\ell$ can arise.

Unless stated otherwise, by an isomorphism of $\ell$-adic sheaves or
complexes, we will always mean a \emph{geometric isomorphism}.

Whenever we consider a field~$k$ and a prime $\ell$ invertible in~$k$, we
also \emph{fix} an isomorphism $\iota\colon \bQl\to\Cc$, and \emph{we
  use it implicitly to identify $\ell$-adic numbers with complex
  numbers}. In particular, if $k$ is a finite field, we can speak of the
value in~$\Cc$ of the \emph{trace function} of a complex~$M$ on~$X$ at a
point $x\in X(k)$. We sometimes denote this value by $t_M(x;k)$, and
correspondingly write $t_M(x;k_n)$ if $k_n/k$ is a finite extension
of~$k$ of degree~$n$ and~$x\in X(k_n)$. When we speak of mixed or pure
complexes of $\ell$-adic sheaves, we also mean $\iota$-mixed or
$\iota$-pure.

If~$X$ is a smooth and geometrically connected curve
over~$k$,\footnote{\ Which, in this paper, will most often be the
  affine line.} a \emph{middle-extension sheaf} is an $\ell$-adic
sheaf~$\mcF$ on~$X$ such that for some (equivalently any) open
dense subset~$U\subset X$ on which $\mcF$ is lisse, with open immersion $j\colon U\to X$, the
natural adjunction morphism $\mcF\to j_*j^*\mcF$ is an isomorphism (see~\cite[(7.3.1)]{ESDE}).

Under the same assumption on~$X$, let~$\mcF$ be an $\ell$-adic sheaf. By
definition of constructible sheaves, there exists a dense open subset
of~$X$, say~$U$, such that the restriction of~$\mcF$ to~$U$ is
lisse. Thus $\mcF|U$ corresponds to a continuous $\ell$-adic
representation of the \'etale fundamental group of~$U$, well-defined up to
isomorphism. The arithmetic monodromy group of $\mcF|U$ is classically
defined to be the Zariski-closure of the image of this representation,
viewed as a $\bQl$-linear algebraic group. Up to isomorphism, this group
is independent of the choice of~$U$, and we call it the \emph{arithmetic
  monodromy group} of~$\mcF$. Similarly, let~$\ov{k}$ be an algebraic
closure of~$k$; the restriction of the representation to the fundamental
group of $U_{\ov{k}}$ defines the geometric monodromy group of
$\mcF|U_{\ov{k}}$. It is independent of~$U$, up to isomorphism, and
called the \emph{geometric monodromy group} of~$\mcF$.

The \'etale fundamental group of~$U$ is a quotient of the Galois
group~$\Gal(\overline{k(X)}/k(X))$ of the function field of~$X$, and we
will sometimes identify the restriction of~$\mcF$ to~$U$ with a
finite-dimensional $\ell$-adic representation of this Galois group.

We will also often use the fact that the arithmetic (resp. geometric)
monodromy group of~$\mcF$ coincides with that of the
sheaf~$j_!j^*\mcF$. This allows us to assume in many proofs
that~$\mcF$ is the extension by zero of a lisse sheaf on some dense
open subset of~$X$.

The \emph{dual} of~$\mcF$ is defined to be the sheaf
$\mcF^{\vee}=\underline{\mathrm{Hom}}(\mcF,\bQl)$; it is lisse on~$U$
and coincides there with the lisse sheaf associated to the
contragredient of the representation corresponding to~$\mcF$. In
particular, if~$\mcF$ is pure of weight~$0$ on~$U$, then the trace
function of~$\mcF^{\vee}$ on~$U$ is (under~$\iota$) the complex
conjugate of the trace function of~$\mcF$. This property is also true
on all of~$X$ if ~$\mcF$ is a middle-extension sheaf (by a
theorem of Gabber~\cite[proof of Duality
Lemma 1.8.1(1)]{MMP}).

%% TODO:
% conventions of algebraic geometry
% conventions for sheaves, etc
% basics of complexity theory, including RH

As in previous works, we will use the following standard lemma (see
\cite[Lemma~1.7]{Xu}).

\begin{lemma}\label{lm-sz}
  Let $l\geq 1$ be an integer and let $A>0$. Let $q$ be a prime number
  and let $X\subset \Aa^l_{\Fq}$ be an algebraic variety of dimension
  $d\geq 0$ given by the vanishing of $\leq A$ polynomials of
  degree~$\leq A$. Let~$V$ be an integer with $0\leq V<q/2$. Then
  $$
  |\{x=(x_1,\ldots, x_l)\in \Zz^l\,\mid\, V\leq x_i\leq 2V\text{ for }
  1\leq i\leq l \text{ and }x\in X(\Ff_q) \}|\ll V^d,
  $$
  where the implicit constant depends only on $l$ and $A$.
\end{lemma}

We will use Sawin's quantitative sheaf theory~\cite{qst}, and we
recall the framework.  Let $k$ be a field and $X/k$ a quasi-projective
variety defined over $k$, given with a locally-closed embedding
$u\colon X\to \Pp^N$ for some integer~$N\geq 1$. Let $\ell$ be a prime
number invertible in~$k$. Sawin defined the \emph{complexity} $c_u(M)$
of objects of $\Der(X,\bQl)$ (see~\cite[Def.\,6.3]{qst}). This is a
non-negative integer which controls quantitatively most invariants
of~$M$.

If $X$ is an open subvariety of~$\Aa^n$ for some integer~$n\geq 1$, we
will always use the natural embedding $u\colon X\to \Pp^n$ to compute
the complexity of objects of $\Der(X,\bQl)$, and will use the notation
$c(M)$ instead of $c_u(M)$. If~$n=1$, then for a
\emph{middle-extension sheaf}~$\mcF$ on~$\Aa^1_{k}$, viewed as a
complex located in some degree, and with~$u$ the obvious embedding
of~$\Aa^1$ in~$\Pp^1$, the complexity $c_u(\mcF)$ is comparable with
the complexity~$c_{\mathrm{FKM}}(\mcF)$ defined by Fouvry, Kowalski
and Michel~\cite[Def.\,1.13]{FKM1}; more precisely, we have
\[
  c_{\mathrm{FKM}}(\mcF)\leq c_u(\mcF)\leq 3c_{\mathrm{FKM}}(\mcF)^2
\]
by~\cite[Cor.\,7.4]{qst}.

Among the properties of the complexity that we will use, we only quote
here a version of Deligne's Riemann Hypothesis.

\begin{theorem}\label{th-rh}
  Let~$k$ be a finite field. Let~$\mcF$ be an $\ell$-adic sheaf
  on~$\Aa^1_{k}$ which is mixed of weights~$\leq 0$. Assume that
  \[
    H^2_c(\Aa^1_{\ov{k}},\mcF)=0
  \]
  or equivalently that the Galois representation associated to~$\mcF$
  has no coinvariants with respect to the geometric monodromy group. We
  then have
  \[
    \sum_{x\in k} t_{\mcF}(x;k)\ll c_u(\mcF)|k|^{1/2}
    %% +(c_u(\mcF)+c_u(\mcG)),
  \]
  where the implied constant is absolute and the embedding~$u$ is the
  obvious one.
\end{theorem}

\begin{proof}
  This is essentially a special case of~\cite[Th.\,7.13]{qst}, taking
  into account that the complexity of the embedding~$u$ is absolutely
  bounded. More precisely, let~$U$ be an open dense subset such
  that~$\mcF|U$ is lisse. It is elementary that the contribution of
  the points in $k\setminus U(k)$ is $O(1)$ (see the proof of
  loc. cit.). We consider the Jordan--Hölder components of the
  representation corresponding to~$\mcF|U$, and write the sum over
  $x\in U(k)$ as the sum over the components; the assumption implies
  that none of these is geometrically trivial, and thus for each
  component (say~$\mcF_i$) we get
  \[
    \sum_{x\in U(k)} t_{\mcF_i}(x;k)\ll c_u(\mcF_i)|k|^{1/2}
  \]
  by loc. cit., and the conclusion then follows after summing over~$i$
  from~\cite[Th.\,6.15]{qst}.
\end{proof}

% \begin{remark}
%   We also recall that the complexity of the quasi-projective
%   embedding~$u$ is bounded in families~\cite[Prop.\,6.19]{qst}.
% \end{remark}

For a finite field~$k$, an $\ell$-adic additive character~$\psi$ of $k$
or an $\ell$-adic multiplicative character~$\chi$, we use the usual
notation $\mcL_{\psi}$ or $\mcL_{\chi}$ for the associated
Artin--Schreier or Kummer sheaf, with trace functions $x\mapsto\psi(x)$
or $x\mapsto \chi(x)$.

\begin{proposition}\label{pr-fkm2}
  Let $k$ be a finite field of characteristic~$q$ and $\ov{k}$ an
  algebraic closure of~$k$. Let $\sheaf{F}$ be a constructible
  $\ell$-adic sheaf on~$\Aa^1$ over~$k$.  Assume that the restriction
  of~$\mcF$ to a dense open subset is a lisse geometrically irreducible
  sheaf.

  Let
  \[
    \mathbf{B}_{\mcF}=\Aff_\mcF=\Bigl\{\gamma=\begin{pmatrix} a& b\\0&d
    \end{pmatrix}
    \in\PGL_2(\ov{k})\,\mid\, \gamma^*(\mcF)\simeq \mcF\Bigr\},
  \]
  where $\simeq$ denotes geometric isomorphism. This is an algebraic
  subgroup of~$\PGL_2(\ov{k})$, and at least one of the following
  properties holds:
  \par
  \begin{enumth}
  \item The sheaf~$\mcF$ is of generic rank~$1$.
    % geometrically isomorphic to an
    % Artin--Schreier sheaf, or the tensor product of an Artin--Schreier
    % sheaf with a Kummer sheaf.
  \item The group $\hautb_{\sheaf{F}}$ is a finite cyclic group of size
    bounded polynomially in terms of the complexity of~$\mcF$.
    % and
    % \[
    %   |\hautb_{\sheaf{F}}(k)|\leq 10c_{\mathrm{FKM}}(\sheaf{F})^2;
    % \]
  \item The complexity of~$\mcF$ is at least $q^{\delta}$ for some
    absolute constant~$\delta>0$.
    % We have $c_{\mathrm{FKM}}(\mcF)\geq (q/10)^{1/2}$.
  \end{enumth}
\end{proposition}

\begin{proof}
  Let~$j\colon U\to \Aa^1$ be the open immersion of an open dense subset
  where $\mcF$ is lisse. Define $\mcF^*=j_*j^*\mcF$; this is a
  geometrically simple middle-extension sheaf with complexity bounded
  polynomially in terms of the complexity of~$\mcF$ (this is clear from
  the definition of $c_{\mathrm{FKM}}$).  We have
  $\mathbf{B}_{\mcF}\subset \mathbf{B}_{\mcF^*}$: if
  $\gamma\in\mathbf{B}_{\mcF}$, then the existence of a geometric
  isomorphism $\gamma^*(\mcF)\simeq \mcF$ implies that the
  middle-extension sheaves $\mcF^*$ and $\gamma^*\mcF^*$ are
  geometrically isomorphic on some dense open set of~$\Aa^1$, hence are
  geometrically isomorphic.

  If~$\mcF^*$ is an Artin--Schreier sheaf, then the first property of
  our claim holds. Otherwise, the sheaf~$\mcF^*$ is a Fourier sheaf in
  the sense of~\cite[Def.\,7.3.5]{ESDE} and we can
  apply~\cite[Th.\,6.3,\,Prop.\,6.4]{FKM2} to the ``naïve'' (inverse)
  Fourier transform~$\mcG$ of~$\mcF$ (in the sense
  of~\cite[7.3.3]{ESDE}); the group denoted $\mathbf{B}_{\mcG}$
  in~\cite{FKM2} coincides with $\mathbf{B}_{\mcF^*}$ by the Fourier
  inversion formula~\cite[7.3.8]{ESDE}.

  The first alternative in~\cite[Th.\,6.3]{FKM2} is impossible since it
  would imply that~$\mcF^*$ is an Artin--Schreier sheaf, which we assumed is
  not the case.  The second would imply that $\mcG$ is geometrically
  isomorphic to $\mcL_{\psi}\otimes \mcL_{\chi}$ for some additive
  character~$\psi$ (resp. some multiplicative character~$\chi$,
  necessarily non-trivial since otherwise we would be in the case of an
  Artin--Schreier sheaf). It is straightforward that $\mcF^*$ would then
  be of generic rank~$1$ (this amounts to the computation
  \[
    \sum_{x\in k}\chi(x)\psi(x)\psi(xy)=\tau\overline{\chi}(1+y),
  \]
  of trace functions, where $\tau$ is a Gauss sum), thus $\mcF$ would be
  of generic rank~$1$.

  The third possibility is that $\mathbf{B}_{\mcF^*}$ is finite of size
  bounded by $10 c_{\mathrm{FKM}}(\mcG)^2$; since the complexity
  of~$\mcG$ is bounded polynomially in terms of that of~$\mcF$, this
  gives the second possibility in our statement.

  Finally, the fourth possibility is that the complexity
  $c_{\mathrm{FKM}}(\mcG)$ is $\geq (q/10)^{1/2}$; as before this
  implies the third possibility.

  To conclude, we simply note that in the second case, the fact that the
  group~$\hautb_{\mcF^*}$ is cyclic is not stated in~\cite{FKM2}, but
  follows from the proof (see~\cite[p.\,1729]{FKM2}); its
  subgroup~$\hautb_{\mcF}$ is then also cyclic.
\end{proof}

We now state our version of the result of Xu which we already
mentioned. Roughly speaking, it states that bounds for the moments of a
trace function imply stratification results, which then imply further
analytic results about the trace function which are \emph{significantly
  stronger} than those that can be obtained directly from the moment
bounds.  A concrete form of this idea is the following theorem.

\begin{theorem}\label{XuStep}
  Let~$k$ be a finite field and $\ell$ a prime number different from the
  characteristic of~$k$. Let $X$ be a quasiprojective variety with a
  locally closed embedding in $\Pp^d_{k}$ for some integer~$d\geq
  0$. Let $M$ be an object of $\Der(X,\bQl)$ which is mixed of
  \emph{integral} weights.

  Assume that there exists a positive integer $m$, a real number~$A$ and
  a real number $B\geq 1$ such that
  \[
    \sum_{x \in X(k_n)} \abs{ t_M ( x; k_n)}^{2m} \leq B \abs{k_n}^A
  \]
  for all integers~$n\geq 1$, where $k_n$ denotes the extension of
  degree~$n$ of~$k$ in some algebraic closure of~$k$.
  
  There exists a stratification of $X$ into closed subschemes $X^{(w)}$,
  defined for integers $w$, such that
  \begin{enumth}
  \item We have $X^{(w)} \subset X^{(w-1)}$ for all
    $w$.
  \item We have $\dim X^{(w)} \leq \lfloor A \rfloor - mw$ for all~$w$.
  \item Each $X^{(w)}$ is a union of subvarieties of total degree
    bounded only in terms of $d$ and $c(M)$.
  \item For all $n\geq 1$ and $x \in (X\setminus X^{(w+1)})(k_n)$, we
    have
    \[
      t_{M}(x ; k_n) \ll \abs{k_n}^{\frac{w}{2}},
    \]
    where the implicit constant depends only on $d$ and $c(M)$.
  \end{enumth}
\end{theorem}

\begin{proof}
  It follows from~\cite[Lemma\,6.26]{qst} applied to~$M$ that there
  exist subvarieties
  \[
     Y_{d+1}\subset Y_{d}\subset \cdots \subset Y_1\subset Y_0=X
  \]
  such that 
  \begin{itemize}
  \item For $0\leq i\leq d$, each $Y_{i}\setminus Y_{i+1}$ is smooth, and $M$ restricted to
    $Y_{i}\setminus Y_{i+1}$ is lisse;
  \item Each $Y_i$ is the union of $\leq C$ subvarieties of degree
    $\leq C$, for some integer $C$ depending only on $d$ and the
    complexity of~$M$.
  \end{itemize}

  We then define~$X^{(w)}$, for $w\in\Zz$, to be the union over~$i$ of
  the connected components of dimension $\leq A-mw$ of
  $Y_{i}\setminus Y_{i+1}$. Then $X^{(w)}$ is certainly closed as the
  closure of each stratum is a union of strata of equal and smaller
  dimension. Moreover, condition (3) follows from the second property
  above, while the construction implies that $X^{(w)}\subset X^{(w-1)}$,
  and $\dim(X^{(w)})\leq \lfloor A-mw\rfloor=\lfloor A\rfloor -mw$
  (since $mw$ is an integer).

  Thus it only remains to prove the last property~(4).
  % $k$ by an that $ t_{M}(x ; k) \ll \abs{k}^{\frac{w}{2}} $ for a
  % point $x \in X(k)$ with $x \notin X^{(w+1)}$.
  Let $n\geq 1$, $w\in\Zz$ be integers and let
  $x\in (X\setminus X^{(w+1)})(k_n)$.  By definition, this means that $x$
  is contained in some smooth locally closed subvariety $Y$ of $X$ of
  dimension $> A - m(w+1)$ on which $M$ is lisse. This implies that all
  cohomology sheaves of~$M|Y$ are lisse, and hence are iterated
  extensions of irreducible lisse sheaves. Since $M$ is assumed to be
  mixed of integer weights, each of these irreducible lisse sheaves is
  pure of some integer weight. 
  %For $y\in Y(k)$, we have
  For $N\geq 1$ and $y\in Y(k_N)$, we have
  \[
    t_M(y;k_N)=\sum_{i\in\Zz} t_{\mcH^{2i}(M)}(y;k_N)-
    \sum_{i\in\Zz} t_{\mcH^{2i+1}(M)}(y;k_N),
  \]
  the difference of sums of trace functions of pure lisse
  sheaves. We can then apply a result of Xu~\cite[Theorem 3.5]{Xu}. 
  
  This
  shows that either
  \begin{equation}\label{eq-xu-bad}
    \limsup_{N \to\infty} \frac{1}{|k|^{N(\dim(Y)+m(w+1))}} \sum_{y \in
      Y( k_N)} \abs{t_{M}(y; k_N)}^{2m}\geq 1,
  \end{equation}
  or \[
    |t_M(x;k_n)|\leq D|k_n|^{w/2},
  \]
  with
  \[
    D=\sum_{i\in\Zz}\rank(\mcH^i(M|Y)).
  \]

  This last estimate gives the desired conclusion since $D$ is bounded
  in terms of~$d$ and of the complexity of~$M$
  (see~\cite[Th.\,6.8(8)]{qst} and~\cite[Th.\,6.15]{qst}; note that to
  apply the latter, it is important that~$M$ is lisse on~$Y$). So we
  need only check that~(\ref{eq-xu-bad}) is not valid.
  % \begin{enumerate}
  % \item \[ \abs{ t_{M}( x; k)} \leq C \abs{k}^{\frac{w}{2}} \] where
  %   $C$ is the sum of the ranks of the cohomology sheaves of $M$
  % \item or Note that $C$ appearing in case (1) is bounded in terms of
  %   $n$ and $c(M)$ by \cite[Theorem 6.8(8)]{qst}. (In fact this bound
  %   is linear in $c(M)$, but we won't use this). So it suffices to
  %   rule out case (2).
  However, using the assumption we find that
  \[
    \sum_{y \in Y( k_N)} \abs{t_{M}(y; k_N)}^{2m} \leq \sum_{x \in
      X(k_N)} \abs{ t_{M} ( x; k_N)}^{2m}\leq B |k|^{NA}
  \]
  for all~$N\geq 1$, hence
  \[
    \frac{1}{|k|^{N(\dim(Y)+m(w+1))}} \sum_{y \in Y( k_N)} \abs{t_{M}(y;
      k_N)}^{2m}
    \leq B |k|^{N(A-\dim(Y)-m(w+1))},
  \]
  which tends to~$0$ as $n\to +\infty$, so this alternative possibility
  does not occur.
  % But the numerator is certainly bounded by
  % $ \sum_{x \in X(k)} \abs{ t_{M} ( x; k)}^{2m} $ which by assumption
  % is $\ll q^A$ so that the following bound holds
  % $$ \frac{ \sum_{y \in Y( k)} \abs{t_{M}(y; k)}^{2m}}{ \abs{k}^{\dim
  %     Y} \abs{k}^{ (w+1)m}}\ll\abs{k}^{A- \dim Y- (w+1)m}.$$ Since
  % $\dim Y > A -m(w+1)$ the exponent is negative so the lim sup as
  % $\abs{k} \to \infty$ is $0$ and the case $(2)$ is
  % impossible.
\end{proof}

\section{Around Goursat's Lemma}
\label{sec:goursat}

\subsection{Group-theoretic notation}

We recall some standard definitions and notation from the theory of
(algebraic) groups.

Let $N$ be an algebraic group over an algebraically closed field of
characteristic zero and let $l\geq 1$ be an integer. We denote by
$\gamma_{l}(N)$ the $l$-th term of the lower central series of $N$,
i.e., we let
\[
  \gamma_1(N)=N,\ \gamma_2(N)=[N,N],\ \gamma_3(N)=[\gamma_2(N),N],\
  \gamma_l(N)=[\gamma_{l-1}(N),N].
\]

We have then the inclusions
\[
  N=\gamma_1(N)\triangleright \gamma_2(N)\triangleright
  \cdots\triangleright \gamma_{l}(N)\triangleright \cdots\triangleright
  \gamma_\infty(N),
\]
and each $\gamma_l(N)$ is a characteristic subgroup of $N$ (i.e., is
invariant under all automorphisms of~$N$).

In particular, if $N$ is a normal subgroup of some group $G$, then each
$\gamma_l(N)$ is also normal in $G$.

Finally, we recall that $N$ is nilpotent if and only if
$\gamma_\infty(N)=1$.

We also recall a basic lemma which will be very  useful below.

\begin{lemma}\label{coinvlemma}
  Let~$G$ be a group and let~$C$ be a field.  Let
  $\rho\colon G\hookrightarrow \GL(V)$ be a faithful finite dimensional
  irreducible representation on a finite-dimensional $C$-vector space.

  For any non-trivial normal subgroup~$N$ of~$G$, the space coinvariants
  of $\Res^G_B\rho$ is zero. 
\end{lemma}

\begin{proof}
  By definition, this space of coinvariants is the quotient of~$V$ by
  the subsapce
  \[
    V_N=\peter{(\rho(n)-1)v,\ v\in V,\ n\in N}.
  \]
  
  Since $N$ is normal in~$G$, the space $V_N$ is $G$-invariant and is
  therefore either $\{0\}$ or $V$ since~$\rho$ is irreducible. The
  former is impossible, since it would imply that
  $N\subset\ker(\rho)=\{1\}$, contrary to the assumptions.
\end{proof}

%%\subsection{A generalized Goursat-Kolchin-Ribet criterion}

\subsection{Goursat's Lemma}

We recall a version of the classical lemma of Goursat (see also the
version by Serre in \cite[Prop.\,1.6]{SerreBook}).

\begin{lemma}[Goursat's Lemma]\label{lm-goursat}
  Let $G_1$ and $G_2$ be groups and let $G\subset G_1\times G_2$ be a
  subgroup such that the restriction to~$G$ of the projections
  $p_1\colon G_1\times G_2\to G_1$ and $p_2\colon G_1\times G_2\to G_2$
  are surjective.

  There exist normal subgroups $H_i\subset G_i$ and an isomorphism
  $\varphi\colon G_1/H_1\to G_2/H_2$ such that
  \[
    H_1\times H_2\subset G= \{(x,y)\in G_1\times G_2\,\mid\,
    \varphi(xH_1)=yH_2\}.
  \]
\end{lemma}

We recall the proof for completeness.

\begin{proof}
  The kernel of the restriction of $p_1$ (\resp $p_2$) to~$G$ is
  identified with a subgroup $H_2$ of $G_2$ (\resp a subgroup $H_1$ of
  $G_1$). One then checks that the image of the natural morphism
  \[
    G\to G_1/H_1\times G_2/H_2
  \]
  is the graph of an isomorphism $\varphi$; the data of $H_1$, $H_2$ and
  $\varphi$ has the stated property (e.g. for the inclusion
  $H_1\times H_2\subset G$, note that if $(x,y)\in H_1\times H_2$, then
  by definition we have $(x,1)\in G$ and $(1,y)\in G$, so that
  $(x,y)\in G$).
\end{proof}

\begin{definition}
  Let $G_1$ and $G_2$ be groups and let $G\subset G_1\times G_2$ be a
  subgroup such that the restriction to~$G$ of the projections
  $G_1\times G_2\to G_1$ and $G_1\times G_2\to G_2$ are surjective.  The
  data $(H_1,H_2,\varphi)$ from Goursat's Lemma is called a
  \emph{Goursat datum} for~$G$.
\end{definition}

\subsection{A criterion for vanishing of coinvariants}

In applications, the intended conclusion is often that
$G=G_1\times G_2$, which one hopes to achieve using extra assumptions.
Generalizations of this criterion exist to handle subgroups of a product
of finitely many groups, and were used for instance by
Ribet~\cite{ribet} to study $\ell$-adic representations associated to
classical modular forms. We are interested in another variant: the
groups $G_i$ are given with some linear representation, and we wish to
understand when the space of (co)invariants of~$G$ in the tensor product
of the representation spaces may be non-zero.

\begin{proposition}\label{pr-abstract-gkr-nonnilp}
  Let~$C$ be an algebraically closed field of characteristic $0$.
  Let~$G$ be a group, and let $m\geq 2$ be an integer.

  For $1\leq i\leq m$, let $\rho_i\colon G\to \GL(V_i)$ be a
  representation of~$G$ on a finite-dimensional $C$-vector space~$V_i$,
  and let $G_i$ be the Zariski-closure of the image of~$\rho_i$ For
  $1\leq i\not=j\leq m$, let $G_{i,j}\subset G_i\times G_j$ be the
  Zariski-closure of the image of $\rho_i\times\rho_j$ and let
  $(H_{i,j},H_{j,i},\varphi_{ij})$ be its associated Goursat datum.
  
  Assume that, for each $1\leq i\leq m$, there exists a normal subgroup
  $N_i$ of $G_i$ such that the space of coinvariants of
  $\gamma_{m-1}(N_i)$ acting on $V_i$ is zero.

  Then one of the following properties holds:
  \begin{enumth}
  \item\label{zerocoinv} The space of $G$-coinvariants of the representation
    $\rho=\rho_1\otimes\cdots \otimes\rho_m$ is zero.
  \item\label{goursatexists} For each $i$, there exists an integer
    $j\not=i$ such that $N_i$ is \emph{not} contained in  $H_{i,j}$.
  \end{enumth}
\end{proposition}

\begin{proof} Replacing $G$ by its Zariski-closure in
  $G_1\times\cdots\times G_m$, we may assume that
  $G\subset G_1\times\cdots\times G_m$ is an algebraic subgroup and that
  the maps $\rho_i$ are algebraic and that each $G_i$ acts faithfully on
  $V_i$. In particular, $\rho_i$ maps $G$ surjectively onto $G_i$.

  We assume that \ref{goursatexists} is not satisfied and
  prove~\ref{zerocoinv}. We may assume, up to renumbering the
  representations, that the condition in this statement fails for
  $i=1$. For any $j\geq 2$, let $(H_{1,j},H_{j,1},\varphi_{1j})$ be the
  Goursat datum for the image $(\rho_1\times \rho_j)(G)$ of~$G$ in
  $G_1\times G_j$. From the proof of Lemma~\ref{lm-goursat} we have
  $$H_{1,j}=\rho_1(\ker(\rho_j)),\ H_{j,1}=\rho_j(\ker(\rho_1)).$$ 
  
  Our assumption that~\ref{goursatexists} is not satisfied is then that
  for all $2\leq j\leq m$, one has
  \begin{equation}
  	\label{l=2} N_1\subset \rho_1(\ker(\rho_j)). 
  \end{equation}	
    
  For $2\leq l\leq m$, let
  \[
    K_l=\bigcap_{2\leq j\leq l}\ker(\rho_j)\subset G.
  \]

  We claim that 
  \[
    \gamma_{m-1}(N_1)\subset \rho_1(K_m).
  \]

  Once we have this, we can argue as follows: the
  restriction~$\ov{\rho}$ of $\rho=\rho_1\otimes\cdots\otimes \rho_m$ to
  the subgroup $K_m$ is isomorphic to the restriction of
  $\rho_1\otimes 1\otimes\cdots\otimes 1$ to~$K_m$; the space of
  $G$-coinvariants of $\rho$ is a quotient of the space of
  $K_m$-coinvariants of~$\ov{\rho}$, and the latter is contained in the
  space of $\gamma_{m-1}(N_1)$-coinvariants of $\rho_1$ (since
  $\rho_1(K_m)$ contains~$\gamma_{m-1}(N_1)$).  Since the coinvariants
  of the latter are zero by assumption, the statement~\ref{zerocoinv}
  holds, and this concludes the proof.
  
  We now prove by induction on $2\leq l\leq m$ that
  \begin{equation}
    \label{gammainclusion}
    \gamma_{l-1}(N_1)\subset \rho_1(K_l).
  \end{equation} 

  For $l=2$ this is just \eqref{l=2}, and for $l=m$, this provides the
  claim above.
  
  Suppose we have proven \eqref{gammainclusion} for $2\leq l\leq m-1$.
  Since $\gamma_{l}(N_1)=[\gamma_{l-1}(N_1),N_1]$, it suffices to show
  that each commutator $[a,b]$ with $a\in\gamma_{l-1}(N_1)$ and
  $b\in N_1$ belongs to $\rho_1(K_{l+1})$ in order to deduce the
  inclusion for~$l+1$.
  
  By induction, we have $a=\rho_1(x)$ for some $x\in K_l$, and by
  \eqref{l=2} for $j=l+1$, we have $b= \rho_1(y)$ for some
  $y\in \ker(\rho_{l+1})$. We then get the formula
  \[
    [a,b]=\rho_1([x,y]).
  \]
  
  Since $K_{l}$ and $\ker(\rho_{l+1})$ are normal in $G$ we have
  $[x,y]=x(yx^{-1}y^{-1})\in K_{l}$ and
  $[x,y]=(xyx^{-1})y^{-1}\in \ker(\rho_{l+1})$, and thus
  $[x,y]\in K_{l}\cap \ker(\rho_{l+1})=K_{l+1}$, which completes the
  induction (compare with the proof of~\cite[Lemma\,3.3]{ribet}, due to
  Serre).
\end{proof}

Combining Proposition \ref{pr-abstract-gkr-nonnilp} with Lemma
\ref{coinvlemma}, we obtain the next result.

\begin{corollary}
  \label{pr-abstract-gkr-nonnilpirrep}
  Let $C$ be an algebraically closed field of characteristic $0$.
  Let~$G$ be a group, and let $m\geq 2$ be an integer.  For
  $1\leq i\leq m$, let $\rho_i\colon G\to \GL(V_i)$ be an irreducible
  representation of~$G$ on a finite-dimensional $C$-vector
  space~$V_i$. Let $G_i$ be the Zariski-closure of the image
  of~$\rho_i$.
  
  For $1\leq i\not=j\leq m$, let $G_{i,j}\subset G_i\times G_j$ be the
  Zariski-closure of the image of $\rho_i\times\rho_j$ and let
  $(H_{i,j},H_{j,i},\varphi_{ij})$ be its associated Goursat datum.
  
  Assume that for each $1\leq i\leq m$, there exists a normal
  subgroup~$N_i$ of $G_i$ such that $\gamma_{m-1}(N_i)\not=1$.
  % .(for
  % instance if $N_i$ is not nilpotent).

  One of the following properties holds:
  \begin{enumth}
  \item The space of $G$-coinvariants of the representation
    $\rho=\rho_1\otimes\cdots \otimes\rho_m$ is zero.
  \item For each $i$, there exists an integer $j\not=i$ such that $N_i$
    is \emph{not} contained in $H_{i,j}$.
  \end{enumth}
\end{corollary}

\begin{proof}
  Let~$i$ be an integer with $1\leq i\leq m$. Since the action of $G_i$
  on $V_i$ is irreducible and faithful by definition and since
  $\gamma_{m-1}(N_i)$ is a non-trivial normal subgroup of $G_i$, we can
  apply Lemma \ref{coinvlemma} to conclude that the space of
  coinvariants of $\gamma_{m-1}(N_i)$ is zero.  We can therefore apply
  Proposition~\ref{pr-abstract-gkr-nonnilpirrep} to obtain the
  conclusion.
\end{proof}

\begin{remark}
  (1) In our applications, the desired goal will be \ref{zerocoinv}, and
  achieving this will be reduced to excluding~\ref{goursatexists}. In
  practice, we will however only show that this second possibility is
  restricted to ``diagonal'' situations.  Outside of these special
  cases, the Riemann Hypothesis (Theorem~\ref{th-rh}) will imply
  square-root cancellation when the sheaves corresponding to the
  representations $\rho_i$ have weights~$\leq 0$.

  (2) In contrast to the ``classical'' theory of the
  Goursat--Kolchin--Ribet criterion, as in~\cite{sumproducts}, note that
  we \emph{do not} obtain here a ``nice'' formula for the main term in
  the special diagonal cases. Thus the result is mostly useful when one
  does not wish to extract a precise main term from the analysis of sums
  of trace functions.
\end{remark}

\subsection{The key proposition}

In this section, we state the crucial proposition for the proof of our
main result. It will be used to prove that the first statement of
Proposition~\ref{pr-abstract-gkr-nonnilp} holds ``in most cases'' in
certain families of sheaves. The proof of this proposition will be found
in the following sections.

We consider a finite field~$k$ with an algebraic closure~$\ov{k}$, a
prime~$\ell$ invertible in~$k$, and a non-zero integer~$c$.

Given $(r,s)\in \ov{k}\times\ov{k}^{\times}$, we define morphisms
$\gamma_{r,s}$ and~$\gamma_{r,s,c}$ from~$\Aa^1_{\ov{k}}$
to~$\Aa^1_{\ov{k}}$ by
\[
  \gamma_{r,s}(v)=s(v+r),\quad\quad \gamma_{r,s,c}(v)=s(v+r)^c.
\]

Note that $\gamma_{r,s}$ is an element of the group $\Aff(\ov{k})$ of
affine transformations of~$\Aa^1_{\ov{k}}$.

For any $\ell$-adic sheaf~$\mcF$, we denote
$$
  \mcF_{r,s}=\gamma_{r,s}^*\mcF,\quad\quad
  \mcF_{r,s,c}=\gamma_{r,s,c}^*\mcF,$$
  $$\mcF_c=\gamma^*_{0,1,c}\mcF=[x\ra x^c]^*\mcF
$$
which are again $\ell$-adic sheaves on~$\Aa^1_{\ov{k}}$. These satisfy
obvious properties which we will use below, such as
\[
  (\mcF_1\oplus\mcF_2)_{r,s,c}= (\mcF_1)_{r,s,c}\oplus (\mcF_2)_{r,s,c},
\]
and similarly for tensor products.

Let $(r,s)\in \ov{k}\times\ov{k}^{\times}$. Let~$\sigma\in\ov{k}$ be
any element such that $\sigma^c=s$. Then we have
\begin{equation}\label{eq-rsc}
  \mcF_{r,s,c}=[v\mapsto s(v+r)^c]^*\mcF= [v\mapsto
  \sigma(v+r)]^*[v\mapsto v^c]^*\mcF=\gamma_{r,\sigma}^*\mcF_{c}.
\end{equation}

In particular, the arithmetic (resp. geometric) monodromy group
of~$\mcF_{r,s,c}$ is (up to isomorphism) independent of~$(r,s)$, and
coincides with the arithmetic (resp. geometric) monodromy group
of~$\mcF_{c}$. Its complexity is bounded
polynomially in terms of~$c$ and the complexity of~$\mcF$
by~\cite[Th.\,6.8]{qst} (or by elementary properties of the
Fouvry--Kowalski--Michel complexity, in this case).

\begin{definition}
  Let $G$ be the geometric monodromy group of $\mcF_c$. The \emph{core
    subgroup} of $G$ is the subgroup
  \begin{equation}
    \label{coresubgroup}
    N=\begin{cases}
      G^0&\text{ if $G$ is infinite},\\
      \bfP(G)&\text{ if $G$ is finite}.
    \end{cases}	
  \end{equation}
\end{definition}

\begin{remark}\label{remcharac}
  The core subgroup is in all cases a characteristic subgroup of $G$. In
  particular, $N$ is a normal subgroup of the arithmetic monodromy group
  $G^a$ of $\mcF_c$, since~$G$ is normal in~$G^a$.
\end{remark}	

\begin{remark}\label{lemsame}
  Since the group $G$ can be identified with a normal subgroup of the
  geometric monodromy~$G_1$ group of~$\mcF$, with cyclic quotient of
  order dividing~$c$, the core subgroup $N$ is also the core subgroup
  of~$G_1$.
\end{remark}

\begin{proposition}\label{pr-good-gkrlemma}
  Let~$k$ be a finite field. Let~$c$ be a non-zero integer invertible
  in~$k$. Let~$\mcF$ be a \good $\ell$-adic sheaf over~$k$. Denote
  by~$N$ the core subgroup of $\mcF_c$.

  If the characteristic~$q$ of~$k$ is large enough, in terms of~$c$ and
  the complexity of~$\mcF_c$, then there exists a finite symmetric set
 $\mcT_c$ of~$\Aff(\ov{k})$ such that the following properties
  hold:
  \begin{enumth}
  \item the cardinality of $\mcT_c$ is bounded in terms of~$c$ and the complexity
    of~$\mcF$;
  \item for all~$(r_1,s_1)$ and~$(r_2,s_2)$
    in~$\ov{k}\times\ov{k}^{\times}$, the geometric monodromy group of
    the sheaf
    \[
      \mcF_{r_1,s_1,c} \oplus \mcF_{r_2,s_2,c}
    \]
    contains $N \times N$ except possibly if there exist $c$-th
    roots~$\sigma_1$ and~$\sigma_2$ of~$s_1$ and~$s_2$ respectively
    such that
    \[
      \gamma_{r_1,\sigma_1}\gamma_{r_2,\sigma_2}^{-1}\in \mcT_c.
    \]
  \end{enumth}    
\end{proposition}

The proof will be given in the next sections, starting with the case
when~$G$ is infinite. Since this case is probably the most relevant to
most applications, and since the proof in the finite case is
significantly more involved but based on similar ideas, it may be best
to skip Section~\ref{gkrlemmafinite} in a first reading.  In addition,
it may be convenient to read the proofs first by assuming that the
group~$G$ is simple.

% We will distinguish the cases where~$G$ is finite or infinite.  The next
% subsection deals with the infinite case (which is the one relevant for
% most current applications). The finite case is more involved; we devote
% the whole of Section \ref{gkrlemmafinite} to it, but readers interested
% primarily in the infinite case can safely skip that section.

\section{Proof of Proposition \ref{pr-good-gkrlemma} when $G$ is
  infinite}\label{gkrlemmainfinite}

We assume that~$G$ is infinite, and recall that $N=G^0$ in this
case. By Deligne's work~\cite{WeilII}*{Cor. (1.3.9)}), the group~$N$
is a semisimple algebraic group, and in particular is perfect. It is
also a characteristic subgroup of~$G$.

We begin with some preparations, culminating with the definition of
what will be the set~$\mcT_c$ claimed to exist in the proposition.

We denote by $\Omega$ the set of simple normal factors of~$N$, so that
we have an isogeny
\[
  \prod_{M\in \Omega}M\surjecte N
\]
whose kernel is a finite central subgroup.

We denote by $\mfn$ (resp. $\mfm_M$) the Lie algebra of~$N$ (resp. of an
element~$M$ of~$\Omega$). The ideals of~$\mfn$ are parameterized by
subsets~$C$ of~$\Omega$ by the map
\[
  C\mapsto \mfn_C=\bigoplus_{M\in C}\mfn_M.
\]

% Let
% $$\mfn=\mfg:=\Lie(N)=\Lie(N_1)\oplus \cdots\oplus\Lie(N_{c_G})=:\mfn_1\oplus\cdots\oplus \mfn_{c_G}$$
% denote the Lie algebra; the ideals of $\mfn$ are indexed by the subsets of $C\subset\Omega(G)$ and given by
% $$\mfn_C=\bigoplus_{c\in C}\mfn_c.$$

The group~$G$ acts on $\mfn$ by the adjoint representation. Since any
simple factor~$M$ acts trivially on all Lie algebras $\mfm_{Q}$
for~$Q\not=M$, the decomposition
\begin{equation}
  \label{Liedecomp}
  \mfn=  \bigoplus_{M\in\Omega} \mfn_M
\end{equation}
is an isomorphism of $N$-modules.

% since $N_d$ acts trivially on $\mfn_c$ for $d\not=c$, we have an
% isomorphism of $N$-modules
% \begin{equation}
% 	\label{Liedecomp}
% 	\mfn\simeq  \bigoplus_{c\in\Omega(G)} \mfn_c.\end{equation}

The group $G$ also acts on the set~$\Omega$ by conjugation, and the
decomposition of $\Omega$ in $G$-orbits parameterizes a decomposition
\begin{equation}
  \label{irredG}
  \mfn\simeq \bigoplus_{I\in \Omega/G}\mfn_I
\end{equation}
of $\mfn$, viewed as a representation of~$G$, in irreducible
subrepresentations (where each element of $\Omega/G$ is viewed as a
subset of~$\Omega$).

% $G$-action permutes the various $\Omega(G)$ and yields an orbit
% partition
% $$\Omega(G)=\bigsqcup_{I\in \mcO}I;$$
% for $I\in \mcO$ the ideal
% $$\mfn_I:=\bigoplus_{c\in I}\mfn_c$$
% is an irreducible $G$-module and $\mfn$ decomposes into these irreducibles
% \begin{equation}
% 	\label{irredG}
% 	\mfn\simeq \bigoplus_{I\in \mcO}\mfn_I.
% \end{equation}

As we have already observed, the arithmetic monodromy group~$G^a$ of
$\mcF_c$ normalizes $N$ and therefore acts on $\mfn$ by
conjugation. We denote by $\mcN$ the corresponding $\ell$-adic sheaf
over $\Aa^1_{k}$. The $G$-submodules $\mfn_I$ for $I\in \Omega/G$ are
not necessarily stable under the action of $G^a$, but become so up to
replacing $k$ by a finite extension $k'/k$, of degree bounded in terms
of $|\Omega|$, so that they also define sheaves $\mcN_I$ and
\eqref{irredG} induces an arithmetic isomorphism
\begin{equation}
  \label{irredGa}
  \mcN\simeq \bigoplus_{I\in \Omega/G}\mcN_I
\end{equation}
of $\ell$-adic sheaves on $\Aa^1_{k'}$. 
 
By construction, each sheaf $\mcN_I$ is geometrically irreducible, and
has generic rank $r_I=\dim(\mfn_I)\geq 2$, since $\dim \mfn_M\geq 2$
for all~$M\in \Omega$ (these being simple groups).  Moreover, each
sheaf $\mcN_I$ is a summand of $\mcN$, which is itself (isomorphic to)
a subsheaf of $\mcF_c\otimes \mcF_c^\vee$, hence its complexity is
bounded in terms of the complexity $c(\mcF_c)$ of~$\mcF_c$
(see~\cite[Th.\,6.8,\,Prop.\,6.14]{qst}).

For $I\in \Omega/G$, we denote
\[
  T_{I}=\Bigl\{\gamma=\begin{pmatrix} a& b\\0&d
  \end{pmatrix}
  \in\PGL_2(\ov{k})\,\mid\, \gamma^*(\mcN_I)\simeq
  \mcN_I\Bigr\}\subset \Aff_2(\ov k),
\]
where $\simeq$ denotes geometric isomorphism.  By Proposition
\ref{pr-fkm2}, we see that if $q$ is sufficiently large in terms of
the complexity of~$\mcF_c$ then the group $T_I$ is a finite cyclic
group of size bounded in terms of $c(\mcF_c)$.  Indeed the first
possible conclusion of Proposition~\ref{pr-fkm2} does not hold
since~$\mcN_I$ has generic rank~$\geq 2$, and the third cannot be
valid if $q$ is sufficiently large in terms of the complexity
of~$\mcN_I$. %% which is bounded in terms of $c(\mcF)$ and $c$.

Furthermore, for $I$ and $J$ in $\Omega/G$, we define
\[
  T_{I,J}=\{\gamma\in\Aff_2(\ov k),\ \gamma^*\mcN_I\simeq \mcN_J\}.
\]

We have $T_{I,I}=T_I$, and if $I\not=J$ then $T_{I,J}$ is either empty
or a right (resp. left) coset of $T_I$ (resp. $T_J$). Moreover, we
have the equality $T_{JI}=T_{IJ}^{-1}$. In particular, the
set
\begin{equation}
  \label{mcTdef}
  \mcT_c=\bigcup_{I,J\in \Omega/G}T_{IJ}\subset \Aff_2(\ov k)
\end{equation}
is a finite symmetric subset of~$\Aff_2(\ov k)$ of size bounded only
in terms of $c(\mcF)$ and~$c$.

% , and the same holds for the
% image~$\mcD_c=\mathrm{sl}^c(\mcT_c)$ of~$\mcT_c$ under the morphism
%  % We then define
%  % \begin{equation}
%  % 	\label{mcDdef}
%  % 	\mcD_c:=\{\sigma^c,\ \gamma_{r,\sigma}\in\mcT_c\}\subset\ovFqt
%  % \end{equation}
%  % the image of $\mcT_c$ under the morphism
% \[
%   \Aff_2(\ovFq)\to \ovFq^{\times}
% \]
% mapping $\begin{pmatrix}a&b\\0&1
% \end{pmatrix}$ to~$a^c$.

% From our previous discussion, the set $\mcD_c$ is symmetric and of
% size bounded (polynomially) in terms of $c$ and $c(\mcF)$.
  
We are now ready to start the proof of
Proposition~\ref{pr-good-gkrlemma} in our case, by showing that the
set~$\mcT_c$ has the properties claimed in that statement.

Let~$(r_1,s_1)$ and~$(r_2,s_2)$ be elements
of~$\ov{k}\times\ov{k}^{\times}$. For $j=1,2$, we fix elements
$\sigma_j$ such that $\sigma_j^c=s_j$ and we denote
$\gamma_j=\gamma_{r_j,\sigma_j}$, so that
$\mcF_{r_j,s_j,c}=\gamma_j^*\mcF_c$ by \eqref{eq-rsc}.

Let $K$ be the geometric monodromy group of
$\mcF_{r_1,s_1,c}\oplus \mcF_{r_2,s_2,c}$ and assume that $K$ does not
contain $N\times N$ (recall that we can identify the connected
component of the geometric monodromy group of each summand with~$G$,
so this assumption makes sense).

The group~$K$ is a subgroup of $G\times G$ which projects surjectively
onto each factor.
% Assume that $K$ does not contain $N\times N$ (note that the
% connected component of the geometric monodromy group of each summand
% coincides with that of~$\mcF_c$, which is~$N$, so this potential
% inclusion makes sense) and
Let $(H_1,H_2,\vphi)$ be a Goursat datum for $K\subset G\times G$; by
assumption, one of $H_1$ or $H_2$ at least does not contain $N$.
   
Let~$j=1$ or $j=2$. Since, since $H_j$ is normal in $G$, there exist a
subset $O_j$ of~$\Omega$ such that the Lie algebra of~$H_j$ is equal
to $\mfh_j=\mfn_{O_j}$. One at least of the subsets~$O_j$ is different
from $\Omega$, since otherwise $N$ would be contained in~$H_j$. In
view of the isomorphism $\vphi\colon G/H_1\to G/H_2$, which implies
that $\Lie(G)/\mfh_1$ is isomorphic to $\Lie(G)/\mfh_2$, we see that in fact
$O_1$ and~$O_2$ are both proper subsets of~$\Omega$.

% By our assumption, at least one of the two sets
% $O_j, j=1,2$ is not the full set $\mcO$ (otherwise $N$ would be
% contained in the corresponding $H_j$); this implies that neither $O_1$
% nor $O_2$ is equal to $\mcO$ as the isomorphism
% $\vphi:G/H_1\simeq G/H_2$ implies $\mfg/\mfh_1\simeq \mfg/\mfh_2$.

Let $Q_j$ denote the complement of~$O_j$ in~$\Omega$, which is
therefore not empty.
% :=\mcO-O_j$ denote the (non-empty) complements.
The sheaf corresponding to the adjoint action of the geometric
monodromy group of $\gamma_j^*\mcF_c$ on
$\Lie(G/H_j)=\mfn/\mfh_j$ is $\gamma_j^*\bigl(\oplus_{I\in
  Q_j}\mcN_I\bigr)$. The isomorphism
$\vphi$ therefore induces a geometric isomorphism
\[
  \gamma_1^*\Bigl(\bigoplus_{I\in Q_1}\mcN_I\Bigr)\to
  \gamma_2^*\Bigl(\bigoplus_{J\in Q_2}\mcN_J\Bigr).
\]

Since the sheaves $\mcN_I$ are geometrically irreducible, and
the~$\gamma_j$'s are automorphisms, this implies the existence of a
bijection
\[
  \tilde\vphi\colon Q_1\to Q_2,
\]
such that, for $I\in Q_1$, the sheaf $\gamma_1^*\mcN_I$ is
geometrically isomorphic to $\gamma_2^*\mcN_{\tilde{\vphi}(I)}$.
% geometric isomorphisms of sheaves,
% $$\vphi_I:\gamma_1^*\mcN_I\simeq \gamma_2^*\mcN_{\tilde\vphi(I)}.$$
Since~$Q_1$ is not empty, we can apply this to at least
one~$I$, and pulling back by
$\gamma_2^{-1}$, this implies that
$\gamma_1\gamma_2^{-1}=\gamma_{r_1,\sigma_1}\gamma_{r_2,\sigma_2}^{-1}$
is in $\mcT_c$.
% more precisely we have
% $$\gamma_{r_1,\sigma_1}\gamma_{r_2,\sigma_2}^{-1}\in\bigcap_{I\in O_1^c}T_{I,\tilde\vphi(I)}.$$
This conclusion is what we desired in the case $G$ infinite.

\section{Proof of Proposition \ref{pr-good-gkrlemma} when $G$ is
  finite}
\label{gkrlemmafinite}

The ideas of the proof for~$G$ finite are inspired by those of the
previous section, but are more involved. 

\subsection{Background material on finite groups}

We summarize here some notation and standard material concerning
finite groups. Our reference is the book of
Isaacs~\cite{isaacs}. Throughout this section, we fix a finite group~$G$.

%% \subsubsection{The solvable radical}

We denote by $\bfS(G)$ (or just $\bfS$ if $G$ is understood) the
solvable radical of $G$, which is the unique largest normal solvable
subgroup of~$G$, or equivalently the subgroup generated by all
solvable normal subgroups. This is a characteristic subgroup of $G$.

We recall the following basic lemma:

\begin{lemma}\label{SRlemma}
  Let $G$ be a finite group with solvable radical $\bfS$.
  \begin{enumth}
  \item The solvable radical of $G/\bfS$ is trivial. In particular,
    the center of $G/\bfS$ is trivial and if $G$ is not solvable, then
    the quotient group $G/\bfS$ is not solvable.
  \item For any normal subgroup $H\triangleleft G$, the preimage in
    $G$ of the solvable radical $\bfS(G/H)$ contains $\bfS$.
\end{enumth}
\end{lemma}

The \emph{perfect radical} of $G$, denoted $\bfP(G)$ (or $\bfP$ if
there is no ambiguity), is the subgroup of~$G$ generated by all the
perfect subgroups of $G$. Equivalently, it is the largest perfect
subgroup of $G$. This is a characteristic subgroup of $G$, and in
particular it is normal in $G$.

The subgroup~$\bfP$ admits a description in terms of the derived
series. Let $G^{(0)}=G$ and define~$G^{(n)}$ by induction
for~$n\geq 0$ by the formula $G^{(n+1)}=[G^{(n)},G^{(n)}]$.  The
groups $G^{(n)}$ for $n\geq 1$ form a decreasing sequence of
characteristic subgroups of $G$; since $G$ is finite, this sequence
eventually stabilizes and we denote by $G^{(\infty)}$ its ``limit'',
which is also the intersection of all $G^{(n)}$.

\begin{lemma}\label{PRlemma}
  Let $G$ be a finite group with perfect radical $\bfP$. The following
  hold:
  \begin{enumth}
  \item We have $G^{(\infty)}=\bfP$.
  \item We have $\bfP=1$ if and only if $G$ is solvable.
  \item For any normal subgroup $H$ of~$G$, we have
    $\bfP(G)H/H=\bfP(G/H)$.
  \end{enumth}
\end{lemma}

%\begin{proof}
 % Suppose $G^{(n)}=G^{(n+1)}=G^{(\infty)}$, we then have $G^{(n+1)}=[G^{(n)},G^{(n)}]=G^{(n)}$ so $G^{(n)}$ is perfect. Since $G^{(1)}$ and all $G^{(n)}$ are characteristic, $G^{(\infty)}$ is also characteristic hence normal so is contained in $\bfP$.
%
%We have for any $n\geq 1$
%$\bfP=\bfP^{(n)}\subset G^{(n)}$ and taking $n$ large enough we obtain that $\bfP\subset G^{(\infty)}$.
%
%The group $G$ is solvable if and only if $G^{(n)}=1$ for some $n\geq 0$ and therefore if and only if $\bfP=G^{(\infty)}=1$.
%
%
%Given $H\triangleleft G$, we have for any $n\geq 0$
%$$G^{(n)}H/H=(G/H)^{(n)}$$  hence the equality 
%$$G^{(\infty)}H/H=(G/H)^{(\infty)}$$ by taking $n$ large enough.
%\end{proof}

We will also need a useful property of the perfect radical, which we
isolate since it may be satisfied by other subgroups, in which case
they can play a similar role in our arguments when~$G$ is not perfect.

\begin{definition}\label{GAdef}
  Let $G$ be a finite group and $E\subset G$ a subgroup. We say that
  $E$ is Goursat-adapted (relative to $G$) if for any normal subgroups
  $H_1$ and $H_2$ of~$G$, and any isomorphism
  $\vphi\colon G/H_1\to G/H_2$ we have
  \begin{equation}
    \label{EGoursatadapted}
    \vphi(EH_1/H_1)=EH_2/H_2.
  \end{equation}
\end{definition}

\begin{remark}
  Taking $H_1=H_2=1$ with the unique isomorphism, we see that a
  Goursat-adapted subgroup $E$ is necessarily a characteristic
  subgroup of $G$.
\end{remark}

\begin{lemma}\label{RGoursatadaptedlemma}
  The perfect radical of a finite group~$G$ is Goursat-adapted with
  respect to~$G$.
\end{lemma}

\begin{proof}
  Let~$\bfP$ denote the perfect radical.  Let $H_1$ and~$H_2$ be
  normal subgroups of~$G$, and let $\vphi\colon G/H_1\to G/H_2$ be an
  isomorphism. For $i=1,2$, we have $\bfP H_i/H_i=\bfP(G/H_i)$ by
  Lemma \ref{PRlemma}. The definition of the perfect radical implies
  that $\vphi(\bfP(G/H_1))\subset \bfP(G/H_2)$ and switching indices
  we have $\vphi^{-1}(\bfP(G/H_2))\subset \bfP(G/H_1)$ which gives the
  equality.
\end{proof}

\subsection{Jordan's Theorem}

The set $\mcT_c$ will ultimately be constructed from the irreducible
representations of the quotient $G/\bfS$. In order to control its size
in terms of the complexity of the sheaf $\mcF$, we will use the
following well-known result of Jordan (see,
e.g.,~\cite{cr}*{Th.\,36.13}).

\begin{theorem}[Jordan]\label{thmjordan}
  There exists a function $\Jor:\Zz_{\geq 1}\ra\Zz_{\geq1}$ such that
  for any field $K$ of characteristic $0$ and any finite subgroup
  $G\subset\GL_r(K)$, there exists an abelian normal subgroup
  $A\triangleleft G$ with $[G:A]\leq \Jor(r)$.
\end{theorem}

It is known, by work of Collins, that one can take $\Jor(n)=(n+1)!$,
at least if $n\geq 71$ (see~\cite{collins}).

\begin{lemma}\label{lemcontrol}
  Let $K$ be an algebraically closed field of characteristic $0$ and
  $G\subset \GL_r(K)$ a finite group with solvable radical $\bfS$. The bounds
 \begin{equation}
   \label{SAbound}
   |\Aut(G/\bfS)|\leq \Jor(r)!,\quad   |G/\bfS|\leq \Jor(r)
 \end{equation}
 hold.  In particular, the group $G/\bfS$ has $\leq \Jor(r)$
 irreducible representations up to isomorphism, and these have
 dimension $\leq \Jor(r)^{1/2}$.
\end{lemma}

\begin{proof}
  Let $A$ be the abelian normal subgroup of $G$ of
  index~$\leq \Jor(r)$ given by Theorem \ref{thmjordan}. Since
  $A\subset \bfS$ by definition of~$\bfS$, we have
  \[
    |G/\bfS|\leq |G/A|\leq \Jor(r),
  \]
  and we deduce immediately the bound for $|\Aut(G/\bfS)|$. The last
  statement also follows by basic representation theory.
  % because the number of irreducible representations of a finite group
  % is its number of conjugacy classes, and the sum of the squares of
  % their dimensions is its order.
\end{proof}

\subsection{Construction of the diagonal set $\mcT_c$}\label{secdiagTc}

Let $\mcF$ be a \good sheaf with finite geometric monodromy group, and
let $c\in\Zz-\{0\}$ be a non-zero integer. We denote by $G$ the
geometric monodromy group of $\mcF_c$ and by $\bfS$ and $N=\bfP$ its
solvable and perfect radicals, respectively. Recall that $\bfP$ is
also the perfect radical of the geometric monodromy group of $\mcF$
(see Remark~\ref{lemsame}). In particular, $\bfP\not=1$, so that $G$
is not solvable and $G/\bfS$ is not solvable either. We denote by
\[
  \pi\colon G\ra G/\bfS
\]
the canonical projection.

Recall that our objective is to define a symmetric set
\begin{equation}
  \label{mcTdefintro}
  \mcT_{c}\subset \Aff(\ov k)
\end{equation}
of size bounded in terms of~$c$ and the complexity of~$\mcF$ which
satisfies the the conclusion of Proposition \ref{pr-good-gkrlemma}, at
least when $q$ is sufficiently large in terms of~$c(\mcF)$.

Let $\Gamma^g$ (resp. $\Gamma^a$) be the Galois group of the function
field $\ov{k}(X)$ (resp. of $k(X)$). For any $\gamma\in\Aff(\ov{k})$,
the sheaf $\gamma^*\mcF_c$ determines a surjective group homomorphism
\[
  \rho_{\gamma}:\Gamma^g\ra G
\]
(well-defined up to conjugation by an element of $G$). For simplicity,
we denote $\rho=\rho_{\Id}$, and for
$(r,\sigma)\in\ov{k}\times\ov{k}^{\times}$, we define
$\rho_{r,\sigma}=\rho_{\gamma_{r,\sigma}}$.

We denote by~$\mcW$ the set of isomorphism classes of
finite-dimensional $\ell$-adic irreducible representations of~$G/\bfS$
of dimension~$\geq 2$.  This set is not empty since $G/\bfS$ is not
abelian. Moroever, since $\rk(\mcF_c)=\rk(\mcF)\leq c(\mcF)$, Lemma
\ref{lemcontrol} implies that
\begin{equation}\label{mcWbound}
  |\mcW|\ll_{c(\mcF)}1\text{ and } \dim(\varpi)\ll_{c(\mcF)}1\ \text{ for }
  \varpi\in\mcW.
\end{equation}

For $\varpi\in\mcW$ acting on a vector space~$W$, we denote by
$\mcF_{\varpi}$ the $\ell$-adic sheaf corresponding to the composition
\[
  \varpi\circ\pi\circ\rho\colon \Gamma^a\ra G\ra G/\bfS\ra\GL(W).
\]

The sheaf $\mcF_{\varpi}$ is geometrically irreducible of rank
$\dim(\varpi)\geq 2$. Moreover, for any $\gamma\in\Aff(\ov{k})$, the
sheaf $\gamma^*\mcF_{\varpi}$ corresponds to the representation
$\varpi\circ\pi\circ\rho_\gamma$.

For $\varpi$ and $\varpi'\in\mcW$, we define
\[
  T_{\varpi,\varpi'}=\{\gamma\in\Aff(\ov k)\,\mid\,
  \gamma^*\mcF_\varpi\simeq \mcF_{\varpi'}\},\quad
  T_\varpi=T_{\varpi,\varpi},
\]
and we finally define
\begin{equation}\label{mcTdef2}
  \mcT_c=\bigcup_{\varpi,\varpi'\in\mcW}T_{\varpi,\varpi'}\subset\Aff(\ov k).
\end{equation}

\begin{proposition}\label{Bounded}
  If $q$ is large enough in terms of $c$ and $c(\mcF)$, then the set
  $\mcT_c$ is symmetric, contains the identity, and is a finite set of
  size bounded only in terms of $c(\mcF)$ and $c$.
\end{proposition}

\begin{proof}
  For any $\varpi\in\mcW$, the set $T_\varpi$ is the stabilizer in
  $\Aff(\ov k)$ of the geometric isomorphism class of $\mcF_\varpi$.
  By construction, we have
  \[
    T_{\varpi,\varpi'}=T_{\varpi',\varpi}^{-1}
  \]
  for any $\varpi$ and $\varpi'\in\mcW$. Moreover $T_{\varpi,\varpi'}$
  is either empty or is a right coset of $T_\varpi$ (resp. a left
  coset of $T_{\varpi'}$). The set $\mcT_c$ is therefore symmetric and
  contains the identity. 
  
  In view of \eqref{mcWbound}, it is enough to bound the size of any
  $T_{\varpi}$ in order to deduce that~$\mcT_c$ has size bounded in
  terms of~$c$ and the complexity of~$\mcF$.
  % by It is then sufficient
  % to show that for $q$ sufficiently large (depending on $c(\mcF_c)$),
  % $T_\varpi$ is finite and that
  % $$\max_{\varpi\in \mcW} |T_\varpi|\ll_{c,c(\mcF)}1.$$
  
  Let $\varpi$ be an element of $\mcW$. The arithmetic monodromy group
  $G^a$ of $\mcF_c$ normalizes $G$ as well as $\bfS$ (since $\bfS$ is
  a characteristic subgroup). Thus, it acts on $G/\bfS$ by
  conjugation, and up to replacing $k$ by a finite extension of degree
  $\leq |\Aut(G/\bfS)|\ll_{c(\mcF)} 1$ (by Lemma \ref{lemcontrol}), we
  may assume that the sheaf $\mcF_\varpi$ is defined over~$k$.

  The complexity of $\mcF_\varpi$ is bounded in terms of that
  of~$\mcF_c$. Indeed, if we use the Fouvry--Kowalski--Michel
  complexity, then we observe that the sheaf $\mcF_\varpi$ is lisse
  wherever $\mcF_c$ is lisse, has the same rank, and that for any
  singularity $x$, the breaks of $\mcF_\varpi$ at~$x$ are bounded by
  the maximal break of $\mcF_c$ so that
  \[
    \mathrm{Swan}_x(\mcF_\varpi)\leq
    \dim(\varpi)\,\mathrm{Swan}_x(\mcF_c).
  \]

  We apply Proposition~\ref{pr-fkm2} to the geometrically irreducible
  sheaf $\mcF_\varpi$. Since its generic rank is $\dim(\varpi)\geq 2$,
  the first possible conclusion of loc. cit. does not hold for it, and
  the third is also impossible if $q$ is large enough in terms of the
  complexity of $\mcF_\varpi$, hence in terms of $c$ and
  $c(\mcF)$. Therefore, for such~$q$, the set $T_\varpi$ is a finite
  cyclic group of size bounded in terms of $c$ and $c(\mcF)$.
\end{proof}

% \begin{remark}
%   Notice that $\mcF_c$ might not be geometrically irreducible but the
%   construction of $\mcT_c$ and Proposition \ref{Bounded} do not make
%   use of this.
% \end{remark}

% We then make the same definition as in \eqref{mcDdef} and set
% $$
% \mcD_c=\{\sigma^c:\ \gamma_{r,\sigma}\in\mcT_c\}\subset\ovFqt
% $$
% be the image of $\mcT_c$ under the morphism
% $$\slo^c:\gamma_{r,\sigma}\in\Aff(\ovFq)\mapsto \sigma^c\in\ovFq^\times.$$

% From our previous discussion, the set $\mcD_c$ is symmetric and of size bounded in terms of $c$ and $c(\mcF)$.

\subsection{Conclusion}\label{pr-good-finite-conclusion}

We recall (Lemma \ref{RGoursatadaptedlemma}) that $\bfP$ is
Goursat-adapted relative to $G$.

For $i=1,2$ let $\sigma_i$ be such that $\sigma_i^c=s_i$, let
$\gamma_i=\gamma_{r_i,\sigma_i}$ and
$\rho_i=\rho_{r_i,\sigma_i}:\Gamma^g\ra G$.

Let $K$ be the geometric monodromy group of
\[
  \mcF_{r_1,s_1,c}\oplus \mcF_{r_2,s_2,c}=\gamma^*_{1}\mcF_c\oplus
  \gamma^*_{2}\mcF_c;
\]
this is a subgroup of $G\times G$ projecting surjectively onto each
factor.

We assume that $K$ does not contain $\bfP\times \bfP$. Let then
$(H_1,H_2,\vphi)$ be a Goursat datum for $K\subset G\times G$. At
least one of $H_1$ or $H_2$ does not contain $\bfP$ and since~$\bfP$
is Goursat-adapted, it follows from the definition (see
\eqref{EGoursatadapted}) that neither of them contains $\bfP$.

It follows that for $i=1,2$, the group $\bfP(G/H_i)=\bfP H_i/H_i$ is
non-trivial, so $G/H_i$ is not solvable. Let $\bfS_i=\bfS(G/H_i)$ be
the solvable radical of $G/H_i$ and let $L_i\triangleleft G$ be its
preimage in $G$. By Lemma \ref{SRlemma}, we have $\bfS\subset L_i$ and
the group
\[
  G/L_i\simeq (G/H_i)/\bfS_i
\]
is not abelian.

Since the solvable radical is unique, we have $\vphi(\bfS_1)=\bfS_2$,
and $\vphi$ induces an isomorphism
\[
  \ov\vphi\colon G/L_1\to G/L_2.
\]

By definition of the Goursat datum, we have
$\vphi(\rho_1(g)H_1)=\rho_2(g)H_2$ for all $g\in\Gamma^g$, and
therefore
\begin{equation}\label{Goursatequiv}
  \ov\vphi(\rho_1(g)L_1)=\rho_2(g)L_2\text{ for all } g\in \Gamma^g.
\end{equation}

Let $\varpi_2$ be an irreducible $\bQl$-representation of $G/L_2$ of
dimension $\geq 2$, which exists since $G/L_2$ is not abelian, and let
$\varpi_1=\varpi_2\circ\ov\vphi$. This is an irreducible
representation of $G/L_1$ of the same dimension. Since
$\bfS\subset L_i$, we may view $\varpi_i$ as an irreducible
representation of $G/\bfS$, so that $\varpi_i\in \mcW$. By
\eqref{Goursatequiv} we have
\[
  \varpi_1\circ\pi\circ\rho_1=\varpi_2\circ\pi\circ\rho_2
\]
as representations of $\Gamma^g$. In other words, we have a geometric
isomorphism
\[
  \gamma_1^*\mcF_{\varpi_1}\simeq\gamma_2^*\mcF_{\varpi_2}
\]
and, pulling back by $\gamma_2^{-1}$, an isomorphism
$(\gamma_1\gamma_2^{-1})^*\mcF_{\varpi_1}\simeq\mcF_{\varpi_2}$, i.e.
\[
  \gamma_1\gamma_2^{-1}\in T_{\varpi_1,\varpi_2}\subset\mcT_c.
\]

% Finally $\slo(\gamma_1\gamma_2^{-1})^c=s_1/s_2$, so
% $s_1/s_2\in\mcD_c$.

This concludes the proof.

 \section{Diagonal estimates}\label{diagonalestimates}

Combining Proposition~\ref{pr-good-gkrlemma} and
Proposition~\ref{pr-abstract-gkr-nonnilp}, we will ultimately seek to obtain
estimates for sums of products of pullbacks of a \good sheaf $\mcF$ by the
transformations $\gamma_{r,s,c}$. We prove here, in different
situations, that ``most'' of the sheaves which arise in sums of products
of a certain type have trivial topmost cohomology group, a property
which translates into bounds for the corresponding sums after applying the
Riemann Hypothesis over finite fields.

We will use repeatedly the following simple combinatorial lemma. By
``graph'', we mean an undirected graph without multiple edges, but where
loops are allowed.

\begin{lemma}\label{lm-graph1}
  Let~$(V,E)$ be a finite graph, with maximal degree~$C\geq
  0$. Let~$m\geq 1$ be an integer. Denote
  \begin{multline*}
    \Delta_{m}(V,E)=\{ (v_1,\ldots,v_{2m})\in V^{2m}\,\mid\, \text{for
      each $i$, the vertex $v_i$}
    \\\text{ is connected to a vertex $v_{j}$
      with $j\not=i$} \}.
  \end{multline*}
  
  We have
  \[
    |\Delta_m(V,E)|\ll |V|^m,
  \]
  where the implied constant depends only on~$m$ and~$C$.
\end{lemma}

\begin{proof}
  We denote by $[2m]$ the set $\{1,\ldots,2m\}$.
  
  For any~$\uple{v}=(v_1,\ldots,v_{2m})\in\Delta_m(V,E)$, define a
  graph~$\Theta_{\uple{v}}$ with vertex set~$[2m]$ and edges
  joining~$i$ to~$j$ if and only if~$v_i$ is connected to~$v_j$ in
  $(V,E)$.

  Since the number of possible graphs~$\Theta_{\uple{v}}$ is bounded
  in terms of~$m$ and~$C$ only, we obtain the estimate
  \[
    |\Delta_m(V,E)|\ll \max_{\Theta}|\{\uple{v}\in V^{2m}\,\mid\,
    \Theta_{\uple{v}}=\Theta\}|,
  \]
  where the implied constant depends only on~$m$ and~$C$, and the
  maximum ranges over graphs~$\Theta$ which arise
  as~$\Theta_{\uple{w}}$ for some~$\uple{w}\in \Delta_m(V,E)$.

  Thus we are reduced to estimating the number of~$\uple{v}$
  where~$\Theta_{\uple{v}}$ is a fixed graph~$\Theta$.  Let~$\pi_0$ be
  the number of connected components of~$\Theta$. Since $\Theta$ is of
  the form~$\Theta_{\uple{w}}$ for~$\uple{w}\in\Delta_m(V,E)$, the
  definition of~$\Delta_m(V,E)$ implies that any vertex~$i$ of~$\Theta$
  is connected to another vertex, and therefore we obtain the estimate
  $\pi_0\leq m$.
  
  Let~$I\subset [2m]$ be one connected component of~$\Theta$,
  and~$i_0\in I$. If~$\uple{v}\in V^{2m}$ satisfies
  $\Theta_{\uple{v}}=\Theta$, then all~$v_i$ for~$i\in I$ are contained
  in the ball~$B(v_{i_0}, |I|)$ of radius~$|I|$ in~$(V,E)$
  around~$v_{i_0}$. This ball has $\leq C^{|I|}$ elements, so it follows
  that there are at most
  \[
    \sum_{v\in V}|B(v;|I|)|\leq C^{|I|}|V|
  \]
  possible choices of~$(v_i)_{i\in I}$ for~$\uple{v}$
  with~$\Theta_{\uple{v}}=\Theta$. Applying this to all connected
  components, there are at most
  \[
    C^{2m}|V|^{\pi_0}\leq C^{2m}|V|^m
  \]
  possible choices of~$\uple{v}$ with~$\Theta_{\uple{v}}=\Theta$.

  This concludes the proof.
\end{proof}

We now assume given the data $(k,\mcF,c)$ of the previous section.  For
$m\geq 1$, and for
$(\uple{r},\uple{s})\in k^{2m}\times (k^{\times})^{2m}$, we denote
\begin{equation}\label{eq-sheaf-mcf-rsc}
  \mcF_{\uple{r},\uple{s},c}=\bigotimes_{1\leq j\leq m}
  \mcF_{r_j,s_j,c}\otimes \mcF_{r_{j+m},s_{j+m},c}^{\vee}.
\end{equation}

The first diagonal statement is the following, which is enough to deal
with Type I sums.

\begin{proposition}\label{pr-diagonal1}
  Let~$\mcF$ be a \good sheaf. Let $m\geq 1$ be an integer. Let
  \[
    \Delta_{m,c}(\mcF;k)=\{(\uple{r},\uple{s})\in k^{2m}\times
    (k^{\times})^{2m}\,\mid\, H^2_c(\Aa^1_{\ov{k}},
    \mcF_{\uple{r},\uple{s},c})\not=0 \}.
  \]

  Then we have
  \[
    |\Delta_{m,c}(\mcF;k)|\ll |k|^{2m}
  \]
  where the implied constant depends on~$m$ and on $c$ and the
  complexity of~$\mcF$.
\end{proposition}

\begin{proof}
  Let~$G$ be the geometric monodromy group of $\mcF_c$ and
  $N\subset G$ be its core subgroup; recall that $N$ is either the
  connected component $G^0$ or the perfect radical $\bfP(G)$; in
  either case $N$ is the same subgroup of the geometric monodromy
  group of $\mcF$.
   
  Let~$\mcT_c\subset \Aff(\ov{k})$ denote the finite symmetric set
  provided by Proposition~\ref{pr-good-gkrlemma}.
  % (see\eqref{mcTdef} or \eqref{mcTdef2} and let
  % and let $\mcD_c=\slo(\mcT_c)^c$.

  Let $\rho_{r,s,c}$ denote the Galois representation associated
  to~$\mcF_{r,s,c}$; its geometric monodromy group is
  $G_{r,s,c}=G$. Let $G_{\uple{r},\uple{s},c}$ be the geometric
  monodromy group of the sheaf
  \[
    \bigoplus_{1\leq j\leq 2m} \mcF_{r_j,s_j,c},
  \]
  which we view as a subgroup of
  \[
    \prod_{1\leq j\leq 2m} G_{r_j,s_j,c}.
  \]

  By the co-invariant formula for the top-degree cohomology of a sheaf,
  the condition
  \[
    H^2_c(\Aa^1_{\ov{k}}, \mcF_{\uple{r},\uple{s},c})\not=0
  \]
  is equivalent to the non-vanishing of the co-invariant space of
  $G_{\uple{r},\uple{s},c}$ acting by the representation
  \[
    \bigotimes _{1\leq j\leq m} \rho_{r_j,s_j,c}\otimes
    \rho_{r_{j+m},s_{j+m},c}^{\vee}.
  \]

  Suppose this space is non-zero; we can then apply Corollary
  \ref{pr-abstract-gkr-nonnilpirrep} to $N_i=N$ for all~$i$ since $N$
  is non-nilpotent (notice that $G$ might not act irreducibly but $N$
  is a normal subgroup of the geometric monodromy group of $\mcF$
  which acts irreducibly).

  This shows that for each $i\leq 2m$, there exists $j\not=i$ such
  that the image of~$G_{\uple{r},\uple{s},c}$ in
  $G_{r_i,s_i,c}\times G_{r_j,s_j,c}$ does not contain $N\times N$. By
  Proposition~\ref{pr-good-gkrlemma}, this implies that
  $\gamma_{r_i,\sigma_i}\gamma_{r_j,\sigma_j}^{-1}\in \mcT_c$ where
  $\sigma_i^c=s_i$ and $\sigma_j^c=s_j$.

  Define a graph $\Gamma=(V,E)$ with vertices $k\times k^{\times}$ and
  edges joining $(r_1,s_1)$ to $(r_2,s_2)$ if there exist
  $\sigma_j\in\ov k^\times,\ j=1,2$ with $\sigma_j^c=s_j$ such that
  $\gamma_{r_1,\sigma_1}\gamma_{r_2,\sigma_2}^{-1}\in \mcT_c$ (in
  particular since $\Id\in\mcT_c$ each vertex has a loop). The
  previous results then show, with the notation of
  Lemma~\ref{lm-graph1}, that the inclusion
  \[
    \Delta_{m,c}(\mcF;k)\subset \Delta_m(\Gamma)
  \]
  holds, and since the degree of~$\Gamma$ is~$|\mcT_c|$, which is
  bounded in terms of~$c$ and the complexity of~$\mcF$, Lemma
  \ref{lm-graph1} concludes the proof.
\end{proof}

The second statement will occur in the study of Type II sums.

\begin{proposition}\label{pr-diagonal2}
  Let~$\mcF$ be a \ssgood\ sheaf and $c\in\Zz-\{0\}$. Let $\mcT_c$ the
  finite symmetric set provided by Proposition~\ref{pr-good-gkrlemma}.
  Define $\mcD_c\subset \ovFqt$ to be the image of $\mcT_c$ by the
  morphism
  \[
    \begin{pmatrix}s & r\\0&1
    \end{pmatrix}\mapsto s^c.
  \]
  %% and $$\mcD_c=\slo(\mcT_c)^c\subset \ovFqt.$$
  
  Given an integer $m\geq 1$,  let
  \[
    Y_{\mcF,c} =\{(s_1,s_2)\in\Gm\times \Gm\,\mid\, s_1/s_2\not\in
    \mcD_c\}
  \]
   and let
  \[
    \Delta'_{m,c}(\mcF;k)=\{(\uple{r},\uple{s}_1,\uple{s}_2)\in
    k^{2m}\times Y_{\mcF,c}(k)^{2m}\,\mid\, H^2_c(\Aa^1_{\ov{k}},
    \mcF_{\uple{r},\uple{s}_1,c} \otimes
    \mcF_{\uple{r},\uple{s}_2,c}^{\vee})\not=0 \}.
  \]

  Then we have
  \[
    |\Delta'_{m,c}(\mcF;k)|\ll |k|^{3m}
  \]
  where the implied constant depends on~$m$, $c$, and the complexity  of~$\mcF$.
\end{proposition}

\begin{proof}
  Since
 \[
   \mcF_{\uple{r},\uple{s}_1,c} \otimes
   \mcF_{\uple{r},\uple{s}_2,c}^{\vee}= \Bigl(\bigotimes_{1\leq j\leq
     m} \rho_{r_j,s_{1,j},c}\otimes
   \rho_{r_{j+m},s_{1,j+m},c}^{\vee}\Bigr) \otimes \Bigl(\bigotimes
   _{1\leq j\leq m} \rho_{r_j,s_{2,j},c}^{\vee}\otimes
   \rho_{r_{j+m},s_{2,j+m},c}\Bigr),
  \]
  we see as in the proof of the previous proposition (keeping its
  notation) that the non-vanishing condition is now equivalent with the
  non-vanishing of the coinvariant space of the representation
  $\displaystyle{\bigotimes_{1\leq j\leq 4m}\rho_{j}}$ where
  \[
    \rho_j=\begin{cases}
      \rho_{r_j,s_{1,j},c}&\text{ if } 1\leq j\leq m,\\
      \rho_{r_{j},s_{1,j},c}^{\vee}& \text{ if }
      m+1\leq j\leq 2m,\\
      \rho_{r_{j-2m},s_{2,j-2m},c}^{\vee}&\text{ if } 2m+1\leq j\leq 3m,\\
      \rho_{r_{j-2m},s_{2,j-2m},c}&\text{ if } 3m+1\leq j\leq 4m.
    \end{cases}
  \]

  We denote by $\delta_j$ the element~$\gamma_{r_j,\sigma_j}\in\Aff(\ov{k})$
  corresponding to each~$\rho_j$ (e.g., if $m+1\leq j\leq 2m$, $\delta_j=\gamma_{r_j,\sigma_{1,j}}$ for some choice of $\sigma_{1,j}$ satisfying $\sigma_{1,j}^c=s_{1,j}$).

  Suppose that the coinvariant space is non-zero. By
  Proposition~\ref{pr-good-gkrlemma}, there exists for each positive
  integer~$j\leq 4m$ an integer~$i\not=j$ with $1\leq i\leq 4m$ and
  $\delta_j\delta_i^{-1}\in \mcT_c$. In particular we have
  $s_i/s_j\in \mcD_c$.  A first application of Lemma \ref{lm-graph1}
  establishes that the number of possible pairs
  $(\uple{s}_1,\uple{s}_2)\in (k^{\times})^{4m}$ with this property is
  $O(|k|^{2m})$.

  Let~$\uple{s}=(\uple{s}_1,\uple{s}_2)\in Y_{\mcF,c}(k)^{2m}$. We will now
  bound (uniformly in terms of these tuples) the size of the set
  $\Delta_{\uple{s}_1,\uple{s}_2}$ of those~$\uple{r}\in k^{2m}$ such
  that $(\uple{r},\uple{s}_1,\uple{s}_2)\in \Delta'_{m,c}(\mcF;k)$.

  We define a graph $\Gamma=(V,E)$ with vertex set~$k$ and with an edge
  joining $r_1$ and $r_2$ if there exist $\sigma_1$ and $\sigma_2$ with $\sigma_1^c$ and $\sigma_2^c$
  from~$\uple{s}$ such that
  $\gamma_{r_1,\sigma_1}\gamma_{r_2,\sigma_2}^{-1}\in \mcT_c$. We claim that
  \[
    \Delta_{\uple{s}_1,\uple{s}_2}\subset \Delta_m(V,E),
  \]
  in which case Lemma \ref{lm-graph1} proves that
  $|\Delta_{\uple{s}_1,\uple{s}_2}|\ll |k|^m$, uniformly in terms of
  $(\uple{s}_1,\uple{s}_2)$, and therefore $|\Delta'_{m,c}(\mcF;k)|\ll  |k|^{3m}$, as desired.

  Let
  $\uple{r}=(r_j)_{1\leq j\leq 2m}\in
  \Delta_{\uple{s}_1,\uple{s}_2}$. We prove that the vertex $r_1$
  of~$\Gamma$ is connected to a vertex $r_j$ for some~$j\not=1$. The
  same argument, up to notation, will apply to $r_2$, \ldots, $r_{2m}$,
  and then conclude the proof.

  We use the notation~$(\delta_j)_{1\leq j\leq 4m}$ introduced above; we
  have therefore $\delta_1=\gamma_{r_1,\sigma_{1,1}}$. Since
  $(\uple{r},\uple{s}_1,\uple{s}_2)$ is in $\Delta'_{m,c}(\mcF;k)$, we
  have seen above that there exists an integer $i$ with $2\leq i\leq 4m$
  such that $\delta_1\delta_i^{-1}\in \mcT_c$. If
  $\delta_i=\gamma_{r_j,\sigma}$ for some $j\geq 2$, then we are done.
  But the only other possibility would be that
  $\delta_i=\gamma_{r_1,\sigma_{2,1}}$, and this is not the case since
  $s_{1,1}/s_{2,1}\not\in \mcD_c$, and \emph{a fortiori}
  $s_{1,1}\not=s_{2,1}$.
\end{proof}

Finally, in the case of Type II sums, we will also require the following
result, which is of a slightly different nature than the previous ones.

\begin{proposition}\label{pr-diagonal3}
  Let~$\mcF$ be a \good sheaf over~$k$. Let $l\geq 1$ be an integer and
  let~$c\in\Zz-\{0\}$ be an integer invertible in~$k$. Assume that the characteristic of $k$ is large enough in terms of the complexity of $\mcF$.

  There exist an integer~$C\geq 0$ and an algebraic subvariety
  $\mcV\subset \Aa^{2l}_{k}$ of dimension $\leq l$ with the following
  property: for any finite extension~$k_n$ of~$k$, any
  $\uple{v}\in (\Aa^{2l}\setminus \mcV)(k_n)$ and any $r\in k_n$ for which $\prod_{i=1}^{2l}(r+v_i)\not=0$, one has
   \[
    H^2_c\Bigl(\Aa^1_{\ov{k}}, \bigotimes_{i=1}^l [s\mapsto
    (r+v_i)^cs]^*\mcF\otimes \bigotimes_{i=1}^l [s\mapsto
    (r+v_{i+l})^cs]^*\mcF^{\vee} \Bigr)=0
  \]
  except perhaps for at most $C$ values of $r$, all non-zero.  

  The constant~$C$ and the degree of~$\mcV$ depend only on~$c$, $l$ and
  the complexity of~$\mcF$.
\end{proposition}

\begin{proof}
  Let~$G$ be the geometric monodromy group of~$\mcF$ and~$N$ be its core
  subgroup. Let~$\mcT\subset \Aff(\ov{k})$ denote the symmetric set  determined by Proposition~\ref{pr-good-gkrlemma} for $c=1$. It has cardinality bounded
  in terms of the complexity of~$\mcF$.

  For $u\in \Gm$, we denote here $\mcG_{u}=[s\mapsto us]^*\mcF$ so that we are considering the cohomology group
  \[
    H^2_c\Bigl(\Aa^1_{\ov{k}},\bigotimes_{i=1}^l\mcG_{(r+v_i)^c}\otimes
    \mcG_{(r+v_{i+l})^c}^{\vee}\Bigr)
  \]
for tuples $r+\uple{v}:=(r+v_1,\cdots,r+v_{2l})\in (k_n^\times)^{2l}$.

  We first consider the case $r=0$, and $\uple{v}\in (\kt)^{2l}$. Note
  that for $v\not=0$, we have $\mcG_{v^c}=\mcF_{0,v^c}$.  Arguing as in
  the beginning of the proof of Proposition~\ref{pr-diagonal1} (for $c=1$) and
  applying Proposition~\ref{pr-good-gkrlemma}, we see that if
  \[
    H^2_c\Bigl(\Aa^1_{\ov{k}},\bigotimes_{i=1}^l\mcG_{v_i^c}\otimes
    \mcG_{v_{i+l}^c}^{\vee}\Bigr)
  \]
  is non-zero, then for all $i$ with $1\leq i\leq 2l$, there
  exists~$j\not=i$ such that
  $\gamma_{0,v_i^c}\gamma_{0,v_j^c}^{-1}=\gamma_{0,v_i^c/v_j^c}\in \mcT$.

  For each fixed $\gamma\in \mcT$ and $(i,j)$ with $1\leq i\not=j\leq 2l$,
  the set of $\uple{v}\in\Gg_m^{2l}$ such that
  $\gamma_{0,v_i^c/v_j^c}=\gamma$ is visibly an algebraic
  subvariety of degree bounded in terms of~$c$. Thus, the set of
  $\uple{v}$ satisfying the above conditions is an algebraic
  subvariety~$\mcW$ of $\Gg_m^{2l}$ of degree bounded in terms of~$c$
  and the size of~$\mcT$. We have in fact shown that for any finite
  extension $k_n$ of $k$, the set of $\uple{v}\in \Gg_m^{2l}(k_n)$ such
  that the cohomology group
  \[
    H^2_c\Bigl(\Aa^1_{\ov{k}}, \bigotimes_{i=1}^l [s\mapsto
    v_i^cs]^*\mcF\otimes \bigotimes_{i=1}^l [s\mapsto
    v_{i+l}^cs]^*\mcF^{\vee} \Bigr)
  \]
  is non-zero, is contained in $\mcW(k_n)$ (in other words, changing the
  base field does not change the equations used to control the situation
  for $\uple{v}$ with coefficients in varying extensions).

  Let $k_n$ be a finite extension of $k$.  Let $\Gamma_n=(V_n,E_n)$ be the
  graph with vertex set $k_n^{\times}$ and with an edge from $v$ to~$w$
  if and only if $\gamma_{0,v^c}\gamma_{0,w^c}^{-1}\in \mcT$. Applying
  Lemma~\ref{lm-graph1}, we deduce that
  \[
    |\mcW(k_n)|\leq |\Delta_l(\Gamma_n)|\ll |k_n|^l,
  \] 
  where the implied constant depends on~$l$ and the size of~$\mcT$ only. It
  follows (by the Lang--Weil bound for instance) that
  $\dim(\mcW)\leq l$.

  We now define $\mcV$ to be the closure  of~$\mcW$ in~$\Aa^{2l}$, and we
  claim that~$\mcV$ has the desired properties. First, by construction,
  it remains of dimension~$\leq l$ and with bounded degree.
  
  Let~$\uple{v}\in k_n^{2l}$. For $r\in k_n$ such that $v_i+r\not=0$ for all $i$, the condition
  \[
    H^2_c\Bigl(\Aa^1_{\ov{k}}, \bigotimes_{i=1}^l [s\mapsto
    (r+v_i)^cs]^*\mcF\otimes \bigotimes_{i=1}^l [s\mapsto
    (r+v_{i+l})^cs]^*\mcF^{\vee} \Bigr) \not=\{0\}
  \]
  implies that $r+\uple{v}\in \mcV(k_n)$. The set of~$r\in\Aa^1$ such
  that $r+\uple{v}\in \mcV$ is a closed subvariety of~$\Aa^1$. If it is
  a proper subvariety, then it is a finite set of size bounded in terms
  of the degree of~$\mcV$. Otherwise, taking $r=0$, we obtain
  $\uple{v}\in \mcV$.
\end{proof}

\begin{remark}
  A typical situation is when $\mcW$ is defined by a union of varieties
  defined by equations stating pairwise equalities like
  \[
    v_1=v_2,\ \ldots,\ v_{2l-1}=v_{2l},
  \]
  with the added condition that the coordinates are
  invertible. Then~$\mcV$ is defined by the same equations without
  invertibility conditions, and we see that the equations are invariant
  under translation by a common element~$r$.
\end{remark}

The following corollary illustrates in general the use of this type of
results.  It will be applied in the proof of Theorem~\ref{thmtriplesum},
but is also of independent interest.

\begin{corollary}\label{cor-sop}
  Let~$q$ be a prime number and let~$\mcF$ be a \lgood sheaf
  on~$\Aa^1_{\Ff_q}$. Let~$l\geq 1$ be an integer. Let~$K$ be the trace
  function of~$\mcF$.

  There exists an algebraic subvariety $\mcV$ of $\Aa^{2l}_{\Ff_q}$ of
  dimension $\leq l$ and degree bounded in terms of $l$ and the
  complexity of~$\mcF$ such that the estimate
  \begin{equation}
    \label{sumproductsbound}
    \sum_{v\in\Fqt}\prod_{i=1}^lK(u_iv)
    \ov{ K(u_{i+l}v)}\ll q^{1/2}
  \end{equation}
  holds for $\bfu=(u_1,\cdots,u_{2l})\in (\Fqt)^{2l}\setminus \mcV(\Ff_q)$, where the implied
  constant depends on~$l$ and the complexity of~$\mcF$.
\end{corollary}

\begin{proof} We assume that $q$ is large enough in terms of the complexity of $\mcF$ otherwise the result is trivial.

  Let~$j\colon U\to \Aa^1$ be the open immersion of a dense open set
  where~$\mcF$ is lisse, and let~$\mcF^*$ be the middle-extension sheaf
  $j_*j^*\mcF$. It is a \lgood sheaf whose trace function~$K^*$
  coincides with that of~$K$ except for a set of values of size bounded
  in terms of the complexity of~$\mcF$. Hence, for any $\bfu\in
  (\Fqt)^{2l}$, the estimate
  \[
    \sum_{v\in\Fqt}\prod_{i=1}^lK(u_iv) \ov{ K(u_{i+l}v)}=
    \sum_{v\in\Fqt}\prod_{i=1}^lK^*(u_iv) \ov{ K^*(u_{i+l}v)}
    +O(1)
  \]
  holds, where the implied constant depends on~$l$ and the complexity
  of~$\mcF$. In particular, the bound~(\ref{sumproductsbound}) is
  equivalent to its analogue for~$K^*$.

  For given $\bfu\in(\Fqt)^{2l}$, the function
  \[
    v\mapsto \prod_{i=1}^lK^*(u_iv) \ov{ K^*(u_{i+l}v)}
  \]
  is the trace function of the sheaf
  \[
    \mcG_{\bfu}=\bigotimes_{i=1}^l [s\mapsto u_is]^*\mcF^*\otimes
    \bigotimes_{i=1}^l [s\mapsto u_{i+l}s]^*(\mcF^*)^{\vee}
  \]
  which is the sheaf appearing in Proposition~\ref{pr-diagonal3} for the
  sheaf~$\mcF^*$, with~$c=1$ and $r=0$. The result then follows with the
  algebraic variety~$\mcV$ given by the proposition (in view of the fact
  that the exceptional values~$r$ in loc. cit. are all non-zero), by the
  Riemann Hypothesis applied to~$\mcG_{\bfu}$ (see Theorem~\ref{th-rh}).
\end{proof}

\section{General estimates for bilinear forms}
\label{bilinearstart}

\subsection{Reduction to complete sums}\label{section8}

The main result of this section is Proposition~\ref{propreduct}, which
gives a general reduction of bounds for bilinear forms (type I or II)
with kernel defined modulo a prime modulus~$q$ to estimates for certain
``complete'' sums.  It is not needed here to assume that~$K$ is a trace
function.

We use throughout the following notation. We denote by $q$ a prime
number, and by $K\colon \Ff_q\to\Cc$ a function which we identify
with a $q$-periodic function on~$\Zz$.  We fix two non-zero integers $b$
and~$c$. 

We also give ourselves two positive integers~$M$ and~$N$ and a subset $\mcM\subset \Fqt$ with $|\mcM|=M$. We then denote
\begin{align*}
  B_{c}(\bfalpha,N;K)&=\sum_{m\in\mcM}\sum_{n\sim N}\alpha_m K(mn^c),
  \\
  B_{c}(\bfalpha,\bfbeta;K)&=\sum_{m\in\mcM}\sum_{n\sim N}\alpha_m
  \beta_nK(mn^c),
\end{align*}
for arbitrary families $\bfalpha=(\alpha_m)_{m\in \mcM}$ and
$\bfbeta=(\beta_n)_{n\sim N}$ of complex numbers.  In fact, since $K$ and $c$ are fixed in this section, we will most often abbreviate and
write
\[
  B(\bfalpha,N)=B_{c}(\bfalpha,N;K),
  \quad
  B(\bfalpha,\bfbeta)=B_{c}(\bfalpha,\bfbeta;K).
\]

We denote
\[
\|\bfalpha\|_2=\Bigl(\sum_{m\in \mcM}|\alpha_m|^2\Bigr)^{1/2},
\]
and similarly for $\|\bfbeta\|_2$.

Let $l\geq 2$ be an integer. For $r$ and $s$ in~$\Fq$ and for
\[
  \bfv=(v_1,\cdots,v_{2l})\in\Ff_q^{2l},
\]
we define
\[
  \bfK_c(r,s,\bfv)=\prod_{i=1}^lK(s(r+v_i)^c) \ov{K(s(r+v_{i+l})^c)}.
\]

% and the sum
% $$
% \Sigma_I(\bfv;q):=\sumsum_{r,s\in \Fq}\bfK(s,r,\bfv).
% $$

The reduction statement is the following.

\begin{proposition}\label{propreduct}
  We keep the notation above. Let $l\geq 1$ be an integer and let
  $M,N,V\geq 1$ be real numbers satisfying
  \begin{equation}
    \label{assumptypeI}
    V\leq \frac{N}{10},\quad M,N, \frac{N^2}{V}\leq q.
  \end{equation}

  \begin{enumth}
  \item For any $\eps>0$, we have
    \begin{equation}
      \label{BalphaNbound}
      B(\bfalpha,N)\ll q^{\eps}{\|\bfalpha\|_2M^{1/2}N}
      \Bigl(\frac{1}{MN^2V^{2l-1}}
      \sum_{\bfv\in [V,2V]^{2l}}|\Sigma_I(\bfv)|\Bigr)^{1/(2l)}
    \end{equation}
    where
    \[
      \Sigma_I(\bfv)=\sumsum_{(r,s)\in \Fq\times\Fqt}\bfK_c(r,s,\bfv)
    \]
    and the implied constant depends only on $\eps$ and $l$.

  \item Let~$\mcD\subset\Fqt$ be a set of cardinality $|\mcD|=d$. For any $\eps>0$, we have
    \begin{equation}
      \label{Balphabetabound}
      B(\bfalpha,\bfbeta)\ll q^{\eps}
      \|\bfalpha\|_2\|\bfbeta\|_2
      (MN)^{1/2}\Bigl(\frac{1}{M}
      +\Bigl(
      \frac{1}{M^2N^2V^{2l-1}}
      \sum_{\bfv\in[V,2V]^{2l}}|\Sigma_{II}^{(\mcD)}(\bfv)|\Bigr)^{1/(2l)}
      \Bigr)^{1/2},
    \end{equation}
    where 
    \[
      \Sigma_{II}^{(\mcD)}(\bfv)=\sum_{r\in\Fq}
      \sumsum_{\substack{s_1,s_2\in\Fqt\\
          s_1/s_2\not\in \mcD}} \bfK_c(r,s_1,\bfv)\ov{\bfK_c(r,s_2,\bfv)}
    \]
    and the implied constant depends only on $\eps$, $\|K\|_\infty$, $d$
    and $l$.
  \end{enumth}
\end{proposition} 

\begin{remark}
  It is easy to check that the dependency on $\|K\|_\infty$ of the
  implied constant in \eqref{Balphabetabound} can be removed if one
  replaces the term $M^{-1}$ on the right-hand side with
  $\|K\|_\infty^2M^{-1}$, either by examining the proof or by an
  amplification argument.  However, the quantity $\|K\|_\infty$ will be
  uniformly bounded in our applications.
  % we consider this is not necessary.
\end{remark}

\begin{remark}
  We recall how such bounds can lead to results like
  Theorem~\ref{thmType12}. We assume that $K$ is bounded by an
  absolute constant (in practice, the complexity of the underlying
  object of which~$K$ is the trace function). Since the factors
  \[
    \|\bfalpha\|_2M^{1/2}N,\quad\quad
    \|\bfalpha\|_2\|\bfbeta\|_2(MN)^{1/2}
  \]
  appearing in the estimates represent the ``trivial'' bounds for
  $B(\bfalpha,N)$ and $B(\bfalpha,\bfbeta)$, the subsequent factors are
  the potential savings. How much one does save depends on our success
  in bounding the sums
  \[
    \sum_{\bfv\in[V,2V]^{2l}}|\Sigma_{I}(\bfv)|,\quad
    \sum_{\bfv\in[V,2V]^{2l}}|\Sigma_{II}^{(\mcD)}(\bfv)|,
  \]
  depending on the unspecified parameters~$V$ and (for the second sum) $d$.
  
  Consider the first sum. Under our assumptions, we have
  $\Sigma_{I}(\bfv)\ll q^2$, but we hope for better, and in fact for
  trace functions of suitable sheaves, as we will see, we can
  % however we expect better and will prove that, when $K$ is the trace
  % function of a sheaf $\mcF$ such that $[c]^*\mcF$ is \good in the
  % sense of Definition \ref{defgoodsheafintro},
  show that square-root cancellation occurs for most values of $\bfv$,
  namely
  \[
    \Sigma_{I}(\bfv)\ll q
  \]
  unless $\bfv$ is in a ``diagonal'' subset of parameters. This set must
  necessarily contain the actual diagonal defined by
  \[
    v_i=v_{i+l},\quad\quad i=1,\cdots, l
  \]
  for which only the trivial bound holds. If there are not many more
  diagonal cases, we can therefore hope to obtain
  \begin{equation*}
    %\label{hope1}
    \sum_{\bfv\in[V,2V]^{2l}}|\Sigma_{I}(\bfv)|\ll V^lq^2+V^{2l}q,
  \end{equation*}
  which is $\ll q^3$ when choosing $V=q^{1/l}$. This bound turns out to
  imply the estimates for Type I sums.
  
  For the type II sums, interpolating similarly between the trivial
  bound $\Sigma_{II}^{(\mcD)}(\bfv)\ll q^{3}$ along a diagonal set and a
  generic square-root cancellation bound
  $\Sigma_{II}^{(\mcD)}(\bfv)\ll q^{3/2}$, we can hope to obtain
  \begin{equation*}
   % \label{hope2}
    \sum_{\bfv\in[V,2V]^{2l}}|\Sigma_{II}^{(\mcD)}(\bfv)|\ll
    V^lq^3+V^{2l}q^{3/2}
  \end{equation*}
  which is $\ll q^{9/2}$ when choosing $V=q^{3/(2l)}$. This leads to the
  bounds for Type II sums.

  Although this intuitive explanation is not too far from the truth, we
  will see that, as in previous works, the actual details are more
  complicated and additional ``stratification'' will be used.
\end{remark}

\subsection{Proof of Proposition \ref{propreduct}}

We first prove the bound for the Type I  sums.

Let~$V\geq 1$ be such that~(\ref{assumptypeI}) holds. We define
$U=N/(10V)$, so that $U\geq 1$.

Let $f_N$ denote the piecewise linear function on $\Rr$ which is equal
to~$1$ in the interval $[N,2N]$ and is zero in
$\mathopen]-\infty,N-1]\cup [2N+1,\infty\mathclose[$.  Let
$\what{f}_N$ be the Fourier transform of~$f_N$, defined by
\[
  \what{f}_N(t)=\int_{\Rr}f_N(s)e(-st)ds
\]
for $t\in\Rr$. The estimate
\[
  |\what{f}_N(t)|\ll \Xi_N(t)
\]
holds, where $\Xi_N(t)=\min(N,|t|^{-1},|t|^{-2})$.

%For any real
%number $u\geq 1$, the bound
%\begin{equation}
%	\label{logNbound}
%	 \int_{\Rr}\frac{1}{u}\Bigl|\what{f}_N(\frac{t}{u})\Bigr|dt=O(\log N)
%\end{equation}
%holds, where  

By elementary changes of variable, we have
\begin{align*}
  B(\bfalpha,N)&=\frac{1}{UV}\sum_{(u,v)\sim U\times V}
  \sumsum_{m\in \mcM,n+uv\sim N}\alpha_m K(m(n+uv)^c)\\
  &=\frac{1}{UV}\sum_{(u,v)\sim U\times V}\sumsum_{m\in \mcM,n\in
    [N/2,2N]}\alpha_m f_N(n+uv)K(m(n+uv)^c).
\end{align*}

Applying the Fourier inversion formula,
\[
  f_N(s)=\int_{\Rr}\what{f}_N(t)e(ts)dt,
\]
we deduce that
\begin{align*}
  B(\bfalpha,N)
  &=\frac{1}{UV}\int_{\Rr}\sum_{(u,v)\sim U\times
    V}\sumsum_{m\in \mcM,n\in [N/2,2N]}\alpha_m e(
    t(n+uv))K(m(n+uv)^c)\what{f_N}(t)dt\\
  &\leq \frac{1}{UV}\int_{\Rr}\sumsum_{m\in \mcM,n\in [N/2,2N]}
    |\alpha_m| \sum_{u\sim
    U}
    \Bigl|\sum_{v\sim V} e(tv)K(u^cm(\ov un+v)^c)\Bigr|
    \frac{1}{u}\Bigl|\what{f_N}\Bigl(\frac{t}{u}\Bigr)\Bigr|dt\\
  &\ll  \frac{1}{UV}\int_{\Rr}\sumsum_{m\in \mcM,n\in [N/2,2N]}
    |\alpha_m| \sum_{u\sim
    U}\Bigl|\sum_{v\sim V} e(tv)K(u^cm(\ov un+v)^c)
    \Bigr|\frac{1}{U}\Xi_N\Bigl(\frac{t}U\Bigr)dt\\
  &\ll \frac{\log N}{UV}\sumsum_{m\in \mcM,n\in
    [N/2,2N]}|\alpha_m|\sum_{u\sim U}\Bigl|\sum_{v\sim V}
    \eps(v)K(u^cm(\ov un+v)^c)\Bigr|
\end{align*}
where $\eps(v)=e( t_0v)$ for some $t_0\in\Rr$. In the last step, we
have also used the fact that
\[
  \int_\Rr\frac{1}{U}\Xi_N\Bigl(\frac{t}U\Bigr)dt=O(\log N).
\]

Given $(r,s)\in\Fqt\times \Fqt$ we define
\[
  \nu(r,s)=\sumsumsum_\stacksum{u^cm\equiv s\mods q}{ \ov un\equiv
    r\mods q}|\alpha_m|,
\]
where the variables are integers restricted to satisfy the size
conditions
\begin{equation}
  \label{mnuvrange}
  m\in \mcM,\quad n\sim N,\quad  u\sim U,\quad  v\sim V.
\end{equation}

Applying Hölder's inequality, we deduce from the previous inequality
that the estimate
\[
  B(\bfalpha,N)\ll \frac{\log N}{UV} \Bigl(\sum_{r,s}|\nu(r,s)|
  \Bigr)^{1-\frac1l} \Bigl(\sum_{r,s}|\nu(r,s)|^2 \Bigr)^{\frac1{2l}}
  \Bigl(\sum_{r,s}\Bigl|\sum_{v\sim
    V}\eps(v)K(s(r+v)^c)\Bigr|^{2l}\Bigr)^\frac{1}{2l}
\]
holds.

\begin{lemma}\label{VboundtypeI}
  We have
  \begin{equation}\label{L1norm}
    \sum_{r,s}|\nu(r,s)|\ll UN\|\bfalpha\|_1\leq UNM^{1/2}\|\bfalpha\|_2
  \end{equation}
  and 
  \begin{equation}\label{L2norm}
    \sum_{r,s}|\nu(r,s)|^2 \ll q^{\eps}
    \|\alpha\|_2^2UN\Bigl(1+\frac{UN}q\Bigr)  \end{equation}
  for any $\eps>0$, where the implied constant depends on $\eps$
  and~$b$.
\end{lemma}

\begin{proof}
  The bound \eqref{L1norm} follows immediately from \eqref{mnuvrange}.

  For \eqref{L2norm}, we write
  \[
    \sum_{r,s}|\nu(r,s)|^2=\sumsum_\stacksum{u_1,u_2,n_1,n_2}{u_2n_1\equiv
      u_1n_2}\sumsum_\stacksum{m_1,m_2}{u_1^cm_1\equiv
      u_2^cm_2}\alpha_{m_1}\ov{\alpha_{m_2}}\ll
    \sumsum_\stacksum{u_1,u_2,n_1,n_2}{u_2n_1\equiv
      u_1n_2}\sumsum_\stacksum{m_1,m_2}{u_1^cm_1\equiv u_2^cm_2}
    |\alpha_{m_1}|^2
  \]
  since
  $|\alpha_{m_1}\ov{\alpha_{m_2}}|\leq
  |\alpha_{m_1}|^2+|\alpha_{m_2}|^2$.
  
  We express the congruence $u_2n_1\equiv u_1n_2\mods q$ as an equality
  $u_1n_2=u_2n_1+qk$ for some integer~$k$ such that $|k|\ll (UN/q+1)$
  according to the size conditions above. Fixing $u_2$, $n_1$ and $k$,
  there are at most $(UN)^{\eps}\ll q^{\eps}$ solutions $(u_1,n_2)$ of
  these equations for any $\eps>0$.  Moreover, for any given
  $(u_1,u_2,m_1)$, the number of $m_2$ satisfying the congruence
  \[
    u_1^cm_1\equiv u_2^cm_2\mods q
  \]
  is $\leq 1$.
  Combining these bounds concludes the proof of the lemma.
\end{proof}

Applying Lemma~\ref{VboundtypeI} and the bound $|\eps(v)|\leq 1$ for
all~$v$, we deduce the estimate
\[
  B(\bfalpha,N)\ll
  q^{\eps}\frac{\|\bfalpha\|_2^{1-1/l}M^{1/2}\|\bfalpha\|_2^{1/l}N}
  {V(UMN)^{1/(2l)}} \Bigl(\sum_{\bfv\sim [V,2V]^{2l}}|\Sigma_I(\bfv;q)|
  \Bigr)^{1/(2l)}
\]
for any $\eps>0$, where the implied constant depends on~$\eps$ and~$b$.

We now prove the Type II bound, and we will be brief. We denote again
$U=N/(10V)$.

Applying the Cauchy--Schwarz inequality and following the argument
of~\cite{Pisa}*{\S 4.1}, we deduce the estimate
\[
  B(\bfalpha,\bfbeta)\ll
  \|\bfbeta\|_2(\|\bfalpha\|^2_2N+S^{\not\equiv})^{1/2}
\]
where the implied constant depends on $d$ and
\begin{multline*}
  S^{\not\equiv}=\sumsum_{m_1/m_2\not\in
    \mcD}\alpha_{m_1}\ov{\alpha_{m_2}}
  \sum_{n\sim N}K(m_1n^c)\ov K(m_2n^c)\\
  \ll\frac{\log q}{UV}\sumsum_\stacksum{r,s_1,s_2}{s_1/
    s_2\not\in\mcD}\nu(r,s_1,s_2)\Bigl |\sum_{v\sim V}\eps(v)K(s_1(r+v)^c)\ov
  K(s_2(r+v)^c)\Bigr|
\end{multline*}
where $\eps(v)=e(t_0v)$ for some $t_0\in \Rr$ and we define
\[
  \nu(r,s_1,s_2)=\sumsum_\stacksum{u,n, \ov un\equiv r}{ u^cm_i\equiv
    s_i,\ i=1,2}|\alpha_{m_1}{\alpha_{m_2}}|,
\]
for $(r,s_1,s_2)\in(\Fqt)^3$, the variables being constrained by the
size conditions
\[
  u\sim U,\quad m_1,m_2\in \mcM,\quad n\sim N.
\]

\begin{lemma}\label{VboundtypeII}
  Assume that $UN\leq q$. We have
  \begin{equation*}%\label{L1norm2}
    \sum_{r,s_1,s_2}\nu(r,s_1,s_2)\leq UMN\|\bfalpha\|^2_2.
  \end{equation*}
For any $\eps>0$, we have
  \begin{equation*}%\label{L2norm2}
    \sum_{r,s_1,s_2}\nu(r,s_1,s_2)^2 \ll q^{\eps}\|\bfalpha\|_2^4UN,
  \end{equation*}
where the implied constants depend only on $c$ and $\eps$.
  %% \footnote{Check}
\end{lemma}

\begin{proof}
  The first bound follows from the elementary estimates
  \[
    \sum_{r,s_1,s_2}\nu(r,s_1,s_2)\leq
    \sumsum_\stacksum{u,n}{m_1,m_2}|\alpha_{m_1}\alpha_{m_2}|\leq
    UN\|\bfalpha\|_1^2\leq UMN\|\bfalpha\|_2^2.
  \]

  To prove the second bound, we first note that
  \[
    \sumsum_{r,s_1,s_2}\nu(r,s_1,s_2)^2=
    \sumsum_{u_1,n_1,m_1,m_2}|\alpha_{m_1}{\alpha_{m_2}}|
    \sumsum_{\substack{u_2,n_2,\mu_1,\mu_2\\
        \ov u_2n_2\equiv \ov u_1 n_1\\
        u_2^c\mu_i\equiv u_1^cm_i,\ i=1,2}}
    |\alpha_{\mu_1}{\alpha_{\mu_2}}|.
  \]

  We express the congruence $\ov u_2n_2\equiv \ov u_1 n_1$ as an
  equation $u_2n_1=u_1n_2+qk$ for some integer $k\ll 1+UN/q$ (by the size
  conditions). Once $u_1$, $n_2$ and $k$ are fixed, there are
  $\ll (UN)^{\eps}$ solutions $(u_2,n_1)$ of the equations, for any
  $\eps>0$, where the implied constant depends only on~$\eps$. Writing
  \[
    |\alpha_{m_1}{\alpha_{m_2}}|\leq |\alpha_{m_1}|^2+|{\alpha_{m_2}}|^2
  \]
  and using the estimate (for any given $u_1,u_2$)
  \[
    \sumsum_{ u_2^c\mu_i\equiv u_1^cm_i}|\alpha_{m_i}|^2\leq
    \|\bfalpha\|_2^2
  \]
  for $i=1$ and $2$, we see
  that
  \[
    \sumsum_{r,s_1,s_2}\nu(r,s_1,s_2)^2\ll
    q^{\eps}\|\bfalpha\|_2^4UN
    \Bigl(1+\frac{UN}q\Bigr)\ll
    q^{\eps}\|\bfalpha\|_2^4UN
   ,
  \]
  which concludes the proof.
\end{proof}

Applying Hölder's inequality twice, as in \cite[\S\,4.1]{Pisa}, we
deduce the bound
\begin{align*}
  \sumsum_{r,s_1,s_2}\nu^{1-\frac1l+\frac1l} \Bigl|\sum_{v\sim
    V}\cdots\Bigr| &\leq \Bigl(\sumsum_{r,s_1,s_2}\nu\Bigr)^{1-\frac1l}
  \Bigl(\sumsum_{r,s_1,s_2}\nu\Bigl|\sum_{v\sim
    V}\cdots\Bigr|^l\Bigr)^{\frac1{l}}
  \\
  &\leq \Bigl(\sumsum_{r,s_1,s_2}\nu\Bigr)^{1-\frac1l}
  \Bigl(\sumsum_{r,s_1,s_2}\nu^2\Bigr)^{\frac{1}{2l}}
  \Bigl(\sumsum_{r,s_1,s_2}\Bigl|\sum_{v\sim
    V}\cdots\Bigr|^{2l}\Bigr)^{\frac1{2l}}
\end{align*}
and using Lemma \ref{VboundtypeII} (under the assumption $UN\leq q$),
we conclude that
\begin{align*}
  S^{\not\equiv }&\ll \frac{q^{\eps}}{UV}
  \|\bfalpha\|_2^2(UMN)^{1-1/l}(UN)^{1/(2l)}
  \Bigl(\sum_{\bfv\in[V,2V]^{2l}}|\Sigma^{(\mcD)}_{II}(K;\bfv)|\Bigr)^{1/(2l)}
  \\
  &\ll
  \frac{q^{\eps}}{N}\|\bfalpha\|_2^2(MN^2/V)^{1-1/l}(N^2/V)^{1/(2l)}
  \Bigl(\sum_{\bfv\in[V,2V]^{2l}}|\Sigma^{(\mcD)}_{II}(K;\bfv)|\Bigr)^{1/(2l)}
  \\
  &\ll q^{\eps}\|\bfalpha\|_2^2MN\Bigl(\frac{1}{M^2N^2V^{2l-1}}
  \sum_{\bfv\in[V,2V]^{2l}}|\Sigma^{(\mcD)}_{II}(K;\bfv)|\Bigr)^{1/(2l)},
\end{align*}
for any $\eps>0$.

This finishes the proof of Proposition~\ref{propreduct}.

\subsection{Reduction to a stratification statement}

We now indicate the precise step which we will use to implement
Proposition~\ref{propreduct} in the case where $K$ is a trace
function. However, we do not yet need to make this assumption.

We use the notation from the previous section. Recall our convention for
the degree of an algebraic variety which is not necessarily irreducible.

\begin{proposition}\label{pr-strat}
  Let~$q$ be a prime number, let $K\colon \Ff_q\to\Cc$ be a function,
  and let $l\geq 2$ be an integer. Let~$C\geq 1$ be a real number.

  \begin{enumth}
  \item Suppose that there exist algebraic varieties $\mcV_1$ and
    $\mcV_1^{\Delta}$ over $\Ff_q$ such that
    \begin{gather*}
      \Aa^{2l}_{\Fq}\supset \mcV_1\supset \mcV_1^\Delta
      \\
      \dim (\mcV_1) \leq l+\Bigl\lceil\frac{l}{2}\Bigr\rceil,\quad \dim
      (\mcV^\Delta_1) \leq l,
      \\
      \max(\deg(\mcV_1),\deg(\mcV_1^{\Delta}))\leq C
    \end{gather*}
    and such that the estimates
    \begin{equation}
      \label{stratboundsI}
      |\Sigma_{I}(\bfv)|\leq
      \begin{cases}
        Cq&\text{ for }\bfv\not\in \mcV_1(\Fq),
        \\
        Cq^{3/2}&\text{ for }\bfv\not\in \mcV^\Delta_1(\Fq),
        \\
        Cq^2&\text{ for } \bfv\in \mcV_1^{\Delta}(\Ff_q),
      \end{cases}
    \end{equation}
    hold. We then have
    \begin{equation}
    	\label{sumsigmaI}
    	\sum_{\bfv\in[V,2V]^{2l}}|\Sigma_{I}(\bfv)|\ll q^{3+1/(2l)},
      % V^{l}q^2+V^{l+\lceil\frac{l}{2}\rceil}q^{3/2}+V^{2l}q \ll
      % q^3+q^{3+1/(2l)}
    \end{equation}
    for $V=q^{1/l}$, where the implied constant depends only on~$C$.

  \item Suppose that there exist algebraic varieties $\mcV_2$ and
    $\mcV_2^{\Delta}$ over $\Ff_q$ such that
    \begin{gather*}
      \Aa^{2l}_{\Fq}\supset \mcV_2\supset \mcV_2^\Delta
      \\
      \dim (\mcV_{2}) \leq l + 2 \Bigl\lceil
      \frac{l}{3}\Bigr\rceil,\quad \dim (\mcV^\Delta_2) \leq l,
      \\
      \max(\deg(\mcV_2),\deg(\mcV_2^{\Delta}))\leq C,
    \end{gather*}
    and such that the estimates
    \begin{equation}
      \label{stratboundsII}
      |\Sigma_{II}^{(\mcD)}(\bfv)|\leq
      \begin{cases}
        Cq^{3/2}&\text{ for }\bfv\not\in \mcV_2(\Fq),
        \\
        Cq^{2}&\text{ for }\bfv\not\in \mcV^\Delta_2(\Fq),
        \\
        Cq^3&\text{ for  }\bfv\in\mcV_2^{\Delta}(\Ff_q),
      \end{cases}
    \end{equation}
    hold. We then have
   \begin{equation}
    	\label{sumsigmaII}
      \sum_{\bfv\in[V,2V]^{2l}}|\Sigma_{II}^{(\mcD)}(\bfv)|\ll
      q^{9/2+2/l},
    \end{equation}
    for $V=q^{3/(2l)}$, where the implied constant depends only on~$C$.
  \end{enumth}
\end{proposition}

\begin{proof}
  (1) Applying Lemma~\ref{lm-sz}, the assumption implies that
  \begin{align*}
    \sum_{\bfv\in[V,2V]^{2l}}|\Sigma_{I}(\bfv)|&\ll
    V^{l}q^2+V^{\dim(\mcV_1)}q^{3/2}+V^{2l}q\\
    & \ll V^{l}q^2+V^{l+\lceil\frac{l}{2}\rceil}q^{3/2}+V^{2l}q \leq
    V^{l}q^2+V^{{3l}/{2}+1/2}q^{3/2}+V^{2l}q,
  \end{align*}
  where the implied constant depends on~$C$. Picking $V=q^{1/l}$, we
  obtain the conclusion.
  
  (2) Similarly, we get
  \[
    \sum_{\bfv\in[V,2V]^{2l}}|\Sigma_{II}^{(\mcD)}(\bfv)|\ll V^{l}q^3+V^{l
      + 2 \lceil \frac{l}{3}\rceil}q^{2}+V^{2l}q^{3/2}\leq 
      V^{l}q^3+V^{
      5{l}/{3}+4/3}q^{2}+V^{2l}q^{3/2},
  \]
  and conclude after taking $V=q^{3/(2l)}$.
\end{proof}

\begin{remark}If $l$ is divisible by $2$ (resp. by $3$) the exponent
  $1/(2l)$ in \eqref{sumsigmaI} (resp. $2/l$ in \eqref{sumsigmaII})
  can be replaced by $0$.
\end{remark}

\begin{remark}
  Note a crucial difference with the similar argument
  in~\cite[Th.\,4.5]{Pisa}: we \emph{do not require} that the
  subvarieties $\mcV_i$ and $\mcV_i^{\Delta}$ be defined
  over~$\Zz$. However, this means that we require a bound on the degree
  of these varieties. In our application, this will be provided by an
  application of Quantitative Sheaf Theory.
\end{remark}

\section{Moment estimates for \good sheaves}\label{secstrat}

This and the next section contain the core of the proof of
Theorem~\ref{thmType12}. We  establish that the assumptions of
Proposition~\ref{pr-strat} are satisfied if~$K$ is the trace function of
a \good sheaf~$\mcF$ modulo~$q$, with the constant~$C$ depending only on
the complexity of~$\mcF$. This will be deduced from Xu's idea
(Theorem~\ref{XuStep}), which we implement in the next section.

%% \subsection{Moments}

Let $q$ be a prime number and $\mcF$ an $\ell$-adic sheaf
on~$\Aa^1_{\Ff_q}$ for some prime $\ell\not=q$. We assume in this
section that $\mcF$ is a middle-extension \lgood sheaf. We suppose also
given a non-zero integer~$c$ coprime to~$q$ and we denote
by $\mcT_c$ the set given by Proposition~\ref{pr-good-gkrlemma}. We
also denote by $\mcD=\mcD_c\subset\ovFqt$ the image of~$\mcT_c$ be the
morphism
\[
  \begin{pmatrix}s&r\\0&1
  \end{pmatrix}\mapsto s^c.
\]

We let $d\geq 1$ denote the cardinality of~$\mcD_c$, and recall that $d$
is bounded in terms of~$c$ and the complexity of~$\mcF$.

We fix an algebraic closure~$\bFq$ of~$\Ff_q$.  For any finite
extension~$k/\Ff_q$ contained in~$\bFq$ and~$x\in k$, we denote
\[
  K(x;k)=t_{\mcF}(x;k),
\]
the values of the trace function of~$\mcF$ over~$k$.
Given an
integer~$l\geq 1$, $(r,s)\in k\times k^{\times}$ and $\uple{v}\in k^{2l}$, we define
\[
  \bfK_c(r,s,\bfv;k)=\prod_{i=1}^lK(s(r+v_i)^c;k)
  \ov{K(s(r+v_{i+l})^c;k)}
\]
and
\begin{gather*}
  \Sigma_I(\uple{v};k)=\sumsum_{(r,s)\in k\times
    k^{\times}}\bfK_c(r,s,\bfv;k),
     \\
  \Sigma_{II}^{(\mcD)}(\bfv;k)=\sum_{r\in k}
  \sumsum_{\substack{s_1,s_2\in k^{\times}\\
      s_1/s_2\not\in\mcD}} \bfK_c(r,s_1,\bfv;k)\ov{\bfK_c(r,s_2,\bfv;k)}.
\end{gather*}

Thus, for $k=\Ff_q$, these sums are equal to the sums
$\Sigma_I(\uple{v})$ and $\Sigma_{II}^{(\mcD)}(\uple{v})$ of the previous
section.

\begin{proposition}\label{pr-moment1}
  Assume that~$q$ is large enough so that
  Proposition~\ref{pr-good-gkrlemma} applies to~$\mcF$ over~$\Ff_q$, and
  let $\mcD\subset \ovFqt$ be as above.  Let~$m\geq 1$ be an integer. With notation as above, the
  estimates
  \begin{align}
    \label{momentsumI}
    \sum_{ \bfv \in k^{2l}} |\Sigma_{I}(\bfv;k) |^{2m}& \ll
    \abs{k}^{2m+2l } + \abs{k}^{4m+l} ,
    \\
    \label{momentsumII}
    \sum_{ \bfv \in k^{2l} } |\Sigma_{II}^{(\mcD)}(\bfv;k) |^{2m} &\ll
    \abs{k}^{3m+2l} + \abs{k}^{6m+l}
  \end{align}
  hold for all finite extensions $k/\Ff_q$ and the implicit constants depend on $m,l,c$ and $c(\mcF)$.
\end{proposition}

The first step of the proof is to express these moments in a different
way in terms of one variable sums. We use for this purpose a simple
lemma.

\begin{lemma}\label{lm-exch}
  Let~$X$ and~$Y$ be finite sets and $\alpha\colon X\times Y\to\Cc$ a
  function on~$X\times Y$.  For integers~$l\geq 1$ and~$m\geq 1$, the
  equality
  \[
    \sum_{\bfx\in X^{2l}}\Bigl|\sum_{y\in
      Y}\beta(\bfx,y)\Bigr|^{2m}= \sum_{\bfy\in
      Y^{2m}}\Bigl|\sum_{x\in X}\gamma(x,\bfy)\Bigr|^{2l}
  \]
  holds, where
  \begin{gather*}
    \beta(\bfx,y)=\prod_{i=1}^{l}\alpha(x_i,y)\overline{\alpha(x_{i+l},y)},
    \\
    \gamma(x,\bfy)=\prod_{j=1}^{m}\alpha(x,y_j)\overline{\alpha(x,y_{j+m})}.
  \end{gather*}
\end{lemma}

\begin{proof}
  Opening out fully either side of the equality, we see that the
  resulting sums over~$\bfx\in X^{2l}$ and~$\bfy\in Y^{2m}$ have exactly
  the same terms.
\end{proof}

Applied with suitable choices, this lemma implies the formula
\begin{equation}\label{SImoment}
  \sum_{ \bfv \in k^{2l} } \abs{\Sigma_{I}(\bfv;k) }^{2m}=
  \sum_{(\bfr, \bfs) \in (k\times \kt)^{2m} }  \Bigl| \sum_{v \in
    k}\bfK'_c(\bfr,\bfs,v;k)\Bigr|^{2l}
\end{equation}
where
\[
  \bfK'_c(\bfr,\bfs,v;k)=\prod_{j=1}^{m} K ( s_j (v+r_j)^c; k)\overline{
    K(s_{j+m}(v+r_{j+m})^c;k)} .
\]

% Let~$\sigma$ denote the complex conjugation. We denote
% \[
%   (\sigma_1,\ldots,\sigma_{2l}),\quad\quad (\tau_1,\ldots,\tau_{2m})
% \]
% the families of automorphisms of~$\Cc$ defined by
% \[
%   \sigma_i=\begin{cases}
%     \mathrm{Id}&\text{ if } 1\leq i\leq l,\\
%     \sigma&\text{ if } l+1\leq i\leq 2l,
%   \end{cases}
%   \quad\quad
%   \tau_j=\begin{cases}
%     \mathrm{Id}&\text{ if } 1\leq j\leq m,\\
%     \sigma&\text{ if } m+1\leq j\leq 2m.
%   \end{cases}
% \]
  
% Thus we can write
% \[
%   \bfK_c(r,s,\uple{v};k)=\prod_{i=1}^{2l}K(s(r+v_i)^c;k)^{\sigma_i}.
% \]

% Furthermore, for $\bfr$ and $\bfs$ in $k^{2m}$ and $v\in k$, we let
% \[
%   \bfK'_c(\bfr,\bfs,v;k)=\prod_{j=1}^{2m} K ( s_j (v+r_j)^c;
%   k)^{\tau_j}.
% \]

% Opening the $2m$-th power, then exchanging summation, we obtain
% straightforwardly
% \begin{equation}\label{SImoment}
%   \sum_{ \bfv \in k^{2l} } \abs{\Sigma_{I}(\bfv;k) }^{2m}=
%   \sum_{(\bfr, \bfs) \in (k\times \kt)^{2m} }  \Bigl| \sum_{v \in
%     k}\bfK'_c(\bfr,\bfs,v;k)\Bigr|^{2l}.
% \end{equation}

Similarly, defining
\begin{equation*}
 % \label{Ydef}
  Y_d =\{(s_1,s_2)\in\Gm\times \Gm\,\mid\, s_1/s_2\not\in\mcD\},
\end{equation*}
we obtain the identity
\begin{equation}\label{SIImoment}
  \sum_{ \bfv \in k^{2l} }  |\Sigma_{II}^{(\mcD)}(\bfv;k) |^{2m}
  = \sum_{ \bfr \in k^{2m}} \sum_{ (\bfs_1,\bfs_2) \in Y_d(k)^{2m} }
  \Bigl| \sum_{v
    \in k} \bfK'_c(\bfr,\bfs_1,v;k)\ov{ \bfK'_c(\bfr,\bfs_2,v;k)}
  \Bigr|^{2l}.
\end{equation}

(Note that in this formula, the fact that
$(\bfs_1,\bfs_2) \in Y_d(k)^{2m}$ means that
\[
  \bfs_1=(s_{1,j})_{1\leq j\leq 2m},\quad\quad
  \bfs_2=(s_{2,j})_{1\leq j\leq 2m},
\]
with $s_{1,j}/s_{2,j}\not\in\mcD$ for all~$j$.)

Thus we are reduced to estimating \emph{one variable} sums, namely
\[
  \sum_{v \in k}\bfK'_c(\bfs,\bfr,v;k)\quad\text{ and }\quad \sum_{v \in
    k} \bfK'_c(\bfs_1,\bfr,v;k)\ov{ \bfK'_c(\bfs_2,\bfr,v;k)}.
\]

These are expressions of ``sums of products'' type, hence we can use the
Goursat--Kolchin--Ribet machinery, and in particular the results of
Section~\ref{sec:goursat}, to approach them.

\begin{proposition}\label{pr-one-variable}
  With notation and assumptions as above, the following hold:
  \begin{enumth}
  \item For all finite extensions $k$ of~$\Ff_q$, the bound
    \begin{equation}\label{K'typeIbound}
      \sum_{v \in k}\bfK'_c(\bfs,\bfr,v;k)\ll \abs{k}^{1/2}
    \end{equation}
    holds for all but $O ( \abs{k}^{2m})$ values of $(\bfr,\bfs)\in
    k^{2m}\times (k^{\times})^{2m}$.
  \item For all finite extensions $k$ of~$\Ff_q$, the bound
    \begin{equation}\label{K'typeIIbound}
      \sum_{v \in k} \bfK'_c(\bfs_1,\bfr,v;k)
      \ov{ \bfK'_c(\bfs_2,\bfr,v;k)}\ll
      \abs{k}^{1/2}
    \end{equation}	
    holds for all but $O(\abs{k}^{3m})$ values of
    $(\bfr, \bfs_1,\bfs_2)\in k^{2m}\times Y_d(k)^{2m}$.
  \end{enumth}

  In both cases, the implied constants depend on $m$, $c$ and the
  complexity of~$\mcF$.
\end{proposition}

\begin{proof}
  (1) The case of the estimate~(\ref{K'typeIbound}) follows immediately
  from Proposition~\ref{pr-diagonal1} in view of the
  Grothendieck--Lefschetz trace formula and the Riemann Hypothesis
  (Theorem~\ref{th-rh}), since the function
  $v\mapsto K'_c(\uple{r},\uple{s},v;k)$ is the trace function over~$k$
  of the sheaf $\mcF_{\uple{r},\uple{s},c}$ of~(\ref{eq-sheaf-mcf-rsc});
  note that we use here our assumption that $\mcF$ is a middle-extension
  sheaf to ensure that the trace function of~$\mcF^{\vee}$ is the
  complex conjugate of that of~$\mcF$.

  (2) Similarly, the sum
  \[
    \sum_{v \in k} \bfK'_c(\bfs_1,\bfr,v;k)\ov{ \bfK'_c(\bfs_2,\bfr,v;k)}
  \]
  is the sum of the trace function of the sheaf
  \[
    \mcF_{\bfs_1,\bfr,c}\otimes \mcF_{\bfs_2,\bfr,c}^{\vee},
  \]
  % \[
  %   \Bigl(\bigotimes_{1\leq j\leq 2m}\mcF_{s_{1,j},r_j,c}\otimes
  %   \mcF_{s_{1,j+m},r_{j+m},c}\Bigr)
  %   \otimes
  %   \Bigl(\bigotimes_{1\leq j\leq 2m}\mcF_{s_{2,j},r_j,c}\otimes
  %   \mcF_{s_{2,j+m},r_{j+m},c}\Bigr).
  % \]
  and hence the estimate~(\ref{K'typeIIbound}) follows similarly from
  Proposition~\ref{pr-diagonal2} using the Grothendieck--Lefschetz trace
  formula and the Riemann Hypothesis (Theorem~\ref{th-rh}).
\end{proof}

We can now complete the proof of Proposition~\ref{pr-moment1}.  Recall
that the trivial bound
\[
  \bfK'_c(\bfs,\bfr,v;k)\ll 1
\]
also holds, where the implied constant depends on~$m$ and the complexity
of~$\mcF$ (simply because the trace function of~$\mcF$, which is mixed
of weights~$\leq 0$, is bounded in terms of the complexity of~$\mcF$
only). Thus the one variable sums in Proposition~\ref{pr-one-variable}
are always $\ll |k|$.

For~(\ref{momentsumI}), using~(\ref{SImoment}), we obtain
\begin{align*}
  \sum_{ \bfv \in k^{2l}} |\Sigma_{I}(\bfv;k) |^{2m}& = \sum_{(\bfr,
    \bfs) \in (k\times \kt)^{2m} } \Bigl| \sum_{v \in
    k}\bfK'_c(\bfs,\bfr,v;k)\Bigr|^{2l}
  \\
  &\ll |k|^{2m+2l}+|k|^{4m+l},
\end{align*}
(the first term accounting for the ``diagonal cases'' and the second for
the generic bound).  Similarly for~(\ref{momentsumII}),
using~(\ref{SIImoment}), we obtain
\begin{align*}
  \sum_{ \bfv \in k^{2l} } |\Sigma_{II}^{(\mcD)}(\bfv;k) |^{2m} &= \sum_{
    \bfr \in k^{2m}} \sum_{ (\bfs_1,\bfs_2) \in Y_d(k)^{2m}} \Bigl|
  \sum_{v \in k} \bfK'_c(\bfs_1,\bfr,v;k)\ov{ \bfK'_c(\bfs_2,\bfr,v;k)}
  \Bigr|^{2l}
  \\
  &\ll |k|^{3m+2l}+|k|^{6m+l},
\end{align*}
as claimed, with implied constants depending only on $m$ and the
complexity of~$\mcF$.

\section{Stratification for \good sheaves and conclusion of the proof}
\label{secproofmainthm}

We will now complete the proof of Theorem~\ref{thmType12}. We first
remark that it is enough to prove it when~$K$ is the trace function of a
sheaf~$\mcF$ which is a \lgood middle-extension sheaf
on~$\Aa^1_{\Ff_q}$. This follows from the following lemma applied to
$\mcF$ and $\mcF^*=j_*j^*\mcF$ for some open immersion
$j\colon U\to \Aa^1$ of a dense open subset on which~$\mcF$ is lisse.

\begin{lemma}
  Let~$\mcF_1$ and~$\mcF_2$ be \lgood sheaves on~$\Aa^1_{\Ff_q}$ whose
  trace functions $K_1$ and~$K_2$ coincide outside of a
  set~$S\subset \Ff_q$. Let~$b$ and~$c$ be positive integers coprime
  to~$q$. We then have
  \[
    \sum_{m}\sum_n\alpha_m\beta_nK_1(m^bn^c)-
    \sum_{m}\sum_n\alpha_m\beta_nK_2(m^bn^c) \ll
    \|\bfalpha\|_2\,\|\bfbeta\|_2,
  \]
  where the implied constant depends only on $c$, the complexities
  of~$\mcF_1$ and~$\mcF_2$, and the size of~$S$.
\end{lemma}

\begin{proof}
  The difference is bounded by
  \[
    (\|K_1\|_{\infty}+\|K_2\|_{\infty}) \sum_{s\in
      S}\sum_{\substack{m\in \mcM,n\sim N\\mn^c=s}}
    |\alpha_m|\,|\beta_n|.
  \]

  We consider the sum over~$m$ and~$n$ for each~$s\in S$ separately.
  For each element~$x$ of~$\Ff_q$ which is a $c$-th power, we fix
  arbitrarily a $c$-th root~$\sqrt[c]{x}$. Then this sum
  is equal to
  \[
    \sum_{\zeta^c=1} \sum_{\substack{ m\in \mcM\\ s\overline{m}\in
        (\Fqt)^{c}}} |\alpha_m|\,|\beta_{\zeta
      \sqrt[c]{s\overline{m}}}| \leq \sum_{\zeta^c=1}
    \Bigl(\sum_{\substack{ m\in \mcM\\ s\overline{m}\in (\Fqt)^{c}}}
    |\alpha_m|^2\Bigr)^{1/2}
    \Bigl(\sum_{\substack{ m\in \mcM\\
        s\overline{m}\in (\Fqt)^{c}}}
    |\beta_{\zeta\sqrt[c]{s\overline{m}
    }}|^2\Bigr)^{1/2},
  \]
  where we extend~$\bfbeta$ by zero outside of the values $n\sim N$.
  
  By positivity, the first sum over~$m$ is $\leq \|\alpha\|_2^2$. For
  the second sum, we note that for each~$n\sim N$, there are at most~$b$
  values of~$m$ such that $\zeta\sqrt[c]{s\overline{m}^b}=n$, hence this
  sum is $\leq b\|\beta\|_2^2$. The result follows.
\end{proof}

We now check that the stratification statement of
Proposition~\ref{pr-strat} holds for the trace function of a \lgood
middle-extension sheaf. We use the notation introduced in
Section~\ref{secstrat}, in particular for the definition of the
set~$\mcD$.

\begin{proposition}\label{prop-Sigmagalant}
  Let~$q$ be a prime number and let~$\mcF$ be a \lgood middle-extension
  sheaf on~$\Aa^1_{\Ff_q}$. Let~$K\colon \Ff_q\to \Cc$ be the trace
  function of~$\mcF$. 

  Let~$l\geq 1$ be an integer, and define $\Sigma_I(\uple{v})$ and
  $\Sigma_{II}^{(\mcD)}(\uple{v})$ for $\uple{v}\in \Ff_q^{2l}$ as in
  Section~\textup{\ref{section8}}.

  There exists an integer~$C\geq 1$, depending only on~$l$, $c$ and the
  complexity of~$\mcF$ such that the following properties hold if $q$ is
  large enough:
  \begin{enumth}
  \item There exist algebraic varieties $\mcV_1$ and $\mcV_1^{\Delta}$
    over $\Ff_q$ such that
    \begin{gather*}
      \Aa^{2l}_{\Fq}\supset \mcV_1\supset \mcV_1^\Delta
      \\
      \dim (\mcV_1) \leq l+\Bigl\lceil\frac{l}{2}\Bigr\rceil,\quad \dim
      (\mcV^\Delta_1) \leq l,
      \\
      \max(\deg(\mcV_1),\deg(\mcV_1^{\Delta}))\leq C
    \end{gather*}
    and 
    \[
      |\Sigma_{I}(\bfv)|\leq
      \begin{cases}
        Cq&\text{ for }\bfv\not\in \mcV_1(\Fq),
        \\
        Cq^{3/2}&\text{ for }\bfv\not\in \mcV^\Delta_1(\Fq),
        \\
        Cq^2&\text{ for } \bfv\in \mcV_1^{\Delta}(\Ff_q).
      \end{cases}
    \]

  \item There exist algebraic varieties $\mcV_2$ and $\mcV_2^{\Delta}$
    over $\Ff_q$ such that
    \begin{gather*}
      \Aa^{2l}_{\Fq}\supset \mcV_2\supset \mcV_2^\Delta
      \\
      \dim (\mcV_{2}) \leq l + 2 \Bigl\lceil
      \frac{l}{3}\Bigr\rceil,\quad \dim (\mcV^\Delta_2) \leq l,
      \\
      \max(\deg(\mcV_2),\deg(\mcV_2^{\Delta}))\leq C,
    \end{gather*}
    and
    \[
      |\Sigma_{II}^{(\mcD)}(\bfv)|\leq
      \begin{cases}
        Cq^{3/2}&\text{ for }\bfv\not\in \mcV_2(\Fq),
        \\
        Cq^{2}&\text{ for }\bfv\not\in \mcV^\Delta_2(\Fq),
        \\
        Cq^3&\text{ for  }\bfv\in\mcV_2^{\Delta}(\Ff_q).
      \end{cases}
    \]
  \end{enumth}
\end{proposition}

Once this Proposition is proven we can conclude the proof of
Theorem~\ref{thmType12}. Indeed, we already reduced to the case of
a middle-extension sheaf. Proposition \ref{prop-Sigmagalant} shows that
the trace function of a \lgood middle-extension sheaf~$\mcF$ satisfies
the conditions of Proposition~\ref{pr-strat}. In the case of Type II
sums, we then combine the resulting estimate
\[
  \sum_{\bfv\in[V,2V]^{2l}}|\Sigma_{II}^{(\mcD)}(\bfv)|\ll
  q^{9/2+2/l}
\]
for $V=q^{3/(2l)}$ (the condition \eqref{assumptypeI} is satisfied by
\eqref{assumptypeIIbasic}) with the bound~(\ref{Balphabetabound}) from
Proposition~\ref{propreduct} to deduce
\begin{align*}
  \sum_{m\in \mcM}\sum_{n\sim N}\alpha_m \beta_nK(m^bn^c)
  &\ll q^{\eps}
  \|\bfalpha\|_2\|\bfbeta\|_2 (MN)^{\demi}\Bigl(\frac{1}{M} +\Bigl(
  \frac{1}{M^2N^2V^{2l-1}}
  \sum_{\bfv\in[V,2V]^{2l}}|\Sigma_{II}^{(\mcD)}(\bfv)|\Bigr)^{\tfrac{1}{2l}}
  \Bigr)^{\demi}
  \\
  &\ll q^{\eps}
  \|\bfalpha\|_2\|\bfbeta\|_2 (MN)^{\demi}\Bigl(\frac{1}{M} +
  \Bigl(\frac{1}{MN}\Bigr)^{\tfrac{1}{l}}
    \Bigl(q^{\tfrac{9}{2}+\tfrac{2}{l}-\tfrac{3(2l-1)}{2l}}\Bigr)^{\tfrac{1}{2l}}
    \Bigr)^{\demi},
\end{align*}
which is~(\ref{BtypeII}). A similar
argument deduces~(\ref{BtypeI}) from~(\ref{BalphaNbound}).

\begin{proof}
  (1) Let
  $p_{12}\colon (\Aa^{1}\times \Gg_m\times \Aa^{2l})_{\Ff_q}\to
  \Aa^{2l}_{\Ff_q}$ be the projection $(r,s,\uple{v})\mapsto
  \uple{v}$. For $1\leq i\leq 2l$, let
  \[
    g_i\colon (\Aa^{1}\times \Gg_m\times \Aa^{2l})_{\Ff_q}\to \Aa^1,
  \]
  be the morphism defined by
  \[
    g_i(r,s,\uple{v})=s(r+v_i)^c.
  \]
  
  Let $M$ be the mixed complex on~$\Aa^{2l}_{\Ff_q}$ defined by
  \[
    M=Rp_{12!}\Bigl( \bigotimes_{i=1}^l g_i^*\mcF\otimes
    g_{i+l}^*\mcF^{\vee} \Bigr),
  \]
  (where~$\mcF$ and~$\mcF^{\vee}$ are viewed as complexes in
  degree~$0$).  By the Grothendieck--Lefschetz trace formula, the
  complex~$M$ satisfies
  \[
    t_M(\uple{v};k)=\Sigma_{I}(\bfv;k)
  \]
  for any finite extension~$k$ of~$\Ff_q$ and any~$\uple{v}\in
  k^{2l}$. By the Riemann Hypothesis (in the form
  of~\cite[Th.\,I]{WeilII}), the complex~$M$ is mixed of \emph{integral}
  weights~$\leq 0$ since we assumed that~$\mcF$ itself is mixed of
  integral weights~$\leq 0$.

  We apply Theorem~\ref{XuStep} to $X = \Aa^{2l}$ and to the
  complex~$M$.  We let $m=\lceil \frac{l}{2} \rceil$ and put
  $A=4\lceil \frac{l}{2} \rceil+l$. Proposition~\ref{pr-moment1} gives
  \[
    \sum_{ \bfv \in k^{2l}} \abs{\Sigma_{I}(\bfv;k) }^{2m} \ll
    \abs{k}^{2\lceil \frac{l}{2} \rceil+2l } + \abs{k}^{4\lceil
      \frac{l}{2} \rceil+l}\leq 2\abs{k}^{A}
  \]
  for any finite extension~$k$ of~$\Ff_q$. Theorem~\ref{XuStep} implies
  that there exist closed subschemes $X^{(3)}\supset X^{(4)}$ with
  \[
    \dim(X^{(3)})\leq A - 3m= \Bigl\lceil \frac{l}{2} \Bigr\rceil
    +l,\quad\quad \dim(X^{(4)})\leq l,
  \]
  such that
  \begin{align*}
    \Sigma_{I}(\bfv;k)&\ll \abs{k}^{\frac{3-1}{2}}=|k|\text{ for
    }\bfv\not \in X^{(3)}(k),
    \\
    \Sigma_{I}(\bfv;k)&\ll \abs{k}^{3/2}\text{ for }\bfv\not \in
    X^{(4)}(k).
  \end{align*}

  This gives the desired stratification with $\mcV_1^\Delta=X^{(4)}$ and
  $\mcV_1 =X^{(3)}$.

  (2) We recall that $Y_d$ is the subvariety of~$\Gg_m\times \Gg_m$
  defined by the equation $s_1/s_2\not\in\mcD$.
  Let~$p_{123}\colon (\Aa^1\times Y_d\times \Aa^{2l})_{\Ff_q}\to
  \Aa^{2l}_{\Ff_q}$ be the projection
  $(r,s_1,s_2,\uple{v})\to \uple{v}$.  For $1\leq i\leq 2l$, let $g_i$
  and $h_i$ be the morphisms
  \[
    (\Aa^{1}\times Y_d\times \Aa^{2l})_{\Ff_q}\to \Aa^1,
  \]
  defined by
  \[
    g_i(r,s_1,s_2,\uple{v})=s_1(r+v_i)^c,\quad\quad
    h_i(r,s_1,s_2,\uple{v})=s_2(r+v_i)^c.
  \]

  Let $M$ be the mixed complex on~$\Aa^{2l}_{\Ff_q}$ defined by
  \[
    M=Rp_{123!}\Bigl(\Bigl( \bigotimes_{i=1}^l g_i^*\mcF\otimes
    g_{i+l}^*\mcF^{\vee} \Bigr)\otimes \Bigl(\bigotimes_{i=1}^l
    h_i^*\mcF^{\vee}\otimes h_{i+l}^*\mcF \Bigr)\Bigr).
  \]

  By the Lefschetz trace formula, the complex~$M$ satisfies
  $t_M(\uple{v};k)=\Sigma^{(\mcD)}_{II}(\bfv;k)$ for all~$\uple{v}\in k^{2l}$,
  and by the Riemann Hypothesis~\cite[Th.\,I]{WeilII}, it is mixed of
  integral weights~$\leq 0$.

  We apply Theorem~\ref{XuStep} to $X = \Aa^{2l}$ and to $M$. For
  $m = \lceil \frac{l}{3} \rceil $ and
  $A = 6 \lceil \frac{l}{3} \rceil + l$, we have
  \[
    \sum_{ \bfv \in \Aa^{2l} (k) }  \abs{\Sigma^{(\mcD)}_{II}(K,\bfv;k) }^{2m}
    \ll\abs{k}^{3m+2l} + \abs{k}^{6m+l} \ll \abs{k}^{A}
  \]
  for any finite extension $k/\Ff_q$ by Proposition~\ref{pr-moment1}.
  Hence we obtain closed subschemes
  \[
    X^{(3)}\supset X^{(4)}\supset X^{(5)}\supset X^{(6)}
  \]
  such that
  \begin{align*}
    \dim(X^{(w)})&\leq A - m w = (6-w) \Bigl\lceil \frac{l}{3}
    \Bigr\rceil+ l,
    \\
    \Sigma^{(\mcD)}_{II}( \bfv;k)& \ll \abs{k}^{ \frac{w-1}{2}} \quad \text{ for
    } \bfv\not\in X^{(w)}(k).
  \end{align*}

  We define $\mcV_{2} = X^{(4)}$, so that
  $\dim(\mcV_2)\leq l + 2 \lceil \frac{l}{3}\rceil$ and
  \[
    \Sigma^{(\mcD)}_{II}( \bfv;k) \ll \abs{k}^{3/2} \quad \text{ for }
    \bfv\not\in \mcV_2(k),
  \]
  which corresponds to the first part of our claim. 

  However, neither~$X^{(5)}$ nor $X^{(6)}$ is suitable as a choice
  of~$\mcV_2^{\Delta}$ (the former's dimension is too large, and the
  bound given outside of the latter is of size $|k|^{5/2}$ instead of
  our goal of $|k|^2$). To define $\mcV_2^{\Delta}$, we use instead a
  different trick.
  % it is not
  % sufficient to take $\mcV_{2}^{\Delta} = X^{(6)}$ in this case, as we
  % only get a bound of $O( \abs{k}^{\frac{5}{2}})$ for points outside
  % $X^{(6)}(k)$ and not the desired $O (\abs{k}^2)$ (and taking
  % $X^{(5)}(k)$ would leave the codimension too high).
  
  Let~$k$ be a finite extension of~$\Ff_q$. We observe that
  \begin{align*}
    \Sigma^{(\mcD)}_{II}(\bfv;k)&=\sum_{r\in k}\Bigl|\sum_{ s\in
      k^\times}\bfK(r,s,\bfv;k)\Bigr|^2-
    \sum_{r\in k}\sum_{\substack{s_1,s_2\in k^{\times}\\
        s_1/s_2\in\mcD}}
    \bfK(r,s_1,\bfv;k)\overline{\bfK(r,s_2,\bfv;k)}\\
    &= \sum_{r\in k}\Bigl|\sum_{ s\in k^\times}\bfK(r,s,\bfv;k)\Bigr|^2
    -\sum_{\substack{\xi\in \mcD}} \sum_{r\in
      k}\sum_{s\in k^{\times}} \bfK(r,s,\bfv;k)\overline{\bfK(r,\xi
      s,\bfv;k)}.
  \end{align*}

  Since~$\mcF$ is mixed of weights~$\leq 0$, we have
  \[
    \sum_{\xi\in\mcD} \sum_{r\in k}\sum_{s\in k^{\times}}
    \bfK(r,s,\bfv;k)\overline{\bfK(r,\xi s,\bfv;k)}\ll |k|^2
  \]
  for all~$\uple{v}\in k^{2l}$, and it is therefore enough to prove the
  existence of a subvariety $\mcV_2^{\Delta}$ over $\Ff_q$ of
  dimension~$\leq l$ and bounded degree such that
  \[
    \sum_{r\in k}\Bigl|\sum_{ s\in k^\times}\bfK(r,s,\bfv;k)\Bigr|^2
    \ll |k|^{2}
  \]
  for $\uple{v}\notin \mcV_2^{\Delta}(k)$.  This follows from
  Proposition~\ref{pr-diagonal3}. Indeed, let $\mcV^{\Delta}_2$ be the
  algebraic variety provided by this proposition applied to~$\mcF$.
  Let~$\uple{v}\notin \mcV^{\Delta}_2(k)$. For all but a bounded number
  of~$r\in k$ (the ones for which $r+v_i=0$ for some $i$ and the ones from Proposition~\ref{pr-diagonal3}), we have
  \[
    H^2_c\Bigl(\Aa^1_{\ov{k}}, \bigotimes_{i=1}^l [s\mapsto
    (r+v_i)^cs]^*\mcF\otimes \bigotimes_{i=1}^l [s\mapsto
    (r+v_{i+l})^cs]^*\mcF^{\vee} \Bigr)=0
  \]
  by the proposition, hence
  \[
    \sum_{s\in k^{\times}} \bfK(r,s,\bfv;k)\ll |k|^{1/2}
  \]
  by the Riemann Hypothesis (Theorem~\ref{th-rh}). For the remaining values of~$r$, we have
  \[
    \sum_{s\in k^{\times}} \bfK(r,s,\bfv;k)\ll |k|
  \]
  by the trivial bound. Hence, we deduce that
  \[
    \sum_{r\in k}\Bigl|\sum_{ s\in k^\times}\bfK(r,s,\bfv;k)\Bigr|^2
    \ll |k|^{2}
  \]
  for all $\uple{v}\notin\mcV^{\Delta}_2(k)$, as desired.
  % We have
  % \[
  %   \sum_{r\in k}\Bigl|\sum_{ s\in k^\times}\bfK(r,s,\bfv;k)\Bigr|^2
  %   =\sum_{r\in k}\Bigl| \sum_{s\in \kt} \prod_{i=1}^{l}K(s(v_i+r)^c;k)
  %   \ov{K(s(v_{i+l}+r)^c;k)}  \Bigr|^2
  % \]
  % % where
  % % \[
  % %   \widetilde{\bfK}(r,s,\bfv;k) =\sum_{s\in \kt}
  % %   \prod_{i=1}^{l}K(s(v_i+r)^c;k) \ov{K(s(v_{i+l}+r)^c;k)}.
  % % \]
  % % where
  % % \[
  % %   (\bfv+\underline r)^c=((v_i+r)^c)_{1\leq i\leq 2l}
  % % \]
  % % and
  % % \begin{align}\label{SKdef}
  % %   S(K,(\bfv+\underline r)^c;k)&=\sum_{s\in \kt}
  % %   \prod_{i=1}^{l}K((v_i+r)^c s;k)
  % %   \ov{K((v_{i+l}+r)^c s;k)}\\
  % %   &=\sum_{s\in \kt}\prod_{i=1}^{l}K(\gamma_{0,(v_i+r)^c} s;k)
  % %   \ov{K(\gamma_{0,(v_{i+l}+r)^c} s;k)}.\nonumber
  % % \end{align}
  % We then take $\mcV_{2}^{\Delta}$ to be the subvariety of tuples
  % $\bfv=(v_1,\cdots,v_{2l})$ such that for any $i=1,\cdots, 2l$ there
  % exists $j\not=i$ for which $(0,(r+v_j)^c)\simt (0,(r+v_i)^c)$ (for the
  % sheaf $\mcF$). By the same reasoning as in \S \ref{secprooflemma61},
  % it has dimension $l$ and the size of $\mcV_{2}^{\Delta}(k)$ is
  % $O(|k|^l)$. Moreover, since $\mcF$ is \good, we have
  % \[
  %   S(K,(\bfv+\underline r)^c;k))\ll |k|^{1/2}\text{ and
  %   }\Sigma_{II}(K,\bfv;k)\ll |k|^{2}\text{ for }\bfv\not\in
  %   \mcV_{2}^{\Delta}(k).
  % \]
\end{proof}

% \begin{remark}
%   The argument following \eqref{SKdef} is also used in the proof of
%   Theorem \ref{trilinearbound}.
% \end{remark}

\section{Bounds for trilinear sums with monomial
  arguments}\label{sec-trilinear}

In this section, we prove Theorem \ref{thmtriplesum}, and we use the
notation from that statement.  As in the proof of
Theorem~\ref{thmType12} in the previous section, we check first
that we can assume that~$\mcF$ is also a middle-extension sheaf. Viewing
the integers $a$, $b$ and~$c$, as well as the trace function~$K$ of the
\good sheaf~$\mcF$ as fixed, we denote simply
\[
  T(\bfalpha,\bfbeta,\bfgamma)= \sumsumsum_{j\sim J,m\sim M,n\sim
    N}\alpha_j\beta_m\gamma_n K(j^am^bn^c)
\]
for families $\bfalpha=(\alpha_j)_{j\sim J}$,
$\bfbeta=(\beta_m)_{m\sim M}$ and $\bfgamma=(\gamma_n)_{n\sim N}$ of
complex numbers. We assume that
\[
  |\alpha_j|\leq 1,\quad |\beta_m|\leq 1,\quad |\gamma_n|\leq 1
\]
for all~$j$, $m$ and~$n$.

We write
\[
  T(\bfalpha,\bfbeta,\bfgamma)=\sumsum_{u,v\in \Fqt}\xi_u\zeta_v
  K(uv)
\]
where
\[
  \xi_u=\sum_{\substack{j\sim J\\j^a\equiv u\mods
      q}}\alpha_j,\quad\quad \zeta_v=\sumsum_\stacksum{m\sim M,n\sim
    N}{m^bn^c\equiv v\mods q}\beta_m\gamma_n.
\]

We have the following simple estimates for the size of these
coefficients.

\begin{lemma}\label{L2pierce}
  We have
  \begin{gather*}
    \|\bfxi\|_1\ll J,\quad\quad \|\bfxi\|^2_2\ll J\Bigl(\frac{J}{q}+1\Bigr),
    \\
    \|\bfzeta\|_1\ll MN,\quad\quad
    \|\bfzeta\|_2^2\ll \frac{(MN)^2}{q}+MN(\log q)^2,
  \end{gather*}
  where the implied constants depends on $a$, $b$ and~$c$.
\end{lemma}

\begin{proof}
  The bounds for $\|\bfxi\|_1$ and~$\|\bfzeta\|_1$ are clear.

  Next, using $|\alpha_j|\leq 1$, we have
  \[
    \|\bfxi\|_2^2= \sum_{u\in\Fqt}|\xi_u|^2\leq \sum_\stacksum{j_1\sim
      J,\ j_2\sim J} {j_1^a\equiv j_2^a\mods q}1\ll J(J/q+1).
  \]

  Finally, we use multiplicative characters for the last estimate. We have
  \begin{align*}
    \|\bfzeta\|^2_2=\sum_{v}|\zeta_v|^2
    &\leq \sum_{\substack{m_1,m_2,n_1,n_2\\m_i\sim M,\, n_i\sim N\\
    m_1^bn_1^c\equiv m_2^bn_2^c\mods q}}1\\
    &=\frac{1}{q-1}\sum_{\chi\mods q}\Bigl|
      \sum_{m\sim M,n\sim N}\chi^b(m)\chi^c(n)
      \Bigr|^2\\
    &\ll \frac{(MN)^2}q+\frac{1}{q}\sum_{{\chi^b,\chi^c\not=1}}
      \Bigl|\sum_{m,n}\chi^b(m)\chi^c(n)\Bigr|^2,	
  \end{align*}
  where the sums over~$\chi$ are over characters of~$\Ff_q^{\times}$.
  
  By the Cauchy--Schwarz inequality, we deduce
  \[
    \|\bfzeta\|^2_2\ll
    \frac{(MN)^2}q+\Bigl(\frac{1}{q-1}\sum_{\chi^b\not=1}|\sum_{m\sim
      M}\chi^b(m)|^4\Bigr)^{1/2}
    \Bigl(\frac{1}{q-1}\sum_{\chi^c\not=1}|\sum_{n\sim
      N}\chi^c(n)|^4\Bigr)^{1/2}.
  \]

  By positivity, note that
  \[
    \frac{1}{q-1}\sum_{\chi^b\not=1}\Bigl|\sum_{m\sim
      M}\chi^b(m)\Bigr|^4\leq \frac{1}{q-1}\sum_{\chi\not=1}\Bigl|
    \sum_{m\sim M}\chi(m)\Bigr|^4\ll M^2(\log q)^2,
  \]
  where the last inequality is a result of Ayyad, Cochrane and
  Zheng~\cite{Ayyad}*{Thm. 2}. Arguing similarly for the sum over
  characters with~$\chi^c\not=1$, we conclude that
  \[
    \|\bfzeta\|^2_2\ll\frac{(MN)^2}{q}+MN(\log q)^2,
  \]
  which finishes the proof.
\end{proof}

Let~$l\geq 2$ be an integer. We write
\[
  |\zeta_v|=|\zeta_v|^{(l-1)/l}|\zeta_v|^{1/l}
\]
and apply H\"older's inequality followed by the Cauchy--Schwarz inequality,
leading to the inequality
\begin{align}\nonumber
  \Bigl|\sumsum_{u,v\in \Fqt}\xi_u\zeta_v
  K(uv)\Bigr|
  &\leq \|\bfzeta\|_1^{(l-1)/l}
    \Bigl(\sum_{v\in\Fqt}|\zeta_v|\, \Bigl|\sum_{u\in\Fqt}\xi_u K(uv)\Bigr|^{l}
    \Bigr)^{1/l}\\
  &\leq \|\bfzeta\|_1^{1-1/l}\|\bfzeta\|_2^{1/l}
    \Bigl(\sum_{v\in\Fqt}
    \Bigl|\sum_{u\in\Fqt}\xi_u K(uv)
    \Bigr|^{2l} \Bigr)^{1/(2l)}.\label{Sxzbound}
\end{align}

Combined with Lemma~\ref{L2pierce}, this implies
\[
  \Bigl|\sumsum_{u,v\in \Fqt}\xi_u\zeta_v K(uv)\Bigr|\ll (\log
  q)^{1/l}(MN)^{1-1/(2l)}\Bigl(\sum_{v\in\Fqt}\Bigl|
  \sum_{u\in\Fqt}\xi_u K(uv)\Bigr|^{2l} \Bigr)^{1/(2l)}.
\]

We then expand the $(2l)$-th power to write
\begin{equation*}
  \sum_{v\in\Fqt}\Bigl|\sum_{u\in\Fqt}\xi_u K(uv)\Bigr|^{2l}
  =\sum_{\bfu\in
    (\Fqt)^{2l}}\xi_{\bfu}\sum_{v\in\Fqt}\prod_{i=1}^lK(u_iv)
  \ov{ K(u_{i+l}v)}
%  =\sum_{\bfu\in (\Fqt)^{2l}}\xi_{\bfu}S(F,\bfu;\Fq),
 % \label{afterholdersecondfactor}
\end{equation*}
where we denote
\[
  \bfu=(u_1,\cdots,u_{2l}),\quad\quad
  \xi_\bfu=\prod_{i=1}^l\xi_{u_i}\ov{\xi_{u_{i+l}}}
\]
for~$\bfu\in (\Fqt)^{2l}$.

Since $\mcF$ is mixed of weights~$\leq 0$, we have the bound
\[
  \sum_{v\in\Fqt}\prod_{i=1}^lK(u_iv) \ov{ K(u_{i+l}v)}\ll q
\]
for all~$\bfu\in (\Ff_q^{\times})^{2l}$. Moreover, since $\mcF$ is
\lgoodp, and a middle-extension, it follows from Corollary~\ref{cor-sop}
that there exists an algebraic subvariety $\mcV$ of $\Aa^{2l}_{\Ff_q}$
of dimension $\leq l$ and degree bounded in terms of $l$ and the
complexity of~$\mcF$ such that the estimate
\[
  \sum_{v\in\Fqt}\prod_{i=1}^lK(u_iv)
  \ov{ K(u_{i+l}v)}\ll q^{1/2}
\]
holds for $\bfu\in (\Fqt)^{2l}\setminus \mcV(\Ff_q)$, where the implied
constant depends on~$l$ and the complexity of~$\mcF$.

We next observe that
\[
  \sum_{\bfu\in \mcV(\Fq)}|\xi_\bfu|= \sum_{\bfu\in \mcV(\Ff_q)}
  \Bigl|\prod_{i=1}^l\xi_{u_i}\ov{\xi_{u_{i+l}}}\Bigr| \leq
  \sum_{\bfu\in \mcV(\Ff_q)} \prod_{i=1}^{2l}\sum_{\substack{j\sim
      J\\j^a=u_i}}1 \leq a^{2l} \sum_{\substack{\uple{j}\in
      [J,2J]^{2l}\\f(\uple{j})\in\mcV(\Ff_q)}} 1,
\]
where we define the morphism $f\colon \Gg_m^{2l}\to \Gg_m^{2l}$ by
\[
  f(x_1,\ldots,x_{2l})=(x_1^a,\ldots,x_{2l}^a).
\]

Applying Lemma~\ref{lm-sz} to $f^{-1}(\mcV)$, which also has
dimension~$\leq l$, we deduce that
\[
  \sum_{\bfu\in \mcV(\Fq)}|\xi_\bfu|\ll J^{l},
\]
and hence
% As long as $(a,q)=1$ the same holds for $\mcV^{\Delta,a}$, the preimage
% of $\mcV^{\Delta}$ under the $a$-power map
% $(u'_i)_{i\leq 2l}\ra ({u'_i}^a)_{i\leq 2l}$. From this we conclude that
%  $$\sum_{\bfu\in \\mcV(\Fq)}|\xi_\bfu|\leq |\mcV^{\Delta,a}(\Fq)\cap[V,2V]^{2l}|\ll V^{l}$$
\begin{align*}
  \sum_{\bfu\in
    (\Fqt)^{2l}}\xi_{\bfu}\sum_{v\in\Fqt}\prod_{i=1}^lK(u_iv)\ov{K(u_{i+l}v)}
  & \ll q\sum_{\bfu\in \mcV(\Fq)} |\xi_{\bfu}|+\sum_{\bfu\not\in
    \mcV(\Fq)} |\xi_{\bfu}|\,\Bigl| \sum_{v\in\Fqt}\prod_{i=1}^lK(u_iv) \ov{
    K(u_{i+l}v)}\Bigr|
  \\
  &\ll J^lq+J^{2l}q^{1/2}.
\end{align*}

Combining the previous estimates, we obtain the inequality
\[
  T(\bfalpha,\bfbeta,\bfgamma)\ll (\log
  q)^{1/l}(MN)^{1-1/(2l)}(J^{2l}q^{1/2}+J^lq)^{1/(2l)}\ll
  q^{\eps}JMN\Bigl(\frac{q^{1/2}}{MN}+\frac{q}{J^lMN}\Bigr)^{1/(2l)},
\]
for any $\eps>0$, which concludes the proof of
Theorem~\ref{thmtriplesum}.
 
\section{Ubiquity of \good sheaves}\label{sec-galant}

This section provides examples of \good sheaves; this will illustrate
that this class is much more flexible than the restricted types of
sheaves allowed in the previous papers such as \cite{KMSAnn}
and~\cite{Pisa}, and that it is very easy to check that most trace
functions of interest fit in this framework. In particular,
Theorem~\ref{thmType12} should have many applications in the future, and
these applications will not require extensive knowledge of algebraic
geometry.\footnote{\ We will not discuss systematically the weights of
  the sheaves which appear; checking that they satisfy the additional
  weight conditions is usually much more straightforward, using the
  Riemann Hypothesis.}

There are two basic principles involved. The first one is that if the
geometric monodromy group is infinite, then the only other condition is
that the sheaf should be geometrically irreducible. For finite
monodromy, one can use many known examples of exact computations of the
geometric monodromy group to check non-solvability.
% is that Katz has computed the geometric monodromy group of a very wide
% variety of families of exponential sums. These give many examples of
% trace functions, and an ``experimental'' fact is that in many cases
% the (connected component of the identity of this) group is a simple
% algebraic group, which implies by definition that the corresponding
% sheaf is \goodp.
The second general principle is that the condition that the geometric
monodromy group of a sheaf be non-solvable is ``robust'', in the sense
that if a non-solvable sheaf is transformed in certain ways, then this
condition remains valid.  Thus, any such operation will transform a
\good sheaf into another one.

We begin by describing some of these transformations.  Throughout, we
denote by~$k$ a finite field and by~$\mcF$ a constructible $\ell$-adic
sheaf on~$\Aa^1$ over~$k$, for some prime~$\ell$ invertible in~$k$. We
denote by~$G$ the geometric monodromy group of~$\mcF$, and by~$G^0$
its connected component of the identity.

% (which a priori applied
% only to certain specific Kloosterman sheave).
% In this section we would like to stress the considerable flexibility
% of the \good criterion (cf. Definition \ref{defgoodsheafintro})
% present in Theorems \ref{thmType1basic}, \ref{thmType2basic} and
% \ref{thmType1intro} \ref{thmType2intro}, a feature which was not
% available in our previous works \cite{KMSAnn,Pisa} (which a priori
% applied only to certain specific Kloosterman sheaves).

% The work of Katz \cite{GKM,ESDE} provide us with plenty of examples of
% \good sheaves appearing in analytic number theory, notably Kloosterman
% sheaves and more generally certain hypergeometric sheaves; these are
% described in the Appendix \ref{sec:hyperKloos}. Once a \good sheaf
% $\mcF$ at disposal, many more can be obtained very easily.

\subsection{Geometric irreducibility}

For the sake of accessibility, we recall here the diophantine criterion
for geometric irreducibility of Katz. This shows that this specific
condition, which is required in the definition of non-solvable sheaf, is
very often accessible to ``hands-on'' computations.

\begin{theorem}[Katz]\label{th-geom-irred}
  Let~$q$ be a prime number. Let~$\mcF$ be an $\ell$-adic
  middle-extension sheaf, pure of weight~$0$, on~$\Aa^1_{\Ff_q}$. For
  $n\geq 1$, let $x\mapsto t_{\mcF}(x;n)$ denote the trace function
  of~$\mcF$ on an extension~$\Ff_{q^n}$ of degree~$n$ of~$\Ff_q$. Then
  $\mcF$ is geometrically irreducible if and only if the limit formula
  \[
    \lim_{n\to+\infty}\frac{1}{q^n}
    \sum_{x\in \Ff_{q^n}}|t_{\mcF}(x;n)|^2=1
  \]
  holds.
\end{theorem}

\subsection{Bountiful sheaves}\label{subsecbountiful}

In the paper~\cite{sumproducts}, Fouvry, Kowalski and Michel presented
in a concrete way the application of the Goursat--Kolchin--Ribet
criterion of Katz in the context of estimating sums of products of
trace functions of the form
\[
  \sum_{x\in \Ff_q} K(\gamma_1\cdot x)\cdots K(\gamma_r \cdot
  x)e\Bigl(\frac{hx}{q}\Bigr)
\]
for some trace function~$K$ and $\gamma_i\in \GL_2(\Ff_q)$ acting as
fractional linear transformations. To encapsulate this method, they
defined \emph{bountiful sheaves} (see~\cite[Def.\,1.2]{sumproducts}). In
particular, these sheaves have geometric monodromy group equal to
either~$\SL_r$ or~$\Sp_r$ for some integer~$r\geq 2$; since these are
simple algebraic groups,
% , and the definition also requires that~$\mcF$ is
% mixed of weights~$\leq 0$ and pure of weight~$0$ on a dense open set
it follows immediately that \emph{any bountiful sheaf is \lgoodp}.

% In fact, our method can be extended to cover the case of {\em products} of galant trace functions:  for instance given a \lgoodp\ sheaf, $\mcF$, with trace function $x\mapsto K(x)$ and $\gamma_1,\cdots,\gamma_r\in\GL_2(\Fq)$ fractional linear transformation such that for $1\leq i\not=j\leq r$, the product $\gamma_i.\gamma_j^{-1}$ does not belong to the subgroup $T$ of affine linear transformations discussed in Proposition~\ref{pr-good-gkrlemma} (that group is finite, cyclic of cardinality bounded only in terms of $c(\mcF)$), then for $c=1$, the bilinear sums $B(\bfalpha,N)$ and $B(\bfalpha,\bfbeta)$ attached to the product function 
% $$x\mapsto K(\gamma_1\cdot x)\cdots K(\gamma_r \cdot
%   x)$$
% satisfy the bounds \eqref{BtypeI} and \eqref{BtypeII} (although the underlying sheaf is not galant).

\subsection{Birationality}

Since the definition of the geometric monodromy group of a sheaf only
depends on the restriction of this sheaf to an arbitrary open dense
subset of~$\Aa^1_k$, it follows that for any open immersion
$j\colon U\to \Aa^1_k$ with~$U$ not empty, the sheaves $j_!j^*\mcF$ and
$j_*j^*\mcF$ are \good if and only if~$\mcF$ is \goodp.

\subsection{Twists}

If~$\mcF$ is \goodp, and if $\mcL$ is a constructible $\ell$-adic sheaf
of generic rank~$1$ on~$\Aa^1$ over~$k$, lisse on an open dense
subset~$U$, then the twisted sheaf $\mcL\otimes\mcF$, which has trace
function
\[
  x\mapsto t_{\mcL}(x;k)t_{\mcF}(x;k)
\]
is also \goodp.

% We obtain a commutative diagram
% \[
%   \begin{matrix}
%     H & \fleche{\rho} & N
%     \\
%     \downarrow && \downarrow
%     \\
%     H/(H\cap \ker(\chi))& \to & N'
%   \end{matrix}
% \]
% with surjective horizontal arrows. We then observe that, on the one
% hand, the image of $\chi\otimes\rho$ contains $N'$, and on the other
% hand, $N'$ is quasisimple (see, e.g.,~\cite[Lemma\,9.2]{isaacs}).

This means, for instance, that if we can apply our results to a
function~$K$ on~$\Ff_q$, then they are also applicable to functions like
\[
  K(x)\chi(g(x))e(f(x)/q)
\]
for any multiplicative character~$\chi$ modulo~$q$, and any rational
functions~$f$ and~$g$ modulo~$q$ which can be written as ratios of
polynomials with degree uniformly bounded as $q$ varies (the last
requirement ensuring that the complexity of the corresponding twisted
sheaf remains bounded).

\subsection{Forms of exponential sums}

By definition, being non-solvable is a condition which only depends on
the geometric monodromy group, and hence only depends on the base change
of~$\mcF$ to an algebraic closure~$\ov{k}$ of the base field~$k$. Using
standard terminology from algebraic geometry, two objects defined
over~$k$ which become isomorphic over~$\ov{k}$ are called \emph{forms}
of each other, and in the context of trace functions, one could speak of
\emph{forms} of exponential sums.

Here is an example to illustrate the fact that this notion can be far
from trivial. Let $r\geq 1$ be an integer and let
$\uple{a}=(a_1,\ldots, a_r)$ be a family of positive integers. For any
prime number~$q$ and any $v\in\Ff_q^{\times}$, we may define the
exponential sums
\begin{equation}\label{eq-ta}
  \widetilde{H}_{\uple{a}}(v;q)=\frac{1}{q^{(r-1)/2}}
  \sum_{\substack{x\in(\Ff_q^{\times})^r\\x_1^{a_1}\cdots x_r^{a_r}=v}}
  e\Bigl(\frac{x_1+\cdots+x_r}{q}\Bigr).
\end{equation}

We claim that this is a ``form'' of the hyper-Kloosterman sums with
characters defined as follows: we let
\[
  a=a_1+\cdots+a_r,
\]
and we let
\[
  \uple{\chi}=(\chi_1,\ldots,\chi_a)
\]
be an arbitrary ordering of the $\ell$-adic characters $\chi$ of $\Fqt$
such that $\chi^{a_i}=1$ for some~$i$, repeated with multiplicity (so
that there are indeed $a$ such characters). Then the corresponding
hyper-Kloosterman sums are defined by
\begin{equation}\label{eq-kla}
  \Kl_a(u,\uple{\chi};q)= \frac{1}{q^{(a-1)/2}}
  \sum_{\substack{y_1,\ldots,y_a\in\Ff_q^{\times}\\y_1\cdots y_a=u}}
  \prod_{j=1}^a\chi_j(y_j)e\Bigl(\frac{y_1+\cdots+y_a}{q}\Bigr)
\end{equation}
for $u\in\Fqt$.

Precisely, the claim is that the natural hypergeometric complexes
over~$\Fq$ (in the sense of Katz) with trace functions~(\ref{eq-ta})
and~(\ref{eq-kla})\footnote{\ Actually, a multiplicative translate of
  the latter.} are geometrically isomorphic. This is a non-trivial fact,
which ultimately depends on a version of the Hasse--Davenport relations
(as in~\cite[p.\,84]{GKM}). One could then deduce a version of
Theorem~\ref{thmType12} for the sums $\widetilde{H}_{\uple{a}}$
from the results of~\cite{Pisa}, which apply to hyper-Kloosterman sums.
However, we will not elaborate on this, since we will see in
Section~\ref{sec-hypergeometric} that ``most'' hypergeometric sheaves
are \goodp.

\subsection{Change of variable}\label{sec-pullback}

Let~$j\colon \Aa^1_k\to\Pp^1_k$ denote the open immersion.  It is
essentially obvious that for any
automorphism~$\varphi$ of~$\Pp^1_k$, the pullback sheaves
$\varphi^*(j_*\mcF)$ and $\varphi^*(j_!\mcF)$ are \good if and only
if~$\mcF$ itself is \goodp.

Concretely, this means that if Theorem~\ref{thmType2simple} applies to
a trace function~$K$ modulo a prime~$q$, then it also applies to the
function
\[
  x\mapsto K\Bigl(\frac{ax+b}{cx+d}\Bigr),
\]
for any $\begin{pmatrix}a&b\\c&d
\end{pmatrix}$ in~$\GL_2(\Ff_q)$, with the convention that
$K(\infty)=0$. This is already quite significant: even the case of a
translation applied to a Kloosterman sheaf, i.e., the case of the
trace function
\[
  a\mapsto \frac{1}{q^{(r-1)/2}}
  \sum_{\substack{x_1,\ldots,x_r\in\Ff_q\\x_1\cdots x_r=a+h}}
  e\Bigl(\frac{x_1+\cdots+x_r}{q}\Bigr),\quad\quad h\in\Ff_q,
\]
was not previously known for any (fixed) $h\not=0$.

More generally, if $\mcF$ is {\em Lie-irreducible}, i.e. the connected
component $G^0$ of the geometric monodromy group acts irreducibly (which
implies that $G^0$ is infinite) then $\mcF$ is certainly \goodp and for
any non-constant morphism $f\colon \Pp^1_{k}\to \Pp^1_k$, the connected
components of the geometric monodromy group of the
sheaves~$j^*f^*j_*\mcF$ or $j^*f^*j_!\mcF$ are equal to $G^0$ and act
irreducibly and therefore these sheaves are \goodp.

This corresponds concretely to a change of variable of the form
\[
  x\mapsto K(f(x)),\quad\quad K(\infty)=0,
\]
for an arbitrary non-constant $f\in \Ff_q(X)$.

If~$\mcF$ is \good but has finite geometric monodromy group $G$ (the second case of Definition~\ref{defgoodsheafintro}), then a restricted version of this principle applies: suppose that the perfect radical $\bfP=\bfP(G)$ of $G$ acts irreducibly (in particular is not trivial) and that $f\colon \Pp^1_k\to \Pp^1_k$ defines 
a \emph{solvable} cover, then the perfect radical of the monodromy group of $j^*f^*j_*\mcF$ or~$j^*f^*j_!\mcF$ is equal to $\bfP$ so these sheaves are also \good.

In particular, if $\mcF$ is \good with $\bfP$ acting irreducibly, then the sheaf
$[x\mapsto x^c]^*\mcF$ is \good if~$c$ is a non-zero integer. For
trace functions, this corresponds to $x\mapsto K(x^c)$.

\subsection{Tannakian operations}

Let~$j\colon U\to \Aa^1_k$ be the open immersion of some open dense
subset of~$\Aa^1_k$ on which~$\mcF$ is lisse.  Let~$G^a$ denote the
\emph{arithmetic} fundamental group of~$j^*\mcF$. Let
\[
  \rho\colon G^a\to \GL(V)
\]
be a finite-dimensional irreducible (continuous)
representation of~$G^a$ on a $\bQl$-vector space~$V$ of dimension $\geq 2$. We can then form
the sheaf~$\rho(j^*\mcF)$ corresponding to the representation obtained
by composing with~$\rho$ the representation associated to~$j^*\mcF$.

Assume that~$\mcF$ is \goodp. If~$G$ is infinite and the restriction
of~$\rho$ to the connected component $G^0$ is irreducible, or if~$G$ is finite and the restriction of~$\rho$ to the perfect radical~$\bfP$ is irreducible, then the sheaves
$j_*\rho(j^*\mcF)$ and $j_!\rho(j^*\mcF)$ are \goodp:  in the infinite case one has $\rho(G)^0=\rho(G^0)$ and in the finite case $\bfP(\rho(G))=\rho(\bfP(G))$ and either acts irreducibly (in particular are not trivial). 

A concrete example here is the following. Consider the Kloosterman
sheaf~$\mcF$ modulo~$q$ with trace function
\[
  a\mapsto
  \Kl_2(a;q)=
  \frac{1}{\sqrt{q}}\sum_{x\in\Ff_q^{\times}}e\Bigl(\frac{ax+\ov{x}}{q}\Bigr),
\]
for $a\in\Fqt$. Write $\Kl_2(a;q)=2\cos(\theta(a;q))$ for some unique
$\theta(a;q)\in [0,\pi]$. Then, because the arithmetic and geometric monodromy groups are both equal to $\SL_2$, we can apply
any non-trivial irreducible representation of~$\SL_2$. These
representations are the symmetric powers of the standard representation,
and this means that our results apply, for any integer~$d\geq 1$, to the
function~$K^{(d)}$ defined by~$K^{(d)}(0)=0$ and
\[
  K^{(d)}(a)=\frac{\sin((d+1)\theta(a;q))}{\sin(\theta(a;q))}
\]
for~$a\in\Ff_q^{\times}$.  See 
\cite{MiInv,FMAnnals,XiInv,XiIMRN} for situations where the problem of bounding this kind of bilinear sums might be useful.

\subsection{Hypergeometric sheaves}\label{sec-hypergeometric}

In this section and the next, we discuss one of the most important class
of examples of trace functions in the context of applications to
analytic number theory, especially automorphic forms and
$L$-functions. These are hypergeometric sums and their associated
sheaves, as defined by Katz (see~\cite{GKM} and~\cite{ESDE}). It turns
out, as we will explain, that many (in some sense, most) such sheaves
are \goodp.  This can be derived from the basic facts established
in~\cite{ESDE}, but more recent work of Guralnick, Katz, Rojas--Le\'on
and Tiep give more definitive statements, and we will quote from them.

Let~$k$ be a finite field and $\ell$ a prime invertible
in~$k$. Let~$\psi$ be a non-trivial additive $\ell$-adic character
of~$k$.  Let~$r\geq 0$ and $t\geq 0$ be integers which are not both
$0$. The basic data is given by two multisets
\[
  \bfchi=\{\chi_1,\ldots,\chi_r\},\quad\quad
  \bfrho=\{\rho_1,\ldots,\rho_t\}
\]
of $\ell$-adic characters of $\kt$, which are assumed to be
\emph{disjoint} (no character has multiplicity $\geq 1$ in both of
them).

\begin{remark}
  In other words, taking $\bfchi$ for example, this is the data for each
  character~$\eta$ of $\kt$ of a non-negative integers
  $n_{\eta}=n_{\eta}(\uple{\chi})$, representing its multiplicity
  in~$\uple{\chi}$, distinct elements of the tuple being exactly those
  characters with $n_{\eta}\geq 1$, with the condition that
  \[
    \sum_{\chi}n_\chi=n,
  \]
  with obvious notation. Another interpretation would be an
  integer-valued non-negative measure on the set of $\ell$-adic
  characters of $\kt$.
\end{remark}

The hypergeometric sums associated to the pair $(\bfchi,\bfrho)$ are
defined for $u\in\kt$ by
\[
  \Hyp(u;\bfchi,\bfrho;\psi)=\frac{1}{|k|^{(r+t-1)/2}}
  \sumdsum_{\substack{x_1,\ldots,x_r,y_1,\ldots, y_t\in k\\
      \frac{x_1\cdots x_r}{y_1\cdots y_t}=u}}
  \psi\Bigl(\sum_{i=1}^r{x_i}-\sum_{j=1}^ty_j\Bigr)
  \prod_{i=1}^r\chi_i(x_i)\prod_{j=1}^t\ov{\rho_j(y_j)}.
\]

Katz constructed in~\cite[Th\,8.4.2]{ESDE} a middle-extension sheaf
$\HYP(\bfchi,\bfrho;\psi)$ on~$\Gg_{m,k}$ (which is denoted
$\HYP_1(!,\psi;\bfchi,\bfrho)$ in loc. cit.) with trace function given
by these hypergeometric sums.  These occur frequently in analytic number
theory, and we are interested in determining which of these are
\goodp.

We first note that the geometric irreducibility is always valid here.

\begin{lemma}
  The middle-extension sheaf $\HYP(\bfchi,\bfrho;\psi)$ on~$\Gg_{m,k}$
  is geometrically irreducible.
\end{lemma}

\begin{proof}
  By~\cite[Th.\,8.4.2\,(1)]{ESDE}, this follows from our assumption
  that the tuples $\bfchi$ and $\bfrho$ are disjoint.
\end{proof}

\begin{remark}
  If $t=0$, then $\bfrho=\emptyset$ and $\HYP(\bfchi,\emptyset;\psi)$ is
  a Kloosterman sheaf, also denoted $\KL(\bfchi;\psi)$. The trace
  function is then the generalized Kloosterman sums
  \[
    \frac{1}{|k|^{(r-1)/2}} \sumdsum_{\substack{x_1,\ldots,x_r\in k
        \\x_1\cdots x_r=u}} \psi(x_1+\cdots+x_r)
    \prod_{i=1}^r\chi_i(x_i).
  \]
\end{remark}
  
% \emph{From now on} we assume that $\bfchi$ and $\bfrho$ are not empty.
Since exchanging $\bfchi$ and $\bfrho$ amounts to applying the
inversion map $x\mapsto 1/x$ and applying complex conjugation to the
multisets (see~\cite[8.2.4]{ESDE}), we may also assume that $r\geq t$
(this relies on the fact that being \good is invariant under such a
transformation, as explained in Section~\ref{sec-pullback}). We then
denote
\[
  n=r=\max(r,t)\geq 1,
\]
which is the rank of the hypergeometric sheaf.

We recall the following definitions:
\begin{itemize}
\item An $n$-element multiset\footnote{\ Meaning that the sum of the
    multiplicities of all characters is~$n$.}  $\uple{\chi}$ is called
  \emph{Kummer-induced} if there exists a divisor $d>1$ of $n$ such
  that for some (or equivalently for any) character $\eta$ of order
  $d$ of~$\kt$, the equality $\eta\cdot\bfchi=\bfchi$ holds, with
  obvious notation for multiplication by a fixed character. We then
  also say that $\uple{\chi}$ is $d$-Kummer-induced. By convention,
  $\bfchi=\emptyset$ is $d$-Kummer-induced for every $d>1$.

\item The pair $(\bfchi,\bfrho)$ is \emph{Kummer-induced}
  (see~\cite[8.9.3]{ESDE}) if there exists a common divisor $d>1$ of
  $r$ and $t$ such that $\bfchi$ and $\bfrho$ are both $d$-Kummer
  induced in the sense of the previous item.

\item The pair $(\bfchi,\bfrho)$ is \emph{Belyi-induced}
  (see~\cite[\S\,8.10.1]{ESDE}) if $r=t$ and if there exists a
  partition $n=d+e$ of~$n$ in positive integers, and characters
  $\alpha$, $\beta$ of $\kt$ such that
  \begin{itemize}
  \item $\beta\not=1$;
  \item $\bfchi$ is the union of the multiset of characters $\eta$
    such that $\eta^d=\alpha$ and the multiset of characters $\eta$
    such that $\eta^e=\beta$ (with multiplicity);
  \item $\bfrho$ is the multiset of all characters $\eta$ such that
    $\eta^r=\alpha\beta$.
  \end{itemize}

  Note that if $t=0$, then $(\bfchi,\bfrho)$ cannot be Belyi-induced.

\item The pair $(\bfchi,\bfrho)$ is \emph{primitive} if it is neither
  Kummer-induced nor Belyi-induced.
\end{itemize}

We first consider the case when $r\not=t$.

\begin{theorem}\label{galanthypergeneric}
  Suppose that $r>t$ and $r\geq 2$. Let $j\colon \Gg_m\to \Aa^1$ denote
  the open immersion. Let
  \[
    \mcF=j_*\HYP(\bfchi,\bfrho;\psi)\quad\text{ or }\quad
    \mcF=j_!\HYP(\bfchi,\bfrho;\psi).
  \]

  Suppose that~$(\bfchi,\bfrho)$ is primitive.  Suppose that
  $q>2(r-t)+1$ and $q\not=5$. %%that $r\notin\{4,8,9\}$.

  Then the geometric monodromy group of~$\mcF$, which is the same as
  that of $\HYP(\bfchi,\bfrho;\psi)$, is infinite and $\mcF$ is
  \goodp.
\end{theorem}

\begin{proof}
  This follows from the recent work of Katz and
  Tiep~\cite{katz-tiep}. More precisely, let~$G$ denote the geometric
  monodromy group of $\HYP(\bfchi,\bfrho;\psi)$.

  First, the primitivity assumption on $(\bfchi,\bfrho)$ implies that
  the representation of~$G$ corresponding to $\HYP(\bfchi,\bfrho;\psi)$
  is primitive: this follows from~\cite[Lem.\,11\&12]{KatzDMJ} if $t=0$,
  and from~\cite[Prop.\,1.2]{KRLT} in general.
  
  Combined with the condition~$q>2(r-t)+1$ and~$q\not=5$, this implies
  by~\cite[Th.\,2.4.4]{katz-tiep} that~$G$ is infinite, since
  alternatives (b) and (c) of loc. cit. are then excluded.
  % Then by~\cite[Th.\,5.2.9]{katz-tiep}, the group $G$ satisfies
  % condition \textbf{(S+)}; by~\cite[Lemma\,1.1.3,\,(1)]{katz-tiep}
  % (which goes back to work of Guralnick and Tiep), this implies that
  % the connected component of~$G$ is a simple linear algebraic group
  % acting irreducibly.
\end{proof}

The results of Katz and Tiep are in fact much more precise, both
determining exactly the geometric monodromy groups in many cases, and
handling small characteristic situations. However, we are essentially
interested in applications where $q$ will be large (tending to
infinity), so we did not attempt to summarize all known results.

In the $r=t$ case, the monodromy group can be finite in many
interesting cases with~$q$ arbitrarily large. We discuss this case in
the next section.

% Under the assumption
% that it is infinite, the work of Katz provides the following:

% \begin{theorem}
%   Suppose that $r=t\geq 2$ and $q>r$.  Assume also that
%   $(\bfchi,\bfrho)$ is primitive.

%   Let $j\colon \Gg_m\to \Aa^1$ denote
%   the open immersion. Let
%   \[
%     \mcF=j_*\HYP(\bfchi,\bfrho;\psi)\quad\text{ or }\quad
%     \mcF=j_!\HYP(\bfchi,\bfrho;\psi).
%   \]

%   If $q$ is large enough compared with~$r$ and $G$ is infinite, then
%   $\mcF$ is \good unless $r=4$ and there exists a character $\eta$ such
%   that the multisets $\eta\cdot \bfchi$ and $\eta\cdot \bfrho$ are both
%   invariant under the inversion map $\xi\mapsto \xi^{-1}$ and
%   \[
%     \prod_{i=1}^4\chi_i\times \prod_{j=1}^4\rho_j^{-1}=\chi_{1/2}
%   \]
%   where $\chi_{1/2}$ is the unique character of exact order $2$. In that
%   case we have $G^0=\SO_4$ acting on $\mathrm{Std}_4$ and the geometric
%   monodromy group $G_\eta$ of the twisted sheaf $\mcF\otimes \mcL_\eta$
%   satisfies
%   \[
%     G_\eta=\begin{cases}\SO_4&\text{ if }\prod_{i=1}^4\eta.\chi_i=1,\\
%       \Ort_4&\text{ if }\prod_{i=1}^4\eta.\chi_i=\chi_{1/2}.
%     \end{cases}
%   \]
% \end{theorem}

\subsection{Hypergeometric sheaves with finite monodromy}

To be effective, the previous result requires a criterion to determine
when $G$ is finite.  A fundamental result of Katz shows that this is the
case if and only if the corresponding hypergeometric differential
equation has finite differential Galois group; one can then apply the
classification of such equations by Beukers and Heckman to determine
which of these have finite geometric monodromy groups, and when this is
the case, one can determine which of these are still \goodp.

More precisely, let~$\gamma$ be a generator of $\kt$. Using~$\gamma$, we
can identify the group of characters of $\Fqt$ with the subgroup
$\frac{1}{|k|-1}\Zz/\Zz\subset \Qq/\Zz$, which we identify further
set-theoretically with $\Qq\cap [0,1\mathclose[$. In particular, the
data of $\bfchi$ (resp. $\bfrho$) corresponds to a multiset
$\uple{x}=\{x_1,\ldots,x_r\}$ in $\Qq/\Zz$ (resp. a multiset
$\uple{y}=\{y_1,\ldots,y_t\}$ in $\Qq/\Zz$). Viewing these as multisets in
$[0,1\mathclose[$, we then have the associated hypergeometric
differential equation $D_{\uple{x},\uple{y}}f=0$, with
\[
  D_{\uple{x},\uple{y}}=\prod_{i=1}^r(\partial
  -x_i)-x\prod_{j=1}^t(\partial-y_j),\quad\quad
  \partial=x\partial_x
\]
(see~\cite[\S\,3.1]{ESDE} and~\cite[\S\,2]{BH}).
Katz~\cite[Th.\,8.17.12,\,Cor.\,8.17.15]{ESDE} proves that the
geometric monodromy group of $\HYP(\bfchi,\bfrho;\psi)$ is finite if
and only if the differential Galois group of this differential
equation is finite, and that these two groups are then isomorphic. 

The criterion of Beukers and Heckman~\cite[Th. 4.8]{BH} is a necessary
and sufficient condition for finiteness of the monodromy; this criterion
requires that $r=t$, that all the multiplicities in $\bfx$ and $\bfy$ be
equal to $1$ and that certain multiples of $\bfx$ and $\bfy$ interlace
on the unit interval.

In addition, Beukers and Heckman~\cite[Th.\,7.1]{BH} have given the
list of all primitive hypergeometric differential equations with
$r\geq 3$, the case $r=2$ being a classical result of Schwarz (see the
notes of Matsuda~\cite{matsuda} for a detailed modern account). The
corresponding finite groups are all finite complex reflection groups,
which have been classified by Shephard and Todd. It is elementary for
each of them to determine (with the help, e.g., of
\textsc{Magma}~\cite{magma}) if they are solvable or not. Using $ST_k$
for $4\leq k\leq 37$ to refer to the groups as classified classically
by Shephard--Todd, the answer is that $ST_k$ is non-solvable
\emph{unless}
\[
  4\leq k\leq 15,\quad\text{ or }\quad k\in \{25,26,28\}.
\]

\subsection{Other  examples with finite monodromy}\label{sec-finite}

As we mentioned in the introduction, one new feature of the present paper,
relying on the definition of \good sheaves, is that we can handle some
sheaves with finite monodromy group. An interesting family of examples
is given by sheaves of the form
\[
  \mcF=f_*\bQl/\bQl,
\]
where $f\in k[X]$ is a non-constant polynomial, viewed as a morphism
$f\colon \Aa^1_k\to \Aa^1_k$. This sheaf has trace function
\[
  K(x)=\sum_{\substack{y\in k\\f(y)=x}}1-1.
\]

The computation of the geometric monodromy group of such sheaves is a
classical question, often phrased in the context of the Galois group of
the polynomial equation
\[
  f(Y)-X=0
\]
in $k(X)[Y]$. A particularly simple case arises when $f$ is a ``supermorse
function'' (in the terminology of~\cite[(7.10.2.2)]{ESDE}), which means
that $\deg(f)$ is strictly less than the characteristic of~$k$, that
zeros of $f'$ are simple, and that $f$ separates the zeros of~$f'$. In
this case, it is known (see,
e.g.,~\cite[proof\,of\,Lemma\,7.10.2.3]{ESDE}) that the geometric
monodromy group of~$\mcF$ is isomorphic to the symmetric group
$S_{\deg(f)}$ (acting by the ``standard'' irreducible representation
of dimension $\deg(f)-1$). If $\deg(f)\geq 5$, the symmetric group $S_{\deg(f)}$
is \goodp, and hence we have here a large variety of examples where our
result applies.  To give a very concrete example, for an integer
$d\geq 6$, we can take
\[
  f=X^d-adX
\]
for some $a\in\Fqt$, where $q>d$ (see~\cite[Th.\,7.10.5]{ESDE}). The
corresponding bilinear forms (for $b=c=1$ for simplicity) are given by
the formula
\[
  \sum_{m\in \mcM}\sum_{n\sim N}\alpha_m\beta_n \Bigl(\sum_{\substack{y\in
      \Ff_q\\ y^d-ady=mn}}1-1\Bigr).
\]

It would be interesting to see concrete arithmetic applications of such
bilinear forms.

\subsection{Further examples}

We include one further class of \good sheaves, with infinite monodromy
group, taken again from the work of Katz~\cite[Ch.\,7]{ESDE}, and
chosen because it has no obvious relation to hypergeometric sheaves.
We consider again $f\in k[X]$, but instead of the sheaf of the
previous section, we consider its (unitarily normalized) Fourier
transform~$\mcF$, with respect to some non-trivial additive
character~$\psi$ of~$k$. The trace function of $\mcF$ is given by is
\[
  K(x)=\frac{1}{\sqrt{|k|}}\sum_{y\in k}\psi(xf(y))
\]
for $x\in\kt$.

Since the Fourier transform preserves geometric irreducibility, the
sheaf~$\mcF$ is geometrically
irreducible.

Assume that~$f$ is a supermorse polynomial.
% , and furthermore that the
% set~$C\subset \ov{k}$ of critical values of~$f$ (the set of values
% $f(y)$ for $y\in\ov{k}$ a root of~$f'$) is a \emph{Sidon set}, i.e.,
% the equation
% \[
%   s_1+s_2=s_3+s_4
% \]
% with $s_i\in C$ has only solutions with $s_1\in\{s_3,s_4\}$. 
Then Katz proved for if the characteristic of~$k$ is~$>2\deg(f)-1$,
the sheaf~$\mcF$ is Lie-irreducible, and in particular has infinite
geometric monodromy group (combine~\cite[Lemma\,7.10.2.3]{ESDE}
with~\cite[Th.\,7.9.4]{ESDE}).
% that, provided the characteristic of~$k$ is $>2\deg(f)-1$, the
% Fourier transform sheaf has geometric monodromy group with connected
% component equal to $\SL_{\deg(f)-1}$.
Thus, this sheaf is non-solvable under this condition.
% as soon as $\deg(f)\geq 3$.

Here also, it would be interesting to see applications of the
corresponding bilinear forms, such as
\[
  \sum_{m\in \mcM}\sum_{n\sim N}\alpha_m\beta_n \sum_{y\in
    \Ff_q}e\Bigl(\frac{mn(y^d-ady)}{q}\Bigr).
\]

\begin{remark}
  The long list of computations of monodromy groups
  in~\cite[Ch.\,7]{ESDE} provides a wide variety of additional examples
  of sheaves which are known to be \goodp.
\end{remark}

\subsection{The rank one case}\label{rank1}

The very definition of a \good sheaf implies that its rank is at least
$2$. Our general method does allows us to handle trace functions
of rank $1$ sheaves, such as
\[
  K(x)=\chi(f(x))\psi(g(x)),
\]
where~$\chi$ (resp. $\psi$) is a non-trivial multiplicative
(resp. additive) character of~$\Ff_q$, and~$f$, $g$ are rational
functions.
% The -trivial additive for $\chi$ and $\psi$ multiplicative and additive characters,
% $f(X),g(X)\in\Fq(X)$ and $\mcF=f^*\mcL_\chi\otimes g^*\mcL_\psi$.

Studying bilinear forms with kernel $K(m^bn^c)$ reduces (with our
approach) to bounding the two families of sums discussed in
Proposition~\ref{pr-one-variable}. These are one-variable sums of trace
functions of rank one sheaves, so cancellation amounts to deciding for
which value of the parameters $(\bfr,\bfs)$ or $(\bfr,\bfs_1,\bfs_2)$
the corresponding sheaf is geometrically trivial or not.

For instance, in the case above, the first of the relevant exponential
sums is
\[
  \sum_{v\in k}\chi\Bigl(\prod_{j=1}^m\frac{f(s_j(v+r_j)^c)}
  {f(s_{j+m}(v+r_{j+m})^c)};k\Bigr)
  \psi\Bigl(\sum_{j=1}^m{(g(s_j(v+r_j)^c)-g(s_{j+m}(v+r_{j+m})^c))};k\Bigr),
\]
for finite extensions~$k$ of $\Ff_q$, and what is required is to
determine for which $(\bfr,\bfs)$ or $(\bfr,\bfs_1,\bfs_2)$, the
rational fractions derived from $f(X)$ and $g(X)$ are constant or not,
and especially to show that they most often are not. (Here we denote
$\chi(x;k)=\chi(N_{k/\Ff_q}(x))$ and
$\psi(x;k)=\psi(\Tr_{k/\Ff_q}(x))$.)

An example where the argument will go through is when $g$ has at
least one pole and $c\geq 1$, since the cancellation of the pole will only
be achieved by ``combinatorial'' restrictions on the parameters, from
which the desired estimates on the size of the diagonal sets will
follow. An example is $g=1/X$, where the condition is that
\[
  \sum_{j=1}^m{\Bigl(\frac{1}{s_j(X+r_j)^c}-
    \frac{1}{s_{j+m}(X+r_{j+m})^c}\Bigr)}
\]
should be constant, which in turn requires that the $(r_{j+m})$ be a
permutation of the $(r_{j})$, with the corresponding $s_j$ and $s_{j+m}$
also equal.

On the other hand, if~$f=1$, $c\geq 1$ and~$g$ is a polynomial, then
there will be many more diagonal configurations. For instance, with
$g=X^2$, $c=1$, the condition is whether the polynomial
\[
  \sum_{j=1}^m{(s_j^2(X+r_j)^2-s_{j+m}^2(X+r_{j+m})^2)}
\]
is constant or not, and this will hold whenever
\begin{gather*}
  \sum_{j=1}^ms_j^2=\sum_{j=m+1}^{2m}s_j^2\\
  \sum_{j=1}^mr_js_j^2=\sum_{j=m+1}^{2m}r_js_j^2,
\end{gather*}
which defines a subvariety of codimension at most~$2$, which will have
many more solutions than desired for our method to work.

We leave further discussion to readers interested in specific
applications.


\begin{thebibliography}{CC}

\bib{Ayyad}{article}{
   author={Ayyad, A.},
   author={Cochrane, T.},
   author={Zheng, Z.},
   title={The congruence $x_1x_2\equiv x_3x_4\pmod p$, the equation
   $x_1x_2=x_3x_4$, and mean values of character sums},
   journal={J. Number Theory},
   volume={59},
   date={1996},
   number={2},
   pages={398--413},
   issn={0022-314X},
   doi={10.1006/jnth.1996.0105},
}

\bib{BV}{article}{
   author={Berta, F.},
   author={zur Verth, S.},
   title={Non-vanishing of central values of
     $L$-functions with angular restrictions},
   note={preprint, 2025},
}

\bib{BH}{article}{
   author={Beukers, F.},
   author={Heckman, G.},
   title={Monodromy for the hypergeometric function $_nF_{n-1}$},
   journal={Invent. Math.},
   volume={95},
   date={1989},
   number={2},
   pages={325--354},
   doi={10.1007/BF01393900},
}

\bib{MAMS}{article}{
   author={Blomer, V.},
   author={Fouvry, É.},
   author={Kowalski, E.},
   author={Michel, Ph.},
   author={Mili\'cevi\'c, D.},
   author={Sawin, W.},
   title={The second moment theory of families of $L$-functions---the case
   of twisted Hecke $L$-functions},
   journal={Mem. Amer. Math. Soc.},
   volume={282},
   date={2023},
   number={1394},
   pages={v+148},
   issn={0065-9266},
   isbn={978-1-4704-5678-8; 978-1-4704-7350-1},
  % review={\MR{4539366}},
   doi={10.1090/memo/1394},
}

\bibitem{magma} W. Bosma, J. Cannon and C. Playoust: \textit{The Magma
    algebra system, I. The user language} J. Symbolic Comput. 24 (1997),
  235--265; also \url{http://magma.maths.usyd.edu.au/magma/}.


  \bib{collins}{article}{
    author={Collins, M. J},
    title={On {J}ordan's theorem for complex linear groups},
    journal={Journal of Group Theory},
  volume={10},
  year={2007},
  pages={411-423},
  }

\bib{cr}{book}{
   author={Curtis, C. W.},
   author={Reiner, I.},
   title={Representation theory of finite groups and associative algebras},
   note={Reprint of the 1962 original},
   publisher={AMS Chelsea Publishing, Providence, RI},
   date={2006},
   pages={xiv+689},
}
  
\bib{WeilII}{article}{
   author={Deligne, P.},
   title={La conjecture de Weil. II},
   language={French},
   journal={Inst. Hautes \'Etudes Sci. Publ. Math.},
   number={52},
   date={1980},
   pages={137--252},
}



\bib{FKM2}{article}{
   author={Fouvry, É.},
   author={Kowalski, E.},
   author={Michel, Ph.},
     TITLE = {Algebraic trace functions over the primes},
   JOURNAL = {Duke Math. J.},
    VOLUME = {163},
      YEAR = {2014},
    NUMBER = {9},
     PAGES = {1683--1736},
}
 
  
\bib{FKM1}{article}{
   author={Fouvry, É.},
   author={Kowalski, E.},
   author={Michel, Ph.},
     TITLE = {Algebraic twists of modular forms and {H}ecke orbits},
   JOURNAL = {Geom. Funct. Anal.},
    VOLUME = {25},
      YEAR = {2015},
    NUMBER = {2},
     PAGES = {580--657},
}

\bib{sumproducts}{article}{
   author={Fouvry, É.},
   author={Kowalski, E.},
   author={Michel, Ph.},
   title={A study in sums of products},
   journal={Philos. Trans. Roy. Soc. A},
   volume={373},
   date={2015},
   number={2040},
   pages={20140309, 26},
}

\bib{FKMAA}{article}{
   author={Fouvry, É.},
   author={Kowalski, E.},
   author={Michel, Ph.},
   title={Toroidal families and averages of $L$-functions, I},
   journal={Acta Arith.},
   volume={214},
   date={2024},
   pages={109--142},
}


\bib{FKMSmoment}{article}{
   author={Fouvry, É.},
   author={Kowalski, E.},
   author={Michel, Ph.},
   author={Sawin, W.},
   title={Toroidal families and averages of $L$-functions, II: cubic moments},
   note={Preprint},
    date={2025},
}

%\bib{FKMSmomentinduced}{article}{
%   author={Fouvry, É.},
%   author={Kowalski, E.},
%   author={Michel, Ph.},
%   author={Sawin, W.},
%   title={Toroidal families and averages of $L$-functions, III: induced cubic moments},
%   note={Preprint},
%   date={2025},
%}
  
\bib{FMAnn}{article}{
   author={Fouvry, É.},
   author={Michel, Ph.},
   title={Sur certaines sommes d'exponentielles sur les nombres premiers},
   journal={Ann. Sci. \'Ecole Norm. Sup. (4)},
   volume={31},
   date={1998},
   number={1},
   pages={93--130},
}

\bib{FMAnnals}{article}{
   author={Fouvry, \'E.},
   author={Michel, Ph.},
   title={Sur le changement de signe des sommes de Kloosterman},
   language={French, with English summary},
   journal={Ann. of Math. (2)},
   volume={165},
   date={2007},
   number={3},
   pages={675--715},
   issn={0003-486X},
   review={\MR{2335794}},
   doi={10.4007/annals.2007.165.675},
}
  
\bib{FI}{article}{
   author={Friedlander, J.},
   author={Iwaniec, H.},
   title={Incomplete Kloosterman sums and a divisor problem},
   note={With an appendix by Bryan J. Birch and Enrico Bombieri},
   journal={Ann. of Math. (2)},
   volume={121},
   date={1985},
   number={2},
   pages={319--350},
}




\bib{isaacs}{book}{
   author={Isaacs, I. M.},
   title={Finite group theory},
   series={Grad. Studies Math.},
   volume={92},
   publisher={AMS},
   date={2008},
}

\bib{KatzDMJ}{article}{
 author={Katz, N. M.},
 title={On the monodromy groups attached to certain families of exponential sums},
 journal={Duke Mathematical Journal},
 volume={54},
 pages={41--56},
 date={1987},
}


\bib{GKM}{book}{
   author={Katz, N. M.},
   title={Gauss sums, Kloosterman sums, and monodromy groups},
   series={Annals of Mathematics Studies},
   volume={116},
   publisher={Princeton University Press, Princeton, NJ},
   date={1988},
}

\bib{ESDE}{book}{
   author={Katz, N. M.},
   title={Exponential sums and differential equations},
   series={Annals of Mathematics Studies},
   volume={124},
   publisher={Princeton University Press},
   address={Princeton, NJ},
   date={1990},
 }


\bib{MMP}{book}{
author={Katz, N. M.},
title={Moments, Monodromy, and Perversity: a Diophantine perspective},
   series={Annals of Mathematics Studies},
   volume={159},
   publisher={Princeton University Press},
   address={Princeton, NJ},
   date={2006},
   }
 
% \bib{CEST}{book}{
%    author={Katz, N. M.},
%    title={Convolution and equidistribution},
%    series={Annals of Mathematics Studies},
%    volume={180},
%    publisher={Princeton University Press, Princeton, NJ},
%    date={2012},
% }


\bib{KRLT}{article}{
 author={Katz, N. M.},
 author={Rojas-Le{\'o}n, A.},
 author={Tiep, P. H.},
 title={A rigid local system with monodromy group the big Conway group {{\(2.\mathsf{Co}_1\)}} and two others with monodromy group the Suzuki group {{\(6.\mathsf{Suz}\)}}},
 journal={Transactions of the American Mathematical Society},
 volume={373},
 number={3},
 pages={2007--2044},
 date={2020},

}

\bib{katz-tiep}{book}{
  author={Katz, N. M.},
  author={Tiep, P. H.},
   title={Exponential sums, hypergeometric sheaves and monodromy groups},
   series={Annals of Mathematics Studies},
   volume={220},
   publisher={Princeton University Press},
   address={Princeton, NJ},
   date={2025},
}


\bib{KMSAnn}{article}{
   author={Kowalski, E.},
   author={Michel, Ph.},
   author={Sawin, W.},
   title={Bilinear forms with Kloosterman sums and applications},
   journal={Ann. of Math. (2)},
   volume={186},
   date={2017},
   number={2},
   pages={413--500},
 }

 \bib{Pisa}{article}{
   author={Kowalski, E.},
   author={Michel, Ph.},
   author={Sawin, W.},
   title={Stratification and averaging for exponential sums: bilinear forms
   with generalized Kloosterman sums},
   journal={Ann. Sc. Norm. Super. Pisa Cl. Sci. (5)},
   volume={21},
   date={2020},
   pages={1453--1530},
}




\bibitem{matsuda} M. Matsuda: \textit{Lectures on algebraic solutions
    of hypergeometric differential equations}, Kyoto Univ. Lectures
  Notes 15, Kinokuniya, 1985; \url{http://hdl.handle.net/2433/84920}.
  
\bib{MiInv}{article}{
   author={Michel, Ph.},
   title={Autour de la conjecture de Sato-Tate pour les sommes de
   Kloosterman, I},
   journal={Invent. Math.},
   volume={121},
   date={1995},
   number={1},
   pages={61--78},
}

% \bibitem{Nordentoft} A.B. Nordentoft: \textit{Wide moments of
%     $L$-functions II: Dirichlet $L$-functions}, Quarterly J. Math. 2022,
%   haac026, \url{https://doi.org/10.1093/qmath/haac026}

\bib{nunes}{article}{
   author={Nunes, R. M.},
   title={On the least squarefree number in an arithmetic progression},
   journal={Mathematika},
   volume={63},
   date={2017},
   number={2},
   pages={483--498},
}

\bib{nunesarx}{article}
{
   author={Nunes, R. M.},
   title={Squarefree integers in large arithmetic progressions},
   date={2016},
   eprint={arXiv:1602.00311},
   url={https://arxiv.org/abs/1602.00311},
   }

  
\bibitem{ribet} K. Ribet: \textit{On $\ell$-adic representations
    attached to modular forms}, Invent. math. 28 (1975), 245--275.
  
\bib{qst}{article}{
   author={Sawin, W.},
   author={Forey, A.},
   author={Fres\'an, J.},
   author={Kowalski, E.},
   title={Quantitative sheaf theory},
   journal={J. Amer. Math. Soc.},
   volume={36},
   date={2023},
   number={3},
   pages={653--726},
}

\bib{SerreBook}{book}{
   author={Serre, J.-P.},
   title={Finite groups: an introduction},
   edition={revised edition},
   note={With assistance in translation provided by Garving K. Luli and Pin
   Yu},
   publisher={International Press, Somerville, MA},
   date={2022},
   pages={xi+177},
   isbn={978-1-57146-410-1},
  % review={\MR{4574841}},
}

\bib{XiIMRN}{article}{
   author={Xi, P.},
   title={Sign changes of Kloosterman sums with almost prime moduli. II},
   journal={Int. Math. Res. Not. IMRN},
   date={2018},
   number={4},
   pages={1200--1227},
   issn={1073-7928},
   review={\MR{3801460}},
   doi={10.1093/imrn/rnw276},
}


\bib{XiInv}{article}{
   author={Xi, Ping},
   title={When Kloosterman sums meet Hecke eigenvalues},
   journal={Invent. Math.},
   volume={220},
   date={2020},
   number={1},
   pages={61--127},
   issn={0020-9910},
%   review={\MR{4071409}},
   doi={10.1007/s00222-019-00924-y},
}


\bib{Xi}{article}{
   author={Xi, P.},
   title={Bilinear forms with trace functions over arbitrary sets and
   applications to Sato-Tate},
   journal={Sci. China Math.},
   volume={66},
   date={2023},
   number={12},
   pages={2819--2834},
   issn={1674-7283},
   review={\MR{4670153}},
   doi={10.1007/s11425-022-2184-9},
}

  
\bib{XZ}{article}{
   author={Xi, P.},
   author={Zheng, J.},
   title={On the Brun--Titchmarsh theorem. II},
   journal={arXiv},
   url={https://arxiv.org/abs/2504.12692},
   date={2025},
%   number={10},
 %  pages={2881--2917},
}
  
\bib{Xu}{article}{
   author={Xu, J.},
   title={Stratification for multiplicative character sums},
   journal={Int. Math. Res. Not. IMRN},
   date={2020},
   number={10},
   pages={2881--2917},
}


\end{thebibliography}
\end{document}